\documentclass[11pt]{amsart}
\usepackage[margin=.5in]{geometry}
\usepackage{amsmath}
\usepackage{amsxtra}
\usepackage{amscd}
\usepackage{amsthm}
\usepackage{amssymb}
\usepackage[normalem]{ulem}
\usepackage{comment}
\usepackage{xcolor}
\usepackage{tikz}
\usetikzlibrary{matrix,arrows,decorations.pathmorphing,automata, matrix, positioning, calc, shapes.multipart}
\usepackage{tikz-cd}
\usepackage{graphicx}
\usepackage{hyperref}
\usepackage[shortlabels]{enumitem}
\usepackage{float}
\usepackage{genyoungtabtikz}
\usepackage{epigraph}
\usepackage{mathrsfs}
\usepackage{mathpazo}
\usepackage{xargs}
\usepackage[capitalize]{cleveref}
\usepackage[colorinlistoftodos,prependcaption,textsize=tiny]{todonotes}
\usepackage{mynewcommands}
\newcommandx{\com}[2][1=]{\todo[inline,linecolor=black,backgroundcolor=red!30,#1]{#2}}
\newcommand{\precdot}{\prec\mathrel{\mkern-5mu}\mathrel{\cdot}}
\usepackage{authblk}
\usepackage[foot]{amsaddr}

\Yboxdim{0.4cm}

\title{Duality and bicrystals on infinite binary matrices}
\author{Thomas Gerber}
\address[T.G.]{\'Ecole Polytechnique F\'ed\'erale de Lausanne, Route Cantonale, 1015 Lausanne, Suisse.}
\email{thomas.gerber@epfl.ch}
\thanks{T.G. is supported by an \textit{Ambizione} grant of the Swiss National Science Foundation.}
\author{C\'edric Lecouvey}
\address[C.L.]{Universit\'e de Tours, Parc de Grandmont, 37200 Tours, France.}
\email{cedric.lecouvey@lmpt.univ-tours.fr}

\begin{document}

\begin{abstract}
The set of finite binary matrices of a given size is known to carry a finite type $A$ bicrystal
structure. 
We first review this classical construction, explain how it yields a
short proof of the equality between Kostka polynomials and one-dimensional
sums together with a natural generalisation of the $2M-X$ Pitman transform.
Next, we show that, once the relevant formalism on families of infinite binary matrices is introduced, this is a particular case of a much more general phenomenon.
Each such family of matrices is proved to be endowed with Kac-Moody bicrystal and tricrystal structures defined from
the classical root systems. 
Moreover, we give an explicit decomposition of these multicrystals, reminiscent of 
the decomposition of characters yielding the Cauchy identities.
\end{abstract}

\maketitle

\tableofcontents

\sloppy

\section{Introduction}

Crystals are oriented graphs which can be interpreted as the
combinatorial skeletons of certain modules for complex Lie algebras and their infinite-dimensional analogues: the Kac-Moody algebras.
Crystal bases were introduced by Lusztig (for any finite root system)
and Kashiwara (for classical root systems) in 1990. 
The graph structure arises from 
the action of the so-called Kashiwara operators,
a certain renormalisation of the Chevalley operators.
In Kashiwara's approach, crystals are obtained
via \textquotedblleft crystallisation\textquotedblright\ at $q=0$ of
representations of the corresponding quantum group, a $q$-deformation of
the Kac-Moody algebra introduced by Jimbo.
Later, it was proved that crystals coincide with Littelmann's graphs defined by using his path model.
Since its introduction, crystal theory has revealed numerous fruitful
interactions with modern particle physics theory and integrable systems.
We refer the reader to \cite{BumpSchilling2017} and the references therein for a recent exposition.

\medskip

The present paper is concerned with generalisations of crystals where
multimodule structures are considered rather than just ordinary module
structures. This means that we consider analogues of crystals for complex vector
spaces endowed with commuting actions of several Lie algebras (or Kac-Moody algebras). 
It turns out that most of the combinatorial structures 
(crystals, Fock spaces, one-dimensional sums, Pitman transforms) 
that we shall consider in the sequel were defined to solve problems connected to theoritical physics. 
We expect similar interactions with the results we establish here.

\medskip

The prototypical example of an interesting bicrystal is obtained by starting from the $\mathfrak{gl}_{n}\times
\mathfrak{gl}_{\ell}$-module $\mathbb{C}^{n}\otimes\mathbb{C}^{\ell}$ and
considering the associated symmetric and antisymmetric $(\mathfrak{gl}_{n}\times\mathfrak{gl}_{\ell})$-modules 
$S(\mathbb{C}^{n}\otimes\mathbb{C}^{\ell})$ and $\Lambda(\mathbb{C}^{n}\otimes\mathbb{C}^{\ell})$.
Using two sets of indeterminates $\{x_{1},\ldots,x_{n}\}$ and
$\{y_{1},\ldots,y_{\ell}\}$, one can check that their characters are given by
the formulas
\[
\mathrm{char} \ S(\mathbb{C}^{n}\otimes\mathbb{C}^{\ell})=\prod_{\substack{1\leq i\leq n \\ 1\leq j\leq \ell}}\frac{1}{1-x_{i}y_{j}}
\mand
\mathrm{char} \ \Lambda(\mathbb{C}^{n}\otimes\mathbb{C}^{\ell})=\prod_{\substack{1\leq i\leq n \\ 1\leq j\leq \ell}}(1+x_{i}y_{j}).
\]
The classical Cauchy identities then gives the decomposition of each
character in terms of the Schur functions. 
More precisely, recall that
the irreducible
finite-dimensional $\mathfrak{gl}_{n}$-modules are parametrised by
partitions of length $n$ (that is, nonincreasing sequences $\lambda=(\lambda_{1},\ldots,\lambda_{n})$
of nonnegative integers) and the character of the irreducible module parametrised by $\lambda$ is the Schur symmetric polynomial $s_{\lambda}(x)$ in the
variables $x_{1},\ldots,x_{n}$. We then have
\begin{align*}
\mathrm{char}\ S(\mathbb{C}^{n}\otimes\mathbb{C}^{\ell})  & =\prod_{\substack{1\leq i\leq n \\ 1\leq j\leq \ell}}\frac{1}{1-x_{i}y_{j}}=
\sum_{
\substack{
\lambda\text{ partition of} \\
\text{length } \min(n,\ell)
}}
s_{\lambda}(x)s_{\lambda}(y)
\mand
\\
\mathrm{char}\ \Lambda(\mathbb{C}^{n}\otimes\mathbb{C}^{\ell})  & =\prod_{\substack{1\leq i\leq n \\ 1\leq j\leq \ell}}(1+x_{i}y_{j})=
\sum_{
\substack{
\lambda\text{ contained in the } \\ \text{rectangle } n\times\ell
}}
s_{\lambda}(x)s_{\lambda^{\mathrm{tr}}}(y)\nonumber
\end{align*}
where $\lambda^{\mathrm{tr}}$ is the transpose of the partition $\lambda$. 
Both identities, which can be seen as a combinatorial version of Howe duality,
admit an elegant combinatorial proof based on the
Robinson-Schensted-Knuth (RSK) correspondence, see \cite{Fulton1997}. 
The idea is first to observe that
\[
\prod_{\substack{1\leq i\leq n \\ 1\leq j\leq \ell}}\frac{1}{1-x_{i}y_{j}}=\sum_{(m_{i,j})\in \sN}\prod_{\substack{1\leq i\leq n \\ 1\leq j\leq \ell}}(x_{i}y_{j})^{m_{i,j}}%
\]
where $\sN$ is the set of $n\times\ell$ matrices with nonnegative integer entries. Next, the RSK correspondence
yields a bijection between $\sN$ and the set of pairs
of semistandard tableaux with the same partition shape. The identity then
follows from the fact that $s_{\lambda}$ is the generating function of the set
of semistandard tableaux with shape $\lambda$ for the evaluation map on
tableaux. Similarly, one has%
\[
\prod_{\substack{1\leq i\leq n \\ 1\leq j\leq \ell}}(1+x_{i}y_{j})=\sum_{(m_{i,j})\in\sM}\prod_{\substack{1\leq i\leq n \\ 1\leq j\leq \ell}}(x_{i}y_{j})^{m_{i,j}}%
\]
where $\sM$ is the set of binary $n\times\ell$ matrices, and
the second Cauchy identity comes from an adapted version of the RSK correspondence.
This construction admits numerous extensions, notably based on
generalisations of the previous Cauchy identities due to Littlewood,
and involving the characters of simple modules corresponding to the
orthogonal or symplectic Lie algebras. 
We refer the reader to
\cite{FuLas2009} and the references therein for a more detailed presentation,
and for a generalisation to the Demazure characters. 

\medskip

We now turn out to the main topic of this article. One can define two
commuting families of crystal operators (one for $\mathfrak{gl}_{n}$ and
another for $\mathfrak{gl}_{\ell}$) directly on the set of matrices
$\sN$ and $\sM$ yielding the structure of a
$(\mathfrak{gl}_{n}\times\mathfrak{gl}_{\ell})$-crystal, or type $(A_{n-1}\times A_{\ell-1})$-crystal.
This approach, explained in detail in \cite{VanLeeuwen2006}, see also \cite[Chapter 11]{Halacheva2016}, permits to overcome the RSK construction.
One gets the decomposition
\[
\sN=
\bigoplus_{\substack{
\lambda\text{ partition of} \\
\text{length } \min(n,\ell)
}}
B(\lambda)\otimes \dB(\lambda)
\mand
\sM=
\bigoplus_{
\substack{
\lambda\text{ contained in the } \\ \text{rectangle } n\times\ell
}}
B(\lambda)\otimes \dB(\lambda^{\mathrm{tr}})
\]
where $B(\lambda)$ (respectively $\dB(\lambda)$) is the crystal of the
$\mathfrak{gl}_{n}$-module (respectively $\mathfrak{gl}_{\ell}$-module) parametrised by $\lambda$. 
This immediately implies the Cauchy identities.
In the following sections, we shall present different
extensions of this construction of multicrystal structures on sets of binary
matrices (possibly infinite). 
They arise naturally from the notion of combinatorial Fock spaces $F(s)$, where $s$ is a integer.
The $A_{n-1}$ and $A_{\ell-1}$ crystal structures of $F(s)$ correspond to the $\mathfrak{gl}_{n}$ and $\mathfrak{gl}_{\ell}$ modules
\[
\bigoplus_{s_{1}+\ldots s_{\ell}=s}
\Lambda^{s_{1}}(\mathbb{C}^{n})\otimes\cdots\otimes\Lambda^{s_{\ell}%
}(\mathbb{C}^{n}),
\quad\text{ and }\quad
\bigoplus_{\dot{s}_{1}+\cdots\dot{s}_{n}=s}
\Lambda^{\dot{s}_{1}}(\mathbb{C}^{\ell})\otimes\cdots\otimes\Lambda^{\dot{s}_{n}}(\mathbb{C}^{\ell}).
\]
A similar phenomenon has been studied in the case where $\gl_n$ is replaced by a Lie superalgebra of type $A$ \cite{KwonPark2015}.
Note also that the notion of Fock space arises in different mathematical physics contexts.
For instance, the corner transfer matrix method, used in the study of the Yang-Baxter equations, gives rise to the combinatorics of weighted paths. 
These, in turn, provide a realisation of the crystal base of the quantum group of $\widehat{sl_n}$. This approach has been extensively studied by the Kyoto group in the late 1980's and early 1990's, see in particular \cite{JMMO1991}. 
More recently, the study of weighted paths has yielded (generalisations of) the Rogers-Ramanujan identities in enumerative combinatorics \cite{FodaWelsh2015}.
In another direction, in affine type $A$, there is a similar construction due to Uglov
\cite{Uglov1999} where the ordinary wedge products are replaced by their
thermodynamical limits. It plays a central role in the representation theory of
Cherednik algebras, see \cite{Shan2011}, and in the construction of some representations of the Virasoro algebra \cite{JMMO1991}. 
We get the structure of an 
$(\widehat{\mathfrak{sl}_{n}}\times \cH \times\widehat{\mathfrak{sl}_{\ell}})$-module 
where $ \cH $ is a Heisenberg algebra, and therefore a tricrystal structure on the affine Fock space.
This structure has been made completely explicit in \cite{GerberNorton2018}.

\medskip

A key tool in our approach is to exploit a combinatorial duality which permits to  easily switch between the different combinatorial
actions defined on $F(s)$. 
In finite type $A$, this coincides with the
transposition of binary matrices, and enables us to bypass the RSK correspondence.
This is particularly convenient because the insertion algorithm
(on which the RSK correspondence is based) in other types is more complicated and less well understood.
This point of view enables us to unify and extend the existing constructions in affine type, 
by considering a block transposition on appropriate infinite matrices.
Through the previous duality, each combinatorial object or Dynkin diagram automorphism for one structure admits a natural counterpart for the other one. 
For example, the cyclage operation introduced by Lascoux and Sch\"{u}tzenberger on the semistandard tableaux of type $A_{n-1}$ will
correspond to the promotion operator on the $A_{\ell-1}$-highest weight vertex in $F(s)$. 
This permits us to give a short proof of the equality between Kostka
polynomials and one-dimensional sums established by Nakayashiki and Yamada in
\cite{NY1995}. 
Note that one-dimensional sums appear in mathematical physics in the context of solvable lattice models and the corner transfer matrix method. 
Using the combinatorics of crystal bases and rigged configurations, one can show that they coincide with the fermionic formula via the Bethe Ansatz, thereby giving a proof of the $X=M$ conjecture in some cases, see \cite{Schilling2007}.

This observation also provides an algebraic interpretation of
the Pitman transform $M-2X$ introduced in \cite{Pit78} to obtain the law of a Brownian motion conditioned to stay positive. 
Historically, Brownian motions have strong connections with thermodynamics and molecular motions.
This is achieved in the spirit of \cite{BBO}, where the dual version of the Pitman transform 
is shown to map each Littelmann path on its associated highest path. 
It is also worth mentioning that the previous combinatorial constructions have interesting applications
in problems related to percolation models \cite{BDJ1999}. 
Further, by considering subsets of the combinatorial Fock spaces
$F(s)$ invariant under Dynkin diagram automorphisms of type $A$ (affine or
not), we get various bicrystal (or tricrystal) structures of classical types on
sets of binary matrices.
In a connected direction, based on the results of
\cite{Lec2009}, we also establish that it is also possible to define a
combinatorial Fock space with an $(X_{\infty}\times A_{\ell-1})$-structure (with
$X$ of type $B,C$ or $D$) on some infinite binary matrices similar to the
finite type $A$ construction.
We expect that this bicrystal to be related
to the charge statistics defined for type $C$ in \cite{Lec2005a}, once the
appropriate duality relating both crystal structures is discovered.
Note that Howe-type dualities and a bicrystal structure involving type $C$ constructions
have been recently studied in \cite{Heo-Kwon2020} and \cite{Lee2019} respectively.
It would be interesting to link these results with that of the present paper.

Finally, as already mentioned, 
the combinatorial Fock spaces that we shall study can be defined in terms of (infinite) binary matrices (this corresponds to the second Cauchy identity). 
This is indeed the natural context corresponding to the existing mathematical material (such as Kashiwara-Nakashima columns tableaux and Uglov's Fock space),
but it would be interesting to get analogous results for infinite matrices with nonnegative integer coefficients.

\medskip

In the present paper, we have chosen to illustrate our results by numerous examples. Quite often, they can be established by adapting proofs existing in the literature to the unified formalism that we propose. We then give precise references rather than complete proofs. The next sections are organised as follows. Section 2 is devoted to a re-exposition of the type $(A_{n-1}\times A_{\ell-1})$-crystal structure on $\sM$.
This serves as a basis for the various generalisations of the subsequent sections.
In particular, we quickly reach a simple proof that the two crystals commute in \Cref{crystals_commute}, 
recovering the results of \cite{VanLeeuwen2006} and \cite{Halacheva2016}.
Also, thanks to the connection between the cyclage and the promotion operator, we are able to give a
short proof of the relation between the charge and the energy function originally proved in \cite{NY1995}.
In Section 3, the classical $(A_{n-1}\times A_{\ell-1})$-crystal structure on $\sM$ 
is made compatible with an affine similar construction given in \cite{Gerber2016}.
The highest weight vertices for the different possible
(simple, double an triple) actions are described.
Also, a new combinatorial interpretation of the $(e,\boldsymbol{s})$-cores introduced in
\cite{JaconLecouvey2019} for describing the blocks of cyclotomic Hecke algebras is proposed.
Section 4 describes a type $X_{\infty}\times A_{\ell-1}$ analogue to the
previous bicrystal.
This means that we define a type  $A_{\ell-1}$-crystal structure on products of $X_{\infty}$-columns. This is done directly in terms of sliding (Jeu de Taquin) operations. 
Nevertheless, due to the lack of a straightforward duality, the results of Section 2 are needed to prove that
this indeed yields the desired $A_{\ell-1}$-crystal structure.
The results of Section 5 focus on the vertices of the
combinatorial type $A$ (of both finite and affine type) Fock spaces 
fixed under the action of Dynkin diagram automorphisms. 
By using results of Naito and Sagaki \cite{NaitoSagaki2004}, they are proved to have various bicrystal (or tricrystal) structures
of classical types. 
Finally in Section 6, we relate Pitman's $2M-X$ transform
on the line to its dual version (the $X-2M$ transform) and the promotion
operator on tensor products of type $A_{1}^{(1)}$ Kirillov-Reshetikhin crystals. 
This permits notably to show that iterations of this transform on any trajectory will eventually tend to the trivial one
(that is, with all steps equal to $1$). 
Using this time $A_{\ell}^{(1)}$-crystals, a highest dimensional generalisation of the $2M-X$ transform which
shares the same convergence behavior is defined. Its probabilistic properties will be studied elsewhere.

\section{Finite type $A$ duality}\label{duality}

In the rest of the paper, fix $n,\ell\in\Z_{\geq2}$.

\subsection{Products of type $A$ columns}\label{columns_finite}

Let $P$ be the weight lattice for the 
Lie algebra $\sl_n$, with basis $\left\{ \om_1,\ldots, \om_{n-1}\right\}$,
where $\om_1,\ldots,\om_{n-1}$ are the fundamental weights for $\sl_n$.
Each partition $\la=(\la_1,\ldots,\la_n)\in\Z_{\geq 0}^n$ shall be identified with the $A_{n-1}$-dominant weight $\la = \sum_{i=1}^{n-1}a_i\om_i$ where for any $i=1,\ldots,n-1$ the integer $a_i$ is the number of columns of height $i$ in the Young diagram of the partition $\la$. Observe that the contribution of the columns of height $n$ is thus equal to zero and we have a one-to-one correspondence between the type $A_{n-1}$-dominant weights and the partitions with at most $n-1$ parts. In what follows, it thus makes sense to use the symbol $\la$ as a partition with at most $n$ parts or a dominant weight of type $A_{n-1}$.

\begin{exa} \Yboxdim{7pt}
Let $n=3$ and $\la=3\om_1+\om_2$. Then the corresponding partition is $(4,1,0)=\yng(4,1)$.
\end{exa}

A \textit{column} of  type $A_{n-1}$ is a subset $c$ of $\{1,\ldots,n\}$ such that $|c|\leq n$, which we identify with
the semistandard Young tableau of shape $\om_{|c|}=(1,\ldots, 1)$ containing the elements of $c$.
\begin{exa}
The set $\{1,3,4\}=\scriptsize\young(1,3,4)$ is a column of type $A_{3}$.
\end{exa}

\begin{defi} Let $c_1,\ldots,c_\ell$ be columns of type $A_{n-1}$. The symbol $b=c_\ell\otimes\cdots\otimes c_1$ is called
\begin{enumerate}
\item  a \textit{tableau}
if the top-aligned juxtaposition $c_1\cdots c_\ell$ yields a semistandard Young tableau.
\item an \textit{antitableau}
if the bottom-aligned juxtaposition $c_1\cdots c_\ell$ yields a semistandard skew Young tableau.
\end{enumerate}
The \textit{shape} of a tableau (respectively of an antitableau) is the partition $(|c_1|,\ldots, |c_\ell|)^\trans$
(respectively $(|c_\ell|,\ldots, |c_1|)^\trans$).
\end{defi}

\begin{exa}Let $\ell=3$ and $n=2$ 
\begin{enumerate}
\item The product $\scriptsize\young(2)\otimes\young(1)\otimes\young(1,3)$
is a tableau which we identify with $\scriptsize\young(112,3)$.
\item The product $\scriptsize\young(1,3)\otimes\young(1,2)\otimes\young(2)$ is an antitableau
which we identify with $\scriptsize\gyoung(:~11,223)$.
\item The product $\scriptsize\young(1,3)\otimes\young(1,2)$ is both a tableau and an antitableau, 
which we identify in with $\scriptsize\young(11,23)$.
\item The products $\scriptsize\young(1,2)\otimes\young(1)\otimes\young(2)$
and $\scriptsize\young(1,2)\otimes\young(1)\otimes\young(1,2)$
are neither tableaux nor antitableaux.
\end{enumerate}
\end{exa}

For the next definition, define first the \textit{word} $\textrm{w}(b)$ of a product $b=c_\ell \otimes\cdots \otimes c_1$
to be the concatenation of the elements of $c_\ell$ (in increasing order),
then $c_{\ell-1}$, and so on.

\begin{defi}
The element $b$ is called \textit{Yamanouchi}\footnote{
Note that the usual convention is to use suffixes instead of prefixes. 
In fact, our definition coincides with the notion of \textit{lattice word} in the literature.
}
if every prefix of $\textrm{w}(b)$ contains 
at least as many letters $i$ as $i+1$, for all $i=1,\ldots, n-1$.
\end{defi}

\begin{exa}
Let $\ell=4$ and $n=3$, and take $b=\scriptsize\young(1)\otimes\young(2)\otimes\young(1,3)\otimes\young(2)$.
Then $\textrm{w}(b)=12132$, and the different prefixes are $1$, $12$, $121$, $1213$, $12132$ and we see that
$b$ is Yamanouchi.
\end{exa}

Clearly, for all $\la\in P_+$, there is a unique Yamanouchi tableau (respectively antitableau) of shape $\la$.
\begin{exa} The Yamanouchi tableau and antitableau of shape $(4,3,1)$ are respectively given by
$${\scriptsize \young(1111,222,3)}\quad \mand\quad  {\scriptsize\gyoung(:~:~:~1,:~112,1223)}.$$
\end{exa}

\subsection{Crystal structures}\label{crystal_finite}

From now on, fix $s\in\Z_{\geq0}$.
Further, for $\bs=(s_1,\ldots,s_\ell)\in R(n,\ell)=\left\{0,\ldots,n\right\}^\ell$
and  $\dbs=(\ds_\ell,\ldots,\ds_1)\in R(\ell,n)=\left\{0,\ldots,\ell\right\}^{n}$,
denote $|\bs|=\sum_{j=1}^\ell s_j$ and $|\dbs|=\sum_{i=1}^n \ds_i$.
For $p=n,\ell$ and for all $A\subseteq\Z^p$, write 
Finally, we denote
$$\sS(s) =\left\{ \la=(\la_1,\ldots,\la_n) \vdash s \mid  \la_1\leq\ell \right\}.$$

The elements introduced in the previous section appear as vertices of certain
tensor products of crystal graphs, which we start by recalling.
For all $j=1,\ldots,\ell$, the crystal of the irreducible highest weight $\mathfrak{sl}_n$-module of highest weight $\om_{s_j}$ (with the convention $\om_{0}=\om_{n}=0$) can be realised using columns of height $s_j$ \cite[Chapter 7]{HongKang2002}. More precisely, $B(\om_{s_j})$ is the \textit{$A_{n-1}$-crystal} with vertices the columns of height $s_j$ and arrows $i$ from $c$ to $c'$ when $c'$ is obtained from $c$ by changing $i$ into $i+1$. Observe that the trivial crystal of highest weight $0$ can so be realised as the graph with a unique vertex: the empty column or the column containing all the integers $1, \ldots ,n$.

\begin{defi} Let $\bs=(s_1,\ldots,s_\ell)\in R(n,\ell)$.
The \textit{combinatorial Fock space} associated to $\bs$ is the $A_{n-1}$-crystal 
$$F(\bs)=B(\om_{s_\ell})\otimes\cdots\otimes B(\om_{s_1})$$ 
\end{defi}

By classical crystal theory, the elements of $F(\bs)$ can be realised 
as tensor products of $\ell$ columns with entries in $\{1,\ldots,n\}$.
Let us recall the rule for computing $F(\bs)$, following \cite[Section 4.4]{HongKang2002}.
Fix $i\in\{1,\ldots,n-1\}$ and let $b=c_\ell\otimes\cdots\otimes c_1\in F(\bs)$.
consider the subword $\textrm{w}_i(b)$ of $\textrm{w}(b)$ obtained by keeping only letters $i$ and $i+1$,
and encode each $i$ by a symbol $+$ and each $i+1$ by a symbol $-$.
Deleting all factors $+-$ recursively yields a word called the \textit{$i$-signature} of $b$.

\begin{thm}\label{thm_crystal}
The action of the $A_{n-1}$-crystal operators on $F(\bs)$ is given by the following rule
\begin{enumerate}
\item The raising crystal operator $e_i$ acts on $b\in F(\bs)$ by changing the entry $i+1$ corresponding
to the rightmost $-$ in the $i$-signature of $b$ into $i$ if it exists; and by $0$ otherwise.
\item The lowering crystal operator $f_i$ acts on $b\in F(\bs)$ by changing the entry $i$ corresponding
to the leftmost $+$ in the $i$-signature into $i+1$ if it exists; and by $0$ otherwise.
\end{enumerate}
\end{thm}

We define similarly the combinatorial Fock space $\dF(\dbs)=B(\om_{\ds_1})\otimes\cdots\otimes B(\om_{\ds_n})$ for $\dbs=(\ds_1,\ldots,\ds_n)\in\Z^n$, and
the rule for computing $\dF(\dbs)$ is the same as for $F(\bs)$, expect that the role of $n$ and $\ell$ have been swapped\footnote{
Note that the order in which the components of $\bs$ and $\dbs$ are enumerated is reversed. 
Though this seems artificial at this point, this will be crucial in \Cref{duality_def}}.

\begin{exa}
Let $\ell=4$, $n=3$, $\bs=(2,2,1,2)$, $b=\scriptsize\young(1,2)\otimes\young(1,3)\otimes\young(1)\otimes\young(1,2)=c_4\otimes c_3\otimes c_2\otimes c_1$
and choose $i=1$.
Then $\textrm{w}(b)=1213112$, so that $\textrm{w}_1(b)=121112$, which gives the encoded word $+-+++-$. 
Thus, the $i$-signature of $b$ is $++$,
whose leftmost $+$ corresponds to the entry $1$ of $c_3$.
Therefore, we have $f_1b=\scriptsize\young(1,2)\otimes\young(2,3)\otimes\young(1)\otimes\young(1,2)=c_4\otimes c_3\otimes c_2\otimes c_1$.
In other terms, in the crystal $F(\bs)$, we have an arrow
\begin{center}
\begin{tikzpicture}
\node (a) at (0,0) {$\scriptsize\young(1,2)\otimes\young(1,3)\otimes\young(1)\otimes\young(1,2)$};
\node (b) at (4.5,0) {$\scriptsize\young(1,2)\otimes\young(2,3)\otimes\young(1)\otimes\young(1,2).$};
\draw[->] (a) --  node[pos=0.5,above]{\tiny 1} (b);
\end{tikzpicture}\end{center}
\end{exa}

Let us now explain how the tableaux and Yamanouchi elements
naturally appear in the context of crystals.
In the following, we set $$F(s)=\bigoplus_{\bs\in R(n,\ell)(s)}F(\bs)\mand \dF(s)=\bigoplus_{\dbs\in R(\ell,n)(s)}\dF(\dbs).$$
The following results are well-known, see for instance \cite{Lothaire2002}.

\begin{thm}\label{thm_yam_tab}\
\begin{enumerate}
\item The set of tableaux in $F(s)$ is closed under the crystal operators,
and tableaux of a given shape $\la\in\sS(s)$ form a connected component of $F(s)$ denoted $B(\la)$.
Moreover, for any $b\in F(s)$, there is a unique tableau $P(b)\in F(s)$
such that the induced map $b\mapsto P(b)$ is an $A_{n-1}$-crystal isomorphism.
\item An element $b\in F(s)$ is a highest weight vertex in the crystal if and only if $b$ is Yamanouchi.
\end{enumerate}
\end{thm}

\begin{rem}
In fact, for all $\la\in \sS(s)$, $B(\la)$ is the crystal of the irreducible highest weight module with highest weight $\la$.
\end{rem}

\Cref{thm_yam_tab} also holds for $\dF(s)$, replacing tableaux by antitableaux.
We denote similarly $\dF(\mu)$ the connected component of $\dF(s)$ consisting of
all antitableaux of shape $\mu\in\dot{\sS}(s)=\left\{ \mu\vdash s\mid \mu^\trans\in\sS(s) \right\}$.

\begin{exa}
Take $n=3$, $\ell=3$ and $\bs=(s_3,s_2,s_1)=(2,1,2)$.
Then one checks that $b=\scriptsize\young(1,2)\otimes\young(3)\otimes\young(1,2)\in F(\bs)$ is Yamanouchi.
We compute the connected component of $F(\bs)$ containing $b$:
\begin{center}
\begin{tikzpicture}
\node (a) at (0,0) {$\scriptsize\young(1,2)\otimes\young(3)\otimes\young(1,2)$};
\node (b) at (3,0) {$\scriptsize\young(1,2)\otimes\young(3)\otimes\young(1,3)$};
\node (c) at (6,0) {$\scriptsize\young(1,2)\otimes\young(3)\otimes\young(2,3).$};

\draw[->] (a) --  node[pos=0.5,above]{\tiny 2} (b);
\draw[->] (b) --  node[pos=0.5,above]{\tiny 1} (c);
\end{tikzpicture}\end{center}
This is isomorphic to the crystal $B(\om_2)$, which we can compute:
\begin{center}
\begin{tikzpicture}
\node (a) at (0,0) {$\scriptsize\young(11,22,3)$};
\node (b) at (2,0) {$\scriptsize\young(11,23,3)$};
\node (c) at (4,0) {$\scriptsize\young(12,23,3).$};

\draw[->] (a) --  node[pos=0.5,above]{\tiny 2} (b);
\draw[->] (b) --  node[pos=0.5,above]{\tiny 1} (c);
\end{tikzpicture}\end{center}
This means that $P(b)=\scriptsize\young(11,22,3)$, and so on.
Alternatively, we can use an isomorphic realisation of this crystal by antitableaux:
\begin{center}
\begin{tikzpicture}
\node (a) at (0,0) {$\scriptsize\gyoung(:~1,12,23)$};
\node (b) at (2,0) {$\scriptsize\gyoung(:~1,12,33)$};
\node (c) at (4,0) {$\scriptsize\gyoung(:~1,22,33).$};

\draw[->] (a) --  node[pos=0.5,above]{\tiny 2} (b);
\draw[->] (b) --  node[pos=0.5,above]{\tiny 1} (c);
\end{tikzpicture}\end{center}
\end{exa}

In general, $P(b)$ can be computed by carrying out one of the following procedures:
\begin{itemize}
\item performing Schensted's insertion on the word $\textrm{w}(b)$  \cite[Section 1.1]{Fulton1997},
\item performing the Jeu de Taquin on the skew Young tableau corresponding to $b$ \cite[Section 1.2]{Fulton1997},
\item applying a sequence of plactic relation to $\textrm{w}(b)$  \cite[Section 2.1]{Fulton1997}.
\end{itemize}

\begin{exa}\label{exa_jdt} Let $\ell=2$, $n=3$ and 
\Yvcentermath1
let $b=\scriptsize\young(1,2,3)\otimes\young(2)$.
Let us compute $P(b)$ by using the first two methods.
The word associated to $b$ is $\textrm{w}(b)=12413$, and Schensted's insertion
yields the following sequence of tableaux
$$
\scriptsize\young(1)\quad,\quad
\scriptsize\young(1,2)\quad,\quad
\scriptsize\young(1,2,3)\quad,\quad
\scriptsize\young(12,2,3)=P(b).$$
Now, the minimal skew Young tableau\footnote{
This skew tableau is minimal in the sense that its skew shape is minimal (for the
inclusion of skew shapes) among all the possible shapes of the skew tableaux associated to $b$.} associated to $b$ is
$${\scriptsize\gyoung(:~1,:~2,23)}$$
and the Jeu de Taquin corresponds to the following two sliding operations
\Yvcentermath1
$$
{\scriptsize\gyoung(:~1,:\bullet2,23)}
\to 
{\scriptsize\gyoung(:~1,22,:\bullet3)}
\to 
{\scriptsize\gyoung(:~1,22,3:\bullet)}
\quad,\quad
{\scriptsize\gyoung(:\bullet1,22,3)}
\to 
{\scriptsize\gyoung(1:\bullet,22,3)}
\to 
{\scriptsize\gyoung(12,2:\bullet,3)}=P(b).
$$
\end{exa}

\subsection{The duality.}\label{duality_def}

There is a duality
$$
\begin{array}{rcl}
F(s)
&
\longleftrightarrow
&
\dF(s)
\\
b
&
\longleftrightarrow
&
b^*
\end{array}
$$
defined as follows.
If $b=c_\ell\otimes\cdots\otimes c_1$ is a tensor product of columns, then
for each $i=1,\ldots,n$, let $d(i)$ be the column with letters in the set 
$$\{ j\in\{1,\ldots, \ell\} \mid  i\in c_j \}.$$
Then we set $b^*=d(1)\otimes\cdots\otimes d(n) \in \dF(s)$.

\Yvcentermath0

\begin{exa}\label{exa_duality}
Let $\ell=5$ and $n=4$.
Take 
$$b=\scriptsize\young(1,2,3)\otimes\young(4)\otimes\young(1,4)\otimes\young(1,2,4)\otimes\young(1,3).$$
Then 
$$ d(1)=\{1,2,3,5\}, \quad d(2)=\{2,5\}, \quad d(3)=\{1,5\}, \quad d(4)=\{2,3,4\},$$
so that
$$b^*=\scriptsize\young(1,2,3,5)\otimes\young(2,5)\otimes\young(1,5)\otimes\young(2,3,4).$$
\end{exa}

\begin{rem}\label{binary_mat}
As mentioned in the introduction, we can use binary matrices to represent elements in $F(s)$ and $\dF(s)$,
which yields an easy description of the duality $*$.
More precisely, encode $b=c_\ell\otimes\cdots\otimes c_1$ by the $n\times \ell$ matrix $M$
defined by 
$$M_{i,j}=\left\{ 
\begin{array}{ll}
1 & \text{ if } i\in c_j \\
0 & \text{ otherwise } \\ 
\end{array}
\right.
$$
Then $b^*$ is the element of $\dF(s)$ encoded by $M^\trans$, the transpose of $M$.
For instance, take $b$ as in \Cref{exa_duality}.
Then $b$ and $b^*$ are respectively encoded by the following matrices  
(remember that we read the columns of $b$ starting from the right)
$$
M =\begin{bmatrix}
  1&1&1&0&1\\
  0&1&0&0&1\\
  1&0&0&0&1\\
  0&1&1&1&0
  \end{bmatrix}
\mand
M^\trans=
\begin{bmatrix}
1&0&1&0\\
1&1&0&1\\
1&0&0&1\\
0&0&0&1\\
1&1&1&0
\end{bmatrix}.
$$
Therefore, we recover the crystal skew Howe duality of \cite[Section 11.2]{Halacheva2016}, also studied in \cite{VanLeeuwen2006}.
\end{rem}

The duality $*$ intertwines several classical notions. 
We already observe some occurences of this phenomenon now,
and will give more results in the upcoming sections.
For all map $\varphi:\dF(s)\to\dF(s)\sqcup \{0\}$, denote 
$\varphi^*:F(s)\to F(s)\sqcup \{0\}$ the map determined by the formula 
$$(\varphi^*(b))^*=\left\{
\begin{array}{cl}
\varphi(b^*) & \text{ if } \varphi(b^*)\in \dF(s)
\\
0 &  \text{ if } \varphi(b^*)=0.
\end{array}
\right.$$
that is, $\varphi^*$ is the conjugation of $\varphi$ by the duality $*$.
Similarly, for all map $\psi:F(s)\to F(s)\sqcup \{0\}$, denote ${}^*\psi:\dF(s)\to\dF(s)\sqcup \{0\}$ the map determined by the formula 
$${}^*\psi(b^*)=\left\{
\begin{array}{cl}
(\psi(b))^* & \text{ if } \psi(b)\in F(s)
\\
0 &  \text{ if } \psi(b)=0.
\end{array}
\right.
$$
Note that we have $\psi=\varphi^*$ if and only if ${}^*\psi=\varphi$.

\begin{prop}\label{duality_intertwines_Yam_tab}
The element $b$ is a tableau (respectively Yamanouchi) if and only if $b^*$ is Yamanouchi (respectively an antitableau).
\end{prop}

\begin{proof}
Let $b\in F(\bs)$ be a tableau. Then $s_j\geq s_{j+1}$ for all $j=1,\ldots,\ell-1$, which means that the number
of $j$'s in $b^*$ is greater than or equal to the number of $j+1$'s, which is a necessary condition to be Yamanouchi.
In fact, the $k$ first components of $b$ correspond by duality to the subtableau of $b$ with letters less or equal to $k$. Since all such subtableau is semistandard, the vertex $b^*$ has the Yamanouchi property. The converse holds by the same observation, and 
the analogue statement with antitableaux holds similarly.
\end{proof}

\Yvcentermath1

\begin{exa}\
\begin{enumerate}
\item Let $\ell=2$, $n=4$, and $b=\scriptsize\young(1,3,4)\otimes\young(1,2,4)$, so that $b$ is a tableau.
Then $b^*=\scriptsize\young(1,2)\otimes\young(1)\otimes\young(2)\otimes\young(1,2)$, which is Yamanouchi.
\item Let $\ell=5$, $n=3$, and $b=\scriptsize\young(1)\otimes\young(1)\otimes\young(1,2)\otimes\young(2)\otimes\young(1,2,3)$, so that $b$ is Yamanouchi.
Then $b^*=\scriptsize\young(1,3,4,5)\otimes\young(1,2,3)\otimes\young(1)$, which is an antitableau.
\end{enumerate}
\end{exa}

We have already recalled the Jeu de Taquin procedure for computing the tableau in \Cref{thm_yam_tab}.
Let us denote by $J_{j}$ the map $F(s)\to F(s)\sqcup \{0\}$ where $J_{j}(b)$ is obtained from $b$
by an elementary horizontal Jeu de Taquin slide from column $j+1$ to column $j$ if possible, and $J_{j}(b)=0$ otherwise.
Let us moreover recall that
there is a unique isomorphism of $A_{n-1}$-crystals
$B(\om_{i})\otimes B(\om_{i'})\overset{\sim}{\lra} B(\om_{i'})\otimes B(\om_{i})$
called the \textit{combinatorial $R$-matrix},
which we can compute by a simple combinatorial procedure, see \cite[Section 4.8]{Shimozono2005} and the references therein.
It induces an isomorphism 
$$R_{j,j'}:
F(s)
\lra
F(s)
$$
permuting the components $B(\om_{s_j})$ and $B(\om_{s_{j'}})$,
which is the composition of $R$-matrices of the form $R_{j,j+1}$, which we denote $R_j$ for simplicity.
Finally, recall that the Weyl group (here the symmetric groups on $n$ letters) acts on the crystal $F(s)$ 
by letting the Coxeter generators $\si_i$, $i=1\ldots, n-1$
act by reversing each $i$-string \cite[Definition 2.35]{BumpSchilling2017},
see also \cite[Section 2.1]{Shimozono2005}.

All of the above maps have counterparts defined on $\dF(s)$
(defined using antitableaux instead of tableaux when needed).
We now prove that the duality $*$ intertwines crystal operators with elementary Jeu de Taquin slides,
as well as the Weyl group action with the $R$-matrix.

\begin{thm}\label{kas-jdt_weyl-Rmat}For all $j=1,\ldots, \ell-1$, we have
$$\text{(1) }\quad \de^*_j  =J_{j} \quad \mand \quad  \text{(2) }\quad  \dot{\si_j}^*=R_{j}.$$
\end{thm}

\begin{proof}\
\begin{enumerate}
\item Assume first that $b^*$ is a highest weight vertex.
Then $\de_jb^*=0$, so we have by definition $\de_j^*b=0=J_{j}(b)$.
Note that this simply means that the Jeu de Taquin is not authorised for $b$, which makes sense
because $b$ is a tableau by \Cref{thm_yam_tab} and \Cref{duality_intertwines_Yam_tab}.
Now, assume that $b^*=d_1\otimes\cdots\otimes d_n$ is not a highest weight vertex, so that $\de_j b^*\in \dF(s)$.
Let $i$ denote the index of the column containing the entry of $b^*$ affected by $\de_j$,
see \Cref{thm_crystal}.
Then by definition of the duality $*$, $\de_j^*$ acts on $b=c_\ell\otimes\cdots\otimes c_1$ 
by sliding entry $i$ from $c_{j+1}$ to $c_j$.
Now, consider the Jeu de Taquin between columns $j$ and $j+1$ of $b$. 
The procedure used to determine $i$ ensures us that
all entries of $c_j$ that are smaller than or equal to $i$ are matched to an element of $c_{j+1}$.
Therefore, the first entry that slides from $c_{j+1}$ to $c_j$ is $i$. 
In other terms, $J_{j}(b)$ is obtained from $b$ by sliding $i$ from $c_{j+1}$ to $c_j$.
Thus $\de_j^*b=J_{j}(b)$.
\item Set $N=|c_{j+1}|-|c_j|$, where $b=c_\ell\otimes \cdots\otimes c_1$. 
If $N\geq0$, then the element $R_{j}(b)$ is obtained by
using $N$ Jeu de Taquin slides between columns $j$ and $j+1$ of $b$.
Therefore, by Part (1), $(R_{j}(b))^*=\de_j^N b^*$.
Now, by definition of the duality $*$, we have
$N=\eps_j(b^*)-\varphi_j(b^*)$, therefore $\de_j^N b^*=\dot{\si_j}(b^*)$
and the claim is proved.
The case $N<0$ is proved similarly by using the lowering operators $\df_j$ instead.
\end{enumerate}
\end{proof}

\begin{exa}
\begin{enumerate}
\item Take $n=6$ and $\ell=4$, and let 
$$b^*={\scriptsize\young(1,2,4)\otimes\young(2,3)\otimes\young(2,3,4)\otimes\young(1)\otimes\young(1,3)\otimes\young(2)},
\text{\quad so that \quad} 
b={\scriptsize\young(1,3)\otimes\young(2,3,5)\otimes\young(1,2,3,6)\otimes\young(1,4,5)}.$$
Then one can check that
$$\de_1 b^*={\scriptsize\young(1,2,4)\otimes\young(2,3)\otimes\young(1,3,4)\otimes\young(1)\otimes\young(1,3)\otimes\young(2)},
\text{\quad whose dual is \quad} 
b'={\scriptsize\young(1,3)\otimes\young(2,3,5)\otimes\young(1,2,6)\otimes\young(1,3,4,5)}.$$
On the other hand, the horizontal slide from column $2$ to column $1$ of $b$ is achieved as
$$
{\scriptsize\gyoung(:~1,:\bullet 2,13,46,5)}
\to
{\scriptsize\gyoung(:~1,12,:\bullet 3,46,5)}
\to
{\scriptsize\gyoung(:~1,12,3:\bullet,46,5)}
\to
{\scriptsize\gyoung(:~1,12,36,4:\bullet,5)}
=
{\scriptsize\gyoung(1,2,6)}\otimes 
{\scriptsize\gyoung(1,3,4,5)},
$$
and we see that we recover $b'$.
\item Take $n=5$, $\ell=4$, and 
$$
b^*={\scriptsize\young(1,3)\otimes\young(3,4)\otimes\young(1,2,3)\otimes\young(3)\otimes\young(1,2)}
\text{\quad so that \quad} 
b={\scriptsize\young(2)\otimes\young(1,2,3,4)\otimes\young(3,5)\otimes\young(1,3,5)}.
$$
Let us look at $j=2$. 
One checks that $\eps_j(b^*)=3$ and $\varphi_j(b^*)=1$, 
(i.e. there are $3$ incoming and $1$ outgoing arrows with color $2$ at vertex $b^*\in \dF(s)$), so that
$$
s_2 b^*={\scriptsize\young(1,3)\otimes\young(2,4)\otimes\young(1,2,3)\otimes\young(2)\otimes\young(1,2)},
\text{\quad whose dual is \quad} 
b'={\scriptsize\young(2)\otimes\young(1,3)\otimes\young(2,3,4,5)\otimes\young(1,3,5)}.
$$
Now, we can compute $R_{2}(b)$, by doing the following two Jeu de Taquin slides between columns $2$ and $3$
$$
{\scriptsize\gyoung(:~1,:~2,:\bullet 3,34,5)}
\to
{\scriptsize\gyoung(:~1,:~2,33,:\bullet4,5)}
\to
{\scriptsize\gyoung(:~1,:~2,33,4:\bullet,5)}
\quad, \quad 
{\scriptsize\gyoung(:~1,:\bullet2,33,4,5)}
\to
{\scriptsize\gyoung(:~1,2:\bullet,33,4,5)}
\to
{\scriptsize\gyoung(:~1,23,3:\bullet,4,5)}
=
{\scriptsize\gyoung(1,3)}\otimes 
{\scriptsize\gyoung(2,3,4,5)},
$$
from which we recover $b'$.
\end{enumerate}
\end{exa}

Of course, \Cref{kas-jdt_weyl-Rmat} also holds when switching $F(s)$ and $\dF(s)$
and taking the dual versions of the different maps.

\subsection{Bicrystal structure}\label{bicrystal}

We are ready to prove that
the $A_{n-1}$ and $A_{\ell-1}$ crystal structures commute (modulo the duality $*$).
This is best stated as follows.

\begin{thm}\label{crystals_commute}
For all $j=1,\ldots,\ell-1$, the restriction of $\df_j^*$ to any connected component of $F(s)$ is either $0$ or an $A_{n-1}$-crystal isomorphism.
Similarly, for all $i=1,\ldots,n-1$, the restriction of ${}^*f_i$ to any connected component of $\dF(s)$ is either $0$ or an $A_{\ell-1}$-crystal isomorphism.
\end{thm}

\begin{proof}
Note that by symmetry, both statements are equivalent, and it is enough to prove that
for all $j=1,\ldots,\ell-1$, the map $\df_j^*$,  is an $A_{n-1}$-crystal isomorphism.
Fix $j\in\{1,\ldots,\ell-1\}$.
First of all, if $b\in F(s)$ is a highest weight vertex, then $b$ is Yamanouchi
by \Cref{thm_yam_tab}(2), and $b^*$ is an antitableau by \Cref{duality_intertwines_Yam_tab}.
By \Cref{thm_yam_tab}(1), $\df_jb^*$ is again an antitableau, 
so $\df_j^*b$ is Yamanouchi, i.e. a highest weight vertex.
This shows that $\df_j^*$ maps highest weight vertices to highest weight vertices with the same weight. 
It remains to show that it commutes with the lowering crystal operators $f_i$.
We have, for all $i=1,\ldots,n-1$,
\begin{alignat*}{4}
\df_j^* f_i & = J_{j}^{-1} f_i && \text{ \quad by \Cref{kas-jdt_weyl-Rmat}(1)}
\\
& = f_i J_{j}^{-1} && \text{ \quad by \Cref{thm_yam_tab}(1)}
\\
&= f_i \df_j^*. &&
\end{alignat*}
\end{proof}

This yields an $(A_{n-1}\times A_{\ell-1})$-crystal structure on $F(s)$.
For $\underline{i}=(i_1,\ldots,i_r)\in \{1, \ldots, n \}^r$ 
(respectively $\underline{j}=(j_1,\ldots,j_t)\in \{1, \ldots, \ell \}^t$), 
denote $f_{\underline{i}}=f_{i_r}\cdots f_{i_1}$
(respectively $ \df_{\underline{j}}=\df_{j_t}\cdots \df_{j_1}$).
The following corollary is immediate from \Cref{duality_intertwines_Yam_tab} and \Cref{crystals_commute}.
For a given $\la\in P_+$, denote $b_\la$ the Yamanouchi tableau of shape $\la$.

\begin{cor}\label{sources_finite}
Each connected component of the $(A_{n-1}\times A_{\ell-1})$-crystal $F(s)$ has a unique source vertex.
The sources are exactly the Yamanouchi tableaux.
In other terms, we have
$$F(s)=\bigoplus_{\substack{\la=\om_{s_1}+\cdots+\om_{s_\ell}\\ s_1+\cdots+ s_\ell=s}} \df_{\underline{j}}^* f_{\underline{i}} b_\la,$$
where the sum runs over all possible $\underline{i}$ and $\underline{j}$.
\end{cor}

\begin{rem}
Accordingly, the $(A_{n-1}\times A_{\ell-1})$-crystal structure can be considered on $\dF(s)$.
In this case, the sources are the Yamanouchi antitableaux.
\end{rem}

\Yvcentermath0

\begin{exa}\label{exa_source}
Take $\ell=4$, $n=3$. The element
$$\scriptsize\young(1)
\otimes
\young(1,2)
\otimes
\young(1,2)
\otimes
\young(1,2,3)
$$
is a Yamanouchi tableau, i.e. a source in the $(A_{n-1}\times A_{\ell-1})$-crystal.
Its dual is the following Yamanouchi antitableau
$$\scriptsize
\young(1,2,3,4)
\otimes
\young(1,2,3)
\otimes
\young(1).
$$
\end{exa}

\Yvcentermath1

By \Cref{sources_finite}, for all $b\in F(s)$, there is a unique Yamanouchi tableau $\overline{b}$
such that $b=\df^*_{\underline{j}} f_{\underline{i}} \overline{b}$
for some $\underline{i}$ and $\underline{j}$.
Set 
$$P(b)=f_{\underline{i}} \overline{b} \mand Q(b)=\df_{\underline{j}}\overline{b}^*.$$

\begin{thm}\label{rsk}
For all $b\in F(s)$, $P(b)$ is a tableau and $Q(b)$ is an antitableau of transpose shape.
Moreover,  the assignment
$$\Phi : b \longmapsto (P(b), Q(b)) $$
yields a bijection $F(s) \to \Phi(F(s))$ called the crystal RSK correspondence.
In particular, we have the decomposition
$$F(s) \simeq \bigoplus_{\la\in\sS(s)} B(\la)\otimes \dB(\la^\trans).$$
\end{thm}

\begin{proof}
By \Cref{crystals_commute}, the two crystals commute, so 
since $\overline{b}$ is a highest weight vertex for the $A_{\ell-1}$-crystal,
$P(b)=f_{\underline{i}} \overline{b}$ is also a highest weight vertex for the $A_{\ell-1}$-crystal,
i.e. $P(b)^*$ is Yamanouchi. By \Cref{duality_intertwines_Yam_tab},
$P(b)$ is a tableau.
Similarly, we get that $Q(b)$ is an antitableau.
Moreover, $\Phi$ is clearly injective, so that $\Phi : F(s) \to \Phi(F(s))$ is a bijection.
\end{proof}

Accordingly, for $b^*\in \dF(s)$, we will write $P(b^*)=Q(b)$ and $Q(b^*)=P(b)$.

\subsection{Crystal structure on self-dual elements}\label{selfdual_finite}

An element $b=c_\ell \otimes \cdots\otimes c_1 \in F(s)$ is called \textit{self-dual}
if $b^*=c_1\otimes \cdots\otimes c_\ell$. In particular, self-dual elements exist only if $n=\ell$.
Alternatively, if $M$ denotes the binary matrix associated to $b$, then $b$ is self-dual if and only if $M^\trans=M$.
We denote $F(s)^*$ the set of self-dual elements of $F(s)$.
Now, for all $i=1,\ldots, n-1$, set
$$f_i^* = \df_i^*f_i.$$
The operators $f_i^*, i=1,\ldots, n-1$, induce an $A_{n-1}$-crystal structure on $F(s)^*$. More precisely, we get the decomposition
$$F(s)^*\simeq \bigoplus_{\substack{\la\in\sS(s)\\ \la^\trans=\la}}B(\la).$$

\subsection{Keys and bikeys}\label{bikeys}

Keys are certain tableaux introduced by Lascoux and Sch\"utzenberger \cite{LS1990}
that are used to compute Demazure crystals \cite{BumpSchilling2017}. In this section, we generalise
the notion of keys using the bicrystal structure of \Cref{bicrystal}.

\begin{defi}
An element $b=c_\ell\otimes\cdots\otimes c_1 \in F(s)$ is called a \textit{key}
(respectively an \textit{antikey}) if $b$ is a tableau (respectively an antitableau) and if
$c_\ell \subseteq\cdots\subseteq c_1$ (respectively $c_1 \subseteq\cdots\subseteq c_\ell)$.
\end{defi}

For $b\in F(s)$, let $\cO_{S_n}(b)$ be the orbit of $b$ under the action of the Weyl group $S_n$.
The following proposition is easy to establish by induction on the length of the elements of $S_n$.

\begin{prop}\label{key_orbit}
The set of all keys of given shape $\la\in\sS(s)$ is equal to $\cO_{S_n}(b_\la)$.
\end{prop}

Therefore, we have the following characterisation of keys (which could also be proved
directly by using the definition of $*$).

\begin{prop}\label{key_dual}
Let $b\in F(s)$.
Then $b$ is a key if and only if $b^*$ is a product of Yamanouchi columns.
\end{prop}

\begin{proof}
We know that $b_\la^*$ is the Yamanouchi antitableau of shape $\la^\trans$.
In particular, $b_\la^*$ is an antikey.
Now by \Cref{kas-jdt_weyl-Rmat}(2) and \Cref{key_orbit}, $b$ is a key if and only if $b^*$ is obtained
from $b_\la^*$ by applying a sequence of $R_j$'s, which simply permute the columns since 
each column of $b_\la^*$ is included in the next.
\end{proof}

\begin{exa}
Take $n=3$, $\ell=2$.
Then the crystal of highest weight $\la=\om_1+\om_2$ 
is realised by tableaux as follows
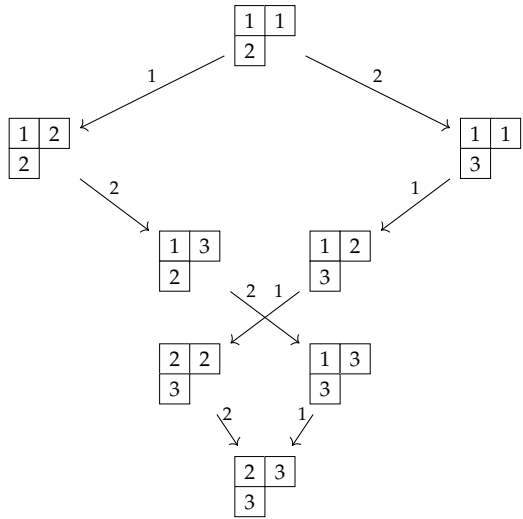
\begin{center}
\begin{tikzpicture}
\node (a) at (0,0) {$\scriptsize\young(11,2)$};
\node (b1) at (-3,-1.5) {$\scriptsize\young(12,2)$};
\node (b2) at (3,-1.5) {$\scriptsize\young(11,3)$};
\node (c1) at (-1,-3) {$\scriptsize\young(13,2)$};
\node (c2) at (1,-3) {$\scriptsize\young(12,3)$};
\node (d1) at (1,-4.5) {$\scriptsize\young(13,3)$};
\node (d2) at (-1,-4.5) {$\scriptsize\young(22,3)$};
\node (e) at (0,-6) {$\scriptsize\young(23,3)$};

\draw[->] (a) --  node[pos=0.5,above]{\tiny 1} (b1);
\draw[->] (a) --  node[pos=0.5,above]{\tiny 2} (b2);
\draw[->] (b1) --  node[pos=0.5,above]{\tiny 2} (c1);
\draw[->] (b2) --  node[pos=0.5,above]{\tiny 1} (c2);
\draw[->] (c1) --  node[pos=0.3,above]{\tiny 2} (d1);
\draw[->] (c2) --  node[pos=0.3,above]{\tiny 1} (d2);
\draw[->] (d1) --  node[pos=0.5,above]{\tiny 1} (e);
\draw[->] (d2) --  node[pos=0.5,above]{\tiny 2} (e);
\end{tikzpicture}\end{center}
The keys of shape $\la$ are obtained by keeping all vertices
at the extremity of each $i$-string (for $i=1,2$), namely
$$
\scriptsize\young(11,2)\quad,\quad 
\scriptsize\young(12,2)\quad,\quad 
\scriptsize\young(11,3)\quad,\quad 
\scriptsize\young(13,3)\quad,\quad 
\scriptsize\young(22,3)\quad,\quad 
\scriptsize\young(23,3)\quad,
$$
whose respective dual are the following products of Yamanouchi columns
$$
\scriptsize\young(1,2)\otimes\young(2)\otimes\emptyset\quad,\quad 
\scriptsize\young(2)\otimes\young(1,2)\otimes\emptyset\quad,\quad 
\scriptsize\young(1,2)\otimes\emptyset\otimes\young(2)\quad,\quad 
\scriptsize\young(2)\otimes\emptyset\otimes\young(1,2)\quad,\quad 
\scriptsize\emptyset\otimes\young(1,2)\otimes\young(2)\quad,\quad 
\scriptsize\emptyset\otimes\young(2)\otimes\young(1,2)\quad .
$$
\end{exa}

We are ready to introduce the generalised notion of key.

\begin{defi}
An element $b=c_{1}\otimes\cdots\otimes c_{\ell}\in F(s)$ is called a \textit{bikey} 
if there exists $\dot{w}\in S_{\ell}$ such that 
$c_{\dot{w}(1)}\subseteq\cdots\subseteq c_{\dot{w}(\ell)}$.
\end{defi}

Write $\cK$ for the set of bikeys in $F(s)$. 
Observe that for bikeys, the distinction between tableaux and antitableaux disappears.
More precisely, 
when $b$ is a bikey, for any $j=1,\ldots,\ell-1$ we have the inclusion
$c_{j}\subseteq c_{j+1}$ or $c_{j+1}\subseteq c_{j}$. It thus follows from \Cref{kas-jdt_weyl-Rmat}(2) that 
$\dot{w}^*(b)=c_{\dot{w}(1)}\otimes\cdots\otimes c_{\dot{w}(\ell)}$.
Moreover, $\dot{w}^*(b)$ is a tableau (and an antitableau), thus for a 
bikey $b$ we must have $P(b)=\dot{w}^*(b)$.
For all $\la\in\sS(s)$, write
$$\cK(\lambda)=\{ b\in\cK \mid P(b) \text{ has shape } \lambda \}.$$
The next result justifies the terminology of the previous definition.

\begin{prop}\label{bikeys_biorbit} For all $\la\in\sS(s)$, we have 
$\cK(\la)= \cO_{S_n\times S_\ell}(b_\la).$
\end{prop}

\begin{proof}
We have already observed that $\dot{w}^*(b)$ is a tableau.
Since $c_{\dot{w}(1)}\subseteq\cdots\subseteq c_{\dot{w}(\ell)}$, it is a key of shape $\lambda$. 
Thus, there exists $w\in S_{n}$ such that $w\dot{w}^*(b)=b_\la$ which proves that $\cK(\lambda)\subseteq \cO_{S_n\times S_\ell}(b_\la)$. 
Now for any $w\in S_{n}$ and $\dot{w}\in S_{n}$, the vertex
$b=w\dot{w}^*(b_\la)=\dot{w}^*w(b_\la)$ is a bikey since
$w(b_\la)$ is a key on which the action of $\dot{w}^*$ 
(which is a combination of $R$-matrices) reduces to a permutation of the columns. Hence we get $\cK(\lambda)= \cO_{S_n\times S_\ell}(b_\la)$.
\end{proof}

Finally, we prove that the duality $*$ maps bikeys to bikeys.
For this, consider the set $\dot{\cK}$ of bikeys in $\dF(s)$,
and for $\mu\vdash s$ let
$$\dot{\cK}(\mu)=\{ \dot{b}\in\dot{\cK} \mid P(\dot{b}) \text{ has shape } \mu \}.$$

\begin{prop}\label{bikeys_biorbit_2} 
	We have $b\in\cK(\la)$ if and only if $b^{\ast}\in\dot{\cK}(\la^\trans)$.
\end{prop}

\begin{proof} 
It suffices to observe that for any $b\in\cK$, 
there exists a partition $\lambda$ and $(w,\dot{w})\in S_{n}\times S_{\ell}$ 
such that $b=w\dot{w}^*(b_\la)$. 
Then $b^{\ast}={}^*w\dot{w}(b_\la^{\ast})={}^*w\dot{w}(b_{\la^\trans})$ 
belongs to $\dot{\cK}$.
\end{proof}

\medskip

Let $\sK$ be the set of binary matrices corresponding to the
bikeys in $\cK$. For any $\lambda\in\sS(s)$, recall
that the binary matrix $M_{\lambda}$ corresponding to $b_\la$ is the
$n\times\ell$ matrix whose entries equal to $1$ form a pattern corresponding
to the Young diagram of $\lambda$. Also, we have a right $S_{\ell}$-action 
(respectively a left $ S_{n}$-action) on the set of $n\times\ell$
matrices where the action of $\dot{w}\in S_{\ell}$ (respectively of
$w\in S_{n}$) is given by the right (respectively left) multiplication by
the permutation matrix $P_{\dot{w}}$ (respectively $P_{w}$) associated to $\dot{w}$
(respectively $w$). Since permuting the columns of $b$ (respectively of $b^{\ast}$) is
equivalent to permuting the columns (respectively the rows) of the corresponding
binary matrix, we get the following corollary.

\begin{cor}
	We have
	\[
	\sK=%
	{\textstyle\bigsqcup\limits_{(w,\dot{w})\in S_{n}\times
			 S_{\ell}}}
	P_{w}M_{\lambda}P_{\dot{w}}%
	\]
	and
	\[
	\sum_{(m_{i,j})\in\sK}\prod_{\substack{1\leq i\leq n \\ 1\leq j\leq \ell}}(x_{i}%
	y_{j})^{m_{i,j}}=
	\sum_{\la\in\sS(s)}
	m_{\lambda}(x)m_{\lambda^\trans}(y).
	\]
	
\end{cor}

\subsection{Cyclage and promotion}\label{cyclage_promotion}

In this section, we show that the duality $*$ intertwines two important maps, namely  the
cyclage and the promotion operators. This will be used in \Cref{charge_energy} to study the relationship between charge and energy.

\begin{defi}\label{def_cyclage_prom}\
\begin{enumerate}
 \item Let $b=c_\ell\otimes\cdots\otimes c_1\in F(s)$.
The \textit{cyclage} of $b$ is the element of $F(s)$ defined by
$$\xi(b)=c_{\ell-1}\otimes\cdots\otimes c_1\otimes c_\ell.$$
\item Let $b^*=d_1\otimes\cdots\otimes d_n\in \dF(s)$.
The \textit{promotion} of $b^*$ is the element of $\dF(s)$ defined by
$\pr(b^*)= \pr(d_1)\otimes \cdots\otimes \pr(d_n)$
where, for all $i=1,\ldots,n$, $$\pr(d_i)=\{ k+1\mod \ell \,;\, k\in d_i \}$$
\end{enumerate}
\end{defi}

\begin{rem}
The promotion operator permits to endow $\dF(s)$ with the structure of an affine type $A^{(1)}_{\ell-1}$-crystal that we shall denote by $\dF(s)^{\mathrm{aff}}$. This structure is obtained by considering the additional crystal operators $\df_{0}=\pr^{-1}\df_{1}\circ\pr$ and $\de_{0}=\pr^{-1}\de_{1}\circ\pr$. The crystal $\dF(s)^{\mathrm{aff}}$ splits into affine connected components. Each such component is isomorphic 
to a tensor product of $n$ column Kirillov-Reshetikhin crystals. Observe that these components are not highest weight crystals.
\end{rem}

The following proposition
is immediate from the definitions.

\begin{prop}\label{duality_intertwines_cyc_prom}\
\begin{enumerate}
	\item For all $b\in F(s)$, we have $  \xi(b)^* = \pr(b^*)$.
	\item For all $b^*\in\dF(s)$, we have $ \pr\circ R(b^*)=R \circ\pr(b^*) $ 
\end{enumerate}
\end{prop}

\Yvcentermath0

\tikzcdset{row sep/normal=1cm}

\begin{exa}
Take the same values as in \Cref{exa_duality}.
Then we have
\begin{center}
\begin{tikzpicture}
\node (11) at (0,0) {$\scriptsize\young(1,2,3)\otimes\young(4)\otimes\young(1,4)\otimes\young(1,2,4)\otimes\young(1,3)$};
\node (12) at (6,0) {$\scriptsize\young(1,2,3,5)\otimes\young(2,5)\otimes\young(1,5)\otimes\young(2,3,4)$};
\node (21) at (0,-3) {$\scriptsize\young(4)\otimes\young(1,4)\otimes\young(1,2,4)\otimes\young(1,3)\otimes \young(1,2,3)$};
\node (22) at (6,-3) {$\scriptsize\young(1,2,3,4)\otimes\young(1,3)\otimes\young(1,2)\otimes\young(3,4,5)$};

\draw[|->] (11) --  node[pos=0.5,above]{$*$} (12);
\draw[|->] (11) --  node[pos=0.5,left]{$\xi$} (21);
\draw[|->] (21) --  node[pos=0.5,above]{$*$} (22);
\draw[|->] (12) --  node[pos=0.5,left]{$\pr$} (22);

\end{tikzpicture}\end{center}
\end{exa}

In order for the cyclage to be relevant for the description of the Kostka polynomials, we need to 
restrict to some cases, which we call \textit{authorised}.

\begin{defi} Let $b=c_\ell\otimes\cdots\otimes c_1\in F(s)$.
\begin{enumerate}
\item We say that the cyclage is \textit{authorised} for $b$ if either
\begin{enumerate}
\item $b$ is of \textit{dominant evaluation}, i.e. 
$b$ has no more $k+1$ than $k$ for all $k=1,\ldots,n-1$, and 
$1\notin c_\ell$.
\item $b$ is a tableau and there is no $k\in\{ 1,\ldots, n \}$ appearing in every column $c_j$.
\end{enumerate}
\item  Suppose that there exists $k\in\{1,\ldots,n\}$ such that $k\in c_j$ for all $j=1\ldots,\ell$.
In particular, $b$ is not authorised.
Let $k_0$ be minimal with this property. The \textit{reduction} of $b$ is 
the element $\red(b)$ obtained by deleting all occurences of $k_0$ in $b$, and
replacing each $k>k_0$ by $k-1$.
\end{enumerate}
\end{defi}

\begin{rem}
In particular, if $b$ verifies (2), there exists $m$ such that the cyclage is authorised for the tableau $\red^m(b)$.
We denote $\overline{\xi} : b \mapsto \xi(\red^m(b))$.  
\end{rem}

\begin{exa}
Take $\ell=2, n=3$ and 
$b=\scriptsize\young(1,2,3)\otimes\young(2,3)$.
We have $k_0=2$ and $\red(b)=\scriptsize\young(1,2)\otimes\young(2)$.
\end{exa}

It can be read off the dual if the cyclage is authorised, and how reduction acts.

\begin{prop} \
\begin{enumerate}
\item Let $b\in F(s)$ and write $b^*=d_1\otimes\cdots\otimes d_n$.
The cyclage is authorised for $b$ if and only if either
\begin{enumerate}
\item $|d_1|\geq \cdots \geq |d_n|$ and $\ell\notin d_1$, or
\item $b^*$ is Yamanouchi and for all $i=1,\ldots,n$, $d_i\neq \{1,\ldots, \ell\}$.
\end{enumerate}
\item $\red(b)^*$ is obtained from $\red(b)$ be deleting the leftmost column $\{1,\ldots,\ell\}$ in $b^*$
\end{enumerate}
\end{prop}

\begin{proof}
For (1), one checks from the definition that an element $b$ is of dominant evaluation if and only if
its dual $b^*$ has columns of non-increasing size, and (a) follows.
Point (b) is a direct consequence of \Cref{duality_intertwines_Yam_tab}.
Point (2) is clear from the definition of $*$.
\end{proof}

\subsection{Charge and energy}\label{charge_energy}

Charge and energy are two classic statistics defined on $F(s)$ and $\dF(s)$.
The goal of this section is to establish that the duality $*$ intertwines these two notions.

\begin{lem}\label{lem_cyclage_column}
For all tableau $b\in F(s)$, there exists $m>1$ such that
$P(\overline{\xi}^m(b))$ is a column.
\end{lem}

\begin{defi}\label{def_charge} The \textit{charge} is defined as follows. 
Let $b=c_\ell\otimes\cdots\otimes c_1\in F(s)$.
\begin{enumerate}
\item If $b$ is a column, set 
$\ch(b)=0.$
\item If $b$ is a tableau, define $\ch(b)$ by induction by setting 
$$\ch(b)=\ch(\xi(b))+|c_\ell| \mand \ch(\red(b))=\ch(b)$$
and using \Cref{lem_cyclage_column}.
\item In general, set $\ch(b)=\ch(P(b))$.
\end{enumerate}
 \end{defi}

We have seen in \ref{crystal_finite} how one can associate to any product $b^*=d_1\otimes d_2\in \dF(s)$ of two columns a unique skew tableau $T$ of minimal shape. Write $\mathrm{leg}(b^*)$ for the number of boxes in the left column of $T$ having no box to their right. 
The following definition of the \textit{local energy} $\H$ is slightly different but equivalent to the original one 
(see \cite{NY1995} and also \cite{Shimozono2005} for an equivalent definition in the case of rows and affine crystals). 
\begin{defi}\label{def_local_energy}
Let $\ell=2$ and $b^*=d_1\otimes d_2\in \dF(s)$. 
$$\H(b^*)=\left\{
\begin{array}{cl}
\mathrm{leg}(b^*) & \quad \text{ if $|d_1|\geq|d_2|$}
\\
\H(R_1(b^*)) & \quad \text{ if $|d_1|<|d_2|$.
}
\end{array}\right.
$$
\end{defi}

\Yvcentermath1
\begin{exa}
The minimal skew tableau corresponding to 
$$b^*= {\scriptsize\young(1,3,4,6)\otimes\young(1,7,8)}
\text{ \quad is \quad }
{
\scriptsize\gyoung(:~1,:~3,:~4,16,7,8)
}
\text{ , \quad therefore we get }
\H(b^*)=2.$$
\end{exa}

Let $b^*=d_1\otimes \cdots\otimes d_n\in \dF(s)$.
For all $1\leq i < j\leq n$, denote 
$$d_1\otimes \cdots\otimes d_{i}\otimes d_{i+1}^{(j)}\otimes \cdots\otimes d_n$$
the element obtained from $b^*$ by applying sucessively the $R$-matrices $R_{j-1},\ldots, R_{i+1}$.

\begin{defi}\label{def_energy}
The energy $\D$ is defined as follows.
For $b^*=d_1\otimes \cdots\otimes d_n\in \dF(s)$, we set 
$$\D(b^*) = \sum_{1\leq i<j\leq n} \H (d_i\otimes d_{i+1}^{(j)}).$$
\end{defi}

\begin{rem}
If $\ds_1=\cdots=\ds_n$, then the $R$-matrices act as the identity, thus
$$\D(b^*) = \sum_{1\leq i \leq n-1} (n-i) \H (d_i\otimes d_{i+1}).$$
\end{rem}

\begin{lem}\label{lem_energy_1}
Let $b^*=d_1\otimes\cdots\otimes d_n\in\dF(s)$ be Yamanouchi such that $d_i = \{1,\ldots, \ell\}$ for some $i=1,\ldots,n$.
Then $\D(b^*) = \D(b^\times),$
where $b^\times$ is the element obtained from $b^*$ by deleting its leftmost
column $\{ 1,\ldots, \ell \}$.
\end{lem}
\begin{proof}
	Since $D$ is invariant under the action of the combinatorial $R$-matrices, we can assume that $d_1 = \{1,\ldots, \ell\}$. Then 
	$$\D(b^*) = \D(b^\times)+\sum_{i=1}^{n-1} (n-i) \H (d_1\otimes d_{i+1}).$$
	But by definition of $\H$, we have $\H (d_1\otimes d_{i+1})=0$ for any $i=1,  \ldots, n-1$.
\end{proof}

We are now going to prove the following proposition.

\begin{prop}
	\label{prop_energy_3} Let $b^{\ast}=d_{1}\otimes\cdots\otimes d_{n}\in\dF(s)$
	be Yamanouchi such that $d_{i}\neq\{1,\ldots,\ell\}$ for all $i=1,\ldots,n$.
	Then $D(\pr(b^{\ast}))=D(b^{\ast})-N,$ where $N$ is the number of
	$\ell$ in $b^{\ast}$.
\end{prop}

To do this, let us set for any $b^{\ast}=d_{1}\otimes\cdots\otimes d_{n}%
\in\dF(s)$%
\[
\Delta_{n}(b^{\ast})=D(b^{\ast})-D(\pr(b^{\ast})).
\]
When $n=2$, one can use \Cref{def_local_energy} to compute $\Delta_{2}$. This
immediately gives:

\begin{lem}
	\label{lem_energy_2} Assume $b^{\ast}=d_{1}\otimes d_{2}$ with $h(d_{1})\geq
	h(d_{2})$. Then%
	\[
	\Delta_{2}(d_{1}\otimes d_{2})=\left\{
	\begin{array}{cl}
	-1 & \text{ if }\ell\in d_{1}\text{ and }\ell\notin d_{2},\\
	1 & \text{ if }\ell\notin d_{1}\text{ and }\ell\in d_{2},\\
	0& \text{ otherwise.}%
	\end{array}
	\right.
	\]
	
\end{lem}

\begin{proof}
	[Proof of Proposition \ref{prop_energy_3}]Since the combinatorial $R$-matrices
	preserve the energy, we can and shall assume that $h(d_{1})\geq\cdots\geq
	h(d_{n})$. For any Yamanouchi vertex $b^{\ast}=d_{1}\otimes\cdots\otimes
	d_{n}\in\dF(s)$, we have
	\[
	\Delta_{n}(b^{\ast})=\Delta_{n}(d_{1}\otimes\cdots\otimes d_{n})=\Delta
	_{n-1}(d_{1}\otimes\cdots\otimes d_{n-1})+\sum_{i=1}^{n-1}\Delta_{2}%
	(d_{i}\otimes d_{i+1}^{(n)}).
	\]
	because the promotion operator commutes with the $R$-matrices. Since $b^{\ast}$ is Yamanouchi, the vertex $d_{1}\otimes\cdots\otimes
	d_{n-1}$ is also Yamanouchi. By induction on $n$, we can assume that
	$\Delta_{n-1}(d_{1}\otimes\cdots\otimes d_{n-1})$ is equal to the number of
	letters $\ell$ in $d_{1}\otimes\cdots\otimes d_{n-1}$. This shows that the
	proposition is in fact equivalent to the assertion%
	\[
	S_{n}(b^{\ast}):=\sum_{i=1}^{n-1}\Delta_{2}(d_{i}\otimes d_{i+1}%
	^{(n)})=\left\{
	\begin{array}
	{cl}%
	1&\text{ if }\ell\in d_{n},\\
	0&\text{ otherwise.}%
	\end{array}
	\right.
	\]
	for any Yamanouchi vertex $b^{\ast}=d_{1}\otimes\cdots\otimes d_{n}\in\dF(s)$.
	Here again, we proceed by induction on $n$. When $n=2$, this follows from
	Lemma \ref{lem_energy_2}. Note that we cannot have $\ell\in d_{1}$ in this case  because we have assumed that $d_{1}\neq\{1,\ldots,\ell\}$. When $n\geq3$, set%
	\[
	b^{\ast}(n-1)=d_{1}\otimes\cdots\otimes d_{n-2}\otimes d_{n-1}^{(n)}.
	\]
	Then $b^{\ast}(n-1)$ is Yamanouchi and
	\[
	S_{n}(b^{\ast})=S_{n-1}(b^{\ast}(n-1))+\Delta_{2}(d_{n-1}\otimes d_{n}).
	\]
	Assume first $\ell\notin d_{n}$ and $\ell\notin d_{n-1}$. Then, we have
	$\ell\notin d_{n-1}^{(n)}$ and by the induction hypothesis $S_{n-1}(b^{\ast
	}(n-1))=0$.\ Also, $\Delta_{2}(d_{n-1}\otimes d_{n})=0$ and we thus get
	$S_{n}(b^{\ast})=0$ as desired. The arguments are similar for the three
	remaining cases : $\ell\in d_{n}$ and $\ell\notin d_{n-1}$, $\ell\notin d_{n}$
	and $\ell\in d_{n-1},$ $\ell\in d_{n}$ and $\ell\in d_{n-1}$.
\end{proof}

Combining \Cref{def_charge}, 
the combinatorial description of the duality $*$, 
\Cref{lem_energy_1} and \Cref{prop_energy_3}, we get the expected following theorem.

\begin{thm}\label{duality_intertwines_charge_energy}
For all $b\in F(s)$, we have
$\ch(b)=\D(b^*).$
\end{thm}

\section{Affine type $A$ duality}\label{duality_affine}

The study of the combinatorics of highest weight representations for the quantum groups of affine type $A$
has been initiated in \cite{JMMO1991} in the context of solvable lattice models.
There, the affine Fock space plays a crucial role. It has been subsequently considered in various settings, ranging
from mathematical physics to symmetric functions or finite group representation theory, see \cite{Leclerc2008}
for a good review.
In this section, we examine the crystal combinatorics of the affine Fock space, 
and show that the results of \Cref{duality} generalise to this setting.

\subsection{Infinite columns}

In this affine setting, columns are semi-infinite, i.e. of the form
$$c=\{x_1,\ldots,x_k\}\sqcup \{ x\in\Z\mid x<x_1\} \text{\quad and represented as } c=\scriptsize\gyoung(<x_1>,|\sesqui\vdts,<x_k>),$$
for some $k\in\Z_{\geq1}$ and some $x_1<\cdots <x_k\in \Z$.
Note that this expression is not unique.
For instance, the values $k=2$, $x_1=1, x_2=3$ and $k=3$, $x_1=0,x_2=1,x_3=3$ both yield the semi-infinite column
$\{\ldots,-2, -1,0,1,3\}$.
If we impose that $x_1$ is minimal in $\Z$ with the condition that $x_1+1\notin c$, then this expression becomes unique:
we call it the \textit{standard form} of $c$.
Finally, we set
$$s(c)=x_1+k-1.$$

\subsection{Affine combinatorial Fock space}\label{affine_setting}

\Yvcentermath1

Following \cite[Chapter 6]{GeckJacon2011}, highest weight crystals of type $A^{(1)}_{n-1}$ and level $\ell$ can be realised by using the combinatorics of symbols (or equivalently of abaci).
By analogy with type $A$, we identify symbols with tensor products of columns
(and abaci with semi-infinite binary matrices as in \Cref{binary_mat}).

\begin{defi} Let $\bs=(s_1,\ldots,s_\ell)\in\Z^\ell$. The \textit{affine combinatorial Fock space} of type $A^{(1)}_{n-1}$ and level $\ell$ associated to $\bs$ is
$$\widehat{F}(\bs)=\{b=c_\ell\otimes \cdots\otimes c_1 \mid s(c_j)=s_j \text{ for all } j=1,\ldots, \ell \}.$$
For all $b\in\widehat{F}(\bs)$, the tuple $\bs$ is called the \textit{shape} of $b$\footnote{
The element $\bs$ is usually called the (multi)charge in the literature, see \cite{JMMO1991}.
We choose this alternative terminology for two reasons : by analogy with the finite case (\Cref{columns_finite}),
since $\bs$ plays a similar role to the partition $\la$ (in fact, $\bs$ can itself be represented by a partition
as in the original approach \cite[Definition 3.1]{JMMO1991});
and because ``charge'' has already been introduced with a different meaning in \Cref{charge_energy}.
}.
\end{defi}

We say that $b\in \widehat{F}(\bs)$ is in \textit{standard form}
if all columns have the same smallest entry and (at least) one of them is in standard form.
For convenience, we represent negative entries $-x$ by $\overline{x}$.

\Yvcentermath0

\begin{exa}
The representation $$\scriptsize\young(<\overline{1}>,0,1)\otimes\young(<\overline{1}>,2)\otimes\young(<\overline{1}>,0,3)$$
is the standard form of the element $b=c_3\otimes c_2\otimes c_1$ where
$c_3=\{\ldots, -2,-1,0,1\}$, $c_2=\{\ldots, -2,-1,2\}$, $c_1=\{\ldots, -2,-1,0,3\}$.
\end{exa}

\Yvcentermath1

\medskip

We now define affine analogues of tableaux in a natural way, by
requesting semistandardness on a cylinder. 
We use the following notation: if $c=\scriptsize\gyoung(<x_1>,|\sesqui\vdts,<x_k>)$ is a column, set 
$c^+={\Yboxdimx{38pt}\scriptsize\gyoung(<x_1>,|\sesqui\vdts,<x_1+n-1>,<x_1+n>,|\sesqui\vdts,<x_k+n>)}.$

\begin{defi}\label{def_cyl}
An element $b=c_\ell\otimes \cdots\otimes c_1\in \widehat{F}(\bs)$ is called 
\begin{enumerate}
\item  \textit{cylindric} if the top-aligned juxtaposition
$c_\ell^+ c_1\cdots c_\ell$ is a semistandard Young tableau (with entries in $\Z$).
\item \textit{anticylindric} if the bottom-aligned juxtaposition
$c_1\cdots c_\ell c_1^+ $ is a semistandard skew Young tableau (with entries in $\Z$).
\end{enumerate}
\end{defi}

\begin{exa} Take $\ell=3$, $n=2$.
\begin{enumerate}
\item The element $$\scriptsize \young(<\overline{1}>,1)\otimes\young(<\overline{1}>,1,4)\otimes \young(<\overline{1}>,0,2)$$ is cylindric because
$$\scriptsize\young(<\overline{1}><\overline{1}><\overline{1}><\overline{1}>,0011,124,3)$$
is a semistandard Young tableau.
\item The element $$\scriptsize \young(<\overline{2}>,<\overline{1}>,0,1,3)\otimes\young(<\overline{2}>,0,1,3)\otimes \young(<\overline{2}>,<\overline{1}>,0,2)$$ is anticylindric because
$$\scriptsize\gyoung(:~:~:~<\overline{2}>,:~:~<\overline{2}><\overline{1}>,<\overline{2}><\overline{2}><\overline{1}>0,<\overline{1}>001,0112,2334)$$
is a semistandard skew Young tableau.
\end{enumerate}
\end{exa}

\begin{rem}
Note that the elements of $\widehat{F}(\bs)$ can alternatively be described as \textit{charged multipartitions}:
they form the standard basis of the (algebraic) Fock space, 
see \cite[Chapter 6]{GeckJacon2011}, \cite{Uglov1999}.
More precisely, tensor products of columns are the representation of charged multipartitions
by their $\be$-sets, also known as Lusztig's symbols \cite[Chapter 5]{GeckJacon2011}.
In turn, tensor products that are cylindric correspond to \textit{cylindric multipartitions},
which are important objects in algebraic combinatorics.
They were first introduced by Gessel and Krattenthaler \cite{GesselKrattenthaler1997},
and have since then been used in combinatorics, representation theory, and mathematical physics,
see \cite{FodaWelsh2015} and \cite{Gerber2018} for more details.
\end{rem}

Further, we give an affine analogue of the Yamanouchi property.
For this, given an element $b=c_\ell\otimes \cdots\otimes c_1\in \widehat{F}(\bs)$, we define
the first $n$-period of $b$ to be the sequence $P_1=(x,x-1,\ldots, x-n+1)$ of entries of $b$
such that $x$ is maximal in $b$ and such that for all $i=0,\ldots, n-1$, $x-i\in c_{k_i}$ where
$k_0\leq\cdots\leq k_{n-1}$ with $k_i$ maximal.
One then defines the $r$-th  $n$-period of $b$ by induction, by setting $P_r$ to be the first period of
the element $b^{(r-1)}$ obtained from $b$ by removing $P_1, \ldots, P_{r-1}$ if they all exist.

\begin{defi}
An element $b\in \widehat{F}(\bs)$ is called \textit{totally $n$-periodic}
if the $r$-th $n$-period of $b$ exists for all $r\in\Z_{\geq 1}$.
\end{defi}

\definecolor{myyellow}{RGB}{252,255,141}
\newcommand{\myyellow}{\Yfillcolour{myyellow}}
\definecolor{myred}{RGB}{255,183,183}
\newcommand{\myred}{\Yfillcolour{myred}}
\definecolor{mygreen}{RGB}{179,255,195}
\newcommand{\mygreen}{\Yfillcolour{mygreen}}
\definecolor{mypurple}{RGB}{234,183,255}
\newcommand{\mypurple}{\Yfillcolour{mypurple}}
\definecolor{myblue}{RGB}{188,222,255}
\newcommand{\myblue}{\Yfillcolour{myblue}}

\Yvcentermath0

\begin{exa}
Take $n=3$ and $\ell=4$. Then the following element is totally periodic
$$\scriptsize 
\gyoung(!\myyellow0,!\myyellow1,!\myblue2,!\myblue3)
\otimes
\gyoung(0,1)
\otimes 
\gyoung(0,!\mypurple2,!\myblue4)
\otimes
\gyoung(0,!\myyellow2,!\mypurple3,!\mypurple4),$$
where we have highlighted the first period in blue, the second in purple and the third in yellow
(and the periods $P_r$ for $r>3$ then obviously exist).
\end{exa}

\subsection{Affine crystal structures}\label{affine_crystal}

In the following, 
for $\bs=(s_\ell,\ldots, s_1)\in\Z^\ell$ and $\dbs=(\ds_1,\ldots, \ds_n)\in\Z^n$, denote 
$$b_\bs= {\scriptsize\young(<s_\ell>)\otimes\cdots\otimes\young(<s_1>)}
\mand
\dot{b}_{\dbs}= {\scriptsize\young(<\ds_1>)\otimes\cdots\otimes\young(<\ds_n>)}.$$
Moreover, define analogues of the sets $\sS(s)$ and $\dot{\sS}(s)$ of \Cref{duality} by considering
$$\sD(s)=\left\{ (s_\ell,\ldots, s_1)\in\Z^\ell(s) \mid  s_\ell\leq\cdots\leq s_1\leq s_\ell+n\right\}
\text{ and }
\dot{\sD}(s)=\left\{ (\ds_1,\ldots, \ds_n)\in\Z^\ell(s) \mid  \ds_1\leq\cdots\leq \ds_n\leq \ds_1+\ell\right\},$$
and set 
$$\sD=\bigsqcup_{s\in \Z}\sD(s)
\mand
\dot{\sD}=\bigsqcup_{s\in \Z}\dot{\sD}(s).
$$
As in \Cref{crystal_finite}, the combinatorial Fock space $\widehat{F}(\bs)$ can be endowed with the structure of an $A^{(1)}_{n-1}$-crystal, 
where the action of the lowering crystal operators $f_i$ is given by changing 
a certain entry $x$ such that $x \mod n=i$ to $x+1$, see for instance \cite[Chapter 6]{GeckJacon2011}.
Also, the connected component containing $b_\bs$ is isomorphic to the crystal  $\hB(\bs)$ of the 
irreducible highest weight module of highest weight $\om_{\bs}=\om_{s_\ell}+\cdots+\om_{s_1}$ where $\om_{0},\ldots,\om_{n-1}$ are the fundamental weights of type $A_{n-1}^{(1)}$. Let us recall the crystal structure on $\widehat{F}(\bs)$, due to \cite{JMMO1991}, \cite{FLOTW1999}, \cite{Uglov1999}.
We construct the word $\textrm{w}_i(b)$ as in \Cref{crystal_finite}, by the following procedure:
\begin{enumerate}
\item For all $k\in\Z$,
let $\textrm{w}_i^{(k)}(b)$ be the word obtained by reading the entries $i+kn$ and $i+kn+1$ in $b$
from left to right (i.e., going through the columns $c_\ell,\ldots,c_1$).
\item Set $\textrm{w}_i(b)$ to be the concatenation of the words $\textrm{w}_i^{(k)}$ ordered by decreasing
values of $k$.
\end{enumerate}

\begin{rem}
Because of the semi-infinite form of each column, it is enough to consider only the integers $k$ such that 
$i+kn$ appears in the standard form of $b$. That way, we construct $\textrm{w}_i(b)$ from finitely many subwords $\textrm{w}_i^{(k)}$.
\end{rem}

\begin{exa}
For $n=4$, $i=3$, $\ell=3$ and 
$b=\scriptsize\young(1,<\boldsymbol{3}>)\otimes\young(1,<\boldsymbol{3}>,<\boldsymbol{4}>,<\boldsymbol{7}>)
\otimes\young(1,2,<\boldsymbol{4}>,6,<\boldsymbol{7}>)=c_3\otimes c_2\otimes c_1$, where we 
highlighted the entries with relevant residues with boldface, we have
$\textrm{w}_3^{(1)}(b)=77$
and 
$\textrm{w}_3^{(0)}(b)=3344,$
therefore $\textrm{w}_3(b)=773344$.
\end{exa}
As in type $A$, this induces a word in the symbols $+$ and $-$ by encoding each 
$i+kn$ by $+$ if and each $i+kn+1$ by $-$.
Deleting all factors $+-$ recursively yields a word called the \textit{$i$-signature} of $b$.
With these conventions, one checks that 
the following result is equivalent to \cite[Theorem 6.2.12]{GeckJacon2011}.

\begin{thm}\label{fock_crystal} 
The set $\widehat{F}(\bs)$ is endowed with an $A^{(1)}_{n-1}$-crystal structure, where
\begin{enumerate}
\item The lowering crystal operators $f_i$ act on $b\in \widehat{F}(\bs)$ 
by changing the entry $x$ corresponding to the leftmost $+$ in the 
$i$-signature of $b$ into $x+1$ if it exists; and by $0$ otherwise.
\item The raising crystal operators $f_i$ act on $b\in \widehat{F}(\bs)$ 
by changing the entry $x+1$ corresponding to the rightmost $-$ in the 
$i$-signature of $b$ into $x$ if it exists; and by $0$ otherwise.
\end{enumerate}
\end{thm}

\begin{exa}
In the previous example, we first get the sequence $++++--$,
and the signature is $++$.
The leftmost $+$ corresponds to the $7$ in $c_2$, and therefore $f_3$
acts on $b$ by changing that $7$ into an $8$, i.e.
$$f_3 b=\scriptsize\young(1,3)\otimes\young(3,4,5,<\boldsymbol{8}>)
\otimes\young(1,4,6,7).$$
\end{exa}

\begin{exa}
Take $\ell=3$, $n=2$, $\bs=(s_3,s_2,s_1)=(7,4,6)$ and $b=c_3\otimes c_2\otimes c_1={\scriptsize\young(2,3,4,5,6,8)\otimes\young(2,4,5)\otimes\young(2,3,5,6,7)}$.
One checks that $e_0b=e_1b=e_2b=0$, i.e. $b$ is a highest weight vertex, and that the beginning of the connected component 
of the $A^{(1)}_{n-1}$-crystal containing $b$ is equal to
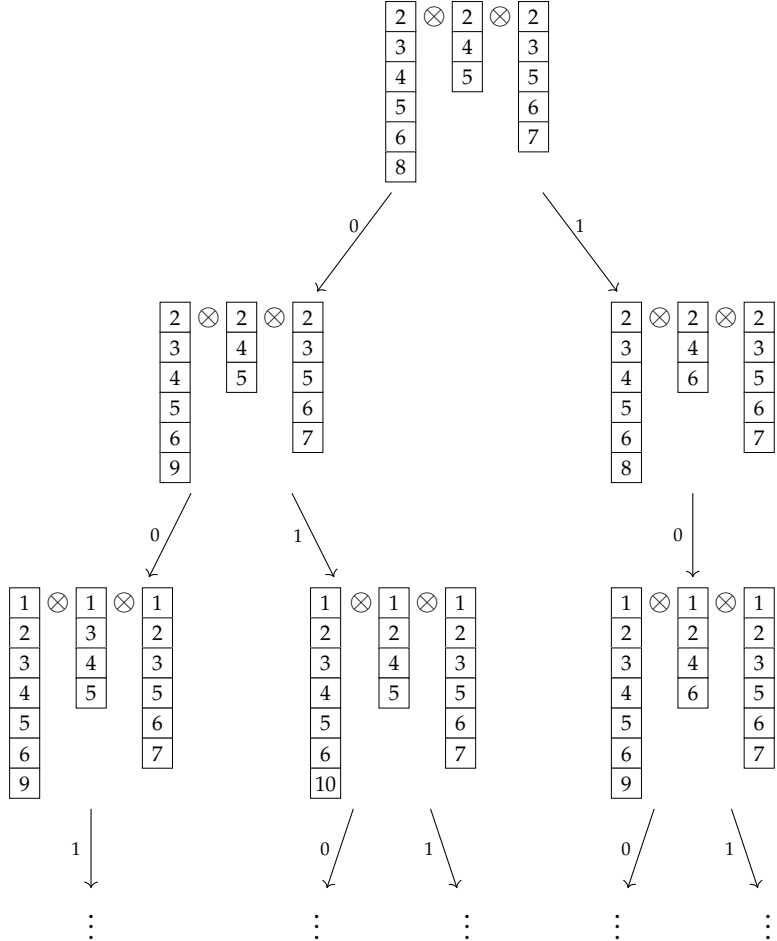
\begin{center}
\begin{tikzpicture}
\node (a) at (0,0) {$\scriptsize\young(2,3,4,5,6,8)\otimes\young(2,4,5)\otimes\young(2,3,5,6,7)$};
\node (b1) at (-3,-4) {$\scriptsize\young(2,3,4,5,6,9)\otimes\young(2,4,5)\otimes\young(2,3,5,6,7)$};
\node (b2) at (3,-4) {$\scriptsize\young(2,3,4,5,6,8)\otimes\young(2,4,6)\otimes\young(2,3,5,6,7)$};
\node (c1) at (-5,-8) {$\scriptsize\young(1,2,3,4,5,6,9)\otimes\young(1,3,4,5)\otimes\young(1,2,3,5,6,7)$};
\node (c2) at (-1,-8) {$\scriptsize\young(1,2,3,4,5,6,<10>)\otimes\young(1,2,4,5)\otimes\young(1,2,3,5,6,7)$};
\node (c3) at (3,-8) {$\scriptsize\young(1,2,3,4,5,6,9)\otimes\young(1,2,4,6)\otimes\young(1,2,3,5,6,7)$};

\node (d1) at (-5,-11) {$\vspace{3mm}\vdots$};
\node (d2) at (-2,-11) {$\vspace{3mm}\vdots$};
\node (d3) at (0,-11) {$\vspace{3mm}\vdots$};
\node (d4) at (2,-11) {$\vspace{3mm}\vdots$};
\node (d5) at (4,-11) {$\vspace{3mm}\vdots$};

\draw[->] (a) --  node[pos=0.5,above]{\tiny 0} (b1);
\draw[->] (a) --  node[pos=0.5,above]{\tiny 1} (b2);
\draw[->] (b1) --  node[pos=0.5,left]{\tiny 0} (c1);
\draw[->] (b1) --  node[pos=0.5,left]{\tiny 1} (c2);
\draw[->] (b2) --  node[pos=0.5,left]{\tiny 0} (c3);
\draw[->] (c1) --  node[pos=0.5,left]{\tiny 1} (d1);
\draw[->] (c2) --  node[pos=0.5,left]{\tiny 0} (d2);
\draw[->] (c2) --  node[pos=0.5,left]{\tiny 1} (d3);
\draw[->] (c3) --  node[pos=0.5,left]{\tiny 0} (d4);
\draw[->] (c3) --  node[pos=0.5,left]{\tiny 1} (d5);
\end{tikzpicture}\end{center}
\end{exa}

\begin{rem}
If $n$ is sufficiently large, any column (in standard form) has at most one entry with 
residue $i$ or $i+1$, and therefore one recovers the finite type $A$ crystal rule (forgetting the arrows indexed by $0$).
\end{rem}

In fact, we have
$$\widehat{F}(\bs)\simeq \widehat{F}((s_\ell))\otimes \cdots\otimes  \widehat{F}((s_1)).$$

\begin{rem}
Note that the rule used to compute the affine crystal above is not the tensor product rule,
which means that the above isomorphism is not an equality in general.
However, we recover the tensor product rule (and hence the equality) if
the differences $s_{j+1}-s_j$ are sufficiently large.
\end{rem}

\newcommand{\dotw}{\dot{\textrm{w}}}

Similarly, for any $n$-tuple of integers $\dbs=(\dot{s}_1,\ldots,\dot{s}_{n})$, there is a level $n$ $A^{(1)}_{\ell-1}$-crystal structure on the set of tensor products of $n$ columns 
$$\dot{\widehat{F}}(\dbs)=\{a=d_1\otimes \cdots\otimes d_n \mid s(d_i)=\ds_i \text{ for all } i\}.$$
In order for these crystals structures to commute (see \Cref{tricrystal}), we need to use a slightly different rule for computing the $A^{(1)}_{\ell-1}$-crystal.
More precisely, we construct a word $\dotw_j(a)$ by the following procedure:
\begin{enumerate}
\item For all $k\in\Z$,
let $\dotw_j^{(k)}(a)$ be the word obtained by reading the entries $j+k\ell$ and $j+k\ell+1$ in $a$
from left to right  (i.e., going through the columns $d_1,\ldots, d_n$).
\item Set $\dotw_j(a)$ to be the concatenation of the words $\dotw_j^{(k)}$ ordered by increasing
values of $k$.
\end{enumerate}
Now, we compute the $j$-signature by the exact same procedure as for the $A^{(1)}_{n-1}$-crystal, and the action of 
the crystal operators $\df_j$ and $\de_j$ is given by the same rule as in \Cref{fock_crystal}.
We obtain again
$$\dot{\widehat{F}}(\bs)\simeq \dot{\widehat{F}}(({\ds_1}))\otimes \cdots\otimes\dot{\widehat{F}}(({\ds_n})).$$

\begin{exa}
For $\ell=3$, $j=0$, $n=4$ and  
$a=\scriptsize
\young(<\overline{1}>,<\boldsymbol{0}>,<\boldsymbol{1}>,2,<\boldsymbol{3}>)
\otimes\young(<\overline{1}>,<\boldsymbol{0}>,<\boldsymbol{1}>,<\boldsymbol{4}>)
\otimes\young(<\overline{1}>,<\boldsymbol{0}>,2,<\boldsymbol{3}>,<\boldsymbol{4}>,5)
\otimes\young(<\overline{1}>,<\boldsymbol{0}>,<\boldsymbol{1}>,2)
=d_1\otimes d_2\otimes d_3\otimes d_4$, where we 
highlighted the entries with relevant residues with boldface, we have
$\dotw_0^{(0)}(a)=0101001$ and $\dotw_0^{(1)}(a)=3434$,
therefore $\dotw_0(a)=01010013434$.
This yields the sequence $+-+-++-+-+-$, so the $0$-signature is $+$,
in which the leftmost $+$ corresponds to the $0$ in $d_3$.
Therefore, $\df_0$ acts on $a$ by changing that $0$ into a $1$, i.e.
$$\df_0 a=
\scriptsize\young(<\overline{1}>,0,1,2,3)
\otimes\young(<\overline{1}>,0,1,4)
\otimes\young(<\overline{1}>,<\boldsymbol{1}>,2,3,4,5)
\otimes\young(<\overline{1}>,0,1,2).$$
In other terms, in the crystal $\widehat{F}(s)$, we have an arrow
\begin{center}
\begin{tikzpicture}
\node (a) at (0,0) 
{$\scriptsize\young(<\overline{1}>,<{0}>,<{1}>,2,<{3}>)
\otimes\young(<\overline{1}>,<{0}>,<{1}>,<{4}>)
\otimes\young(<\overline{1}>,<{0}>,2,<{3}>,<{4}>,5)
\otimes\young(<\overline{1}>,<{0}>,<{1}>,2)$};
\node (b) at (4.5,0) 
{$\scriptsize\young(<\overline{1}>,0,1,2,3)
\otimes\young(<\overline{1}>,0,1,4)
\otimes\young(<\overline{1}>,<{1}>,2,3,4,5)
\otimes\young(<\overline{1}>,0,1,2).$};
\draw[->] (a) --  node[pos=0.5,above]{\tiny 0} (b);
\end{tikzpicture}\end{center}
\end{exa}

Let us set $\widehat{F}(s)=\bigoplus_{\bs\in\Z^\ell(s)}\widehat{F}(\bs)$ and $\dot{\widehat{F}}(s)=\bigoplus_{\dbs\in\Z^n(s)}\dot{\widehat{F}}(\dbs)$.
We can now state an analogue of \Cref{thm_yam_tab}.

\begin{thm}\label{thm_cylindric_totperiodic}\
\begin{enumerate}
\item The set of cylindric elements of $\widehat{F}(\bs)$ is stable under the crystal operators $f_i$, $e_i$, $i=0,\ldots, n-1$.
Moreover, for any $b\in \widehat{F}(\bs)$, there is a cylindric element $\sP(b)\in \widehat{F}(\bs)$
such that $\sP(b)$ and $b$ appear at the same place in two isomorphic connected components of $\widehat{F}(\bs)$.
\item An element $b\in \widehat{F}(s)$ is a highest weight vertex in the $A^{(1)}_{n-1}$-crystal if and only if $b$ is totally periodic.
\end{enumerate}
\end{thm}

\begin{proof}
One checks that \Cref{def_cyl} is equivalent to the definition of cylindricity used in \cite{Tingley2008} and \cite{Gerber2015}.
Therefore, the first statement of (1) translates to \cite[Section 3]{Tingley2008}, see also \cite[Proposition 4.10]{Gerber2015}.
A constructive proof of the second statement of (1) can be derived from \cite[Section 4.3]{Gerber2015}.
Part (2) was obtained in \cite[Theorem 5.9]{JaconLecouvey2012}.
\end{proof}

\begin{rem}
Note that $\sP(b)$ is not unique, unlike the tableau $P(b)$ of \Cref{thm_yam_tab}.
There is uniqueness by putting some extra constraints on $\sP(b)$, see \cite{Gerber2015} for more details.
\end{rem}

\subsection{Uglov's duality}\label{affine duality}

There is a duality
$$
\begin{array}{rcl}
\widehat{F}(s)
&
\longleftrightarrow
&
\dot{\widehat{F}}(s)
\\
b
&
\longleftrightarrow
&
b^\star
\end{array}
$$
defined in \cite[Remark 4.2]{Uglov1999}, see also \cite[Section 1.1.5]{Yvonne2005}.
It is best explained in terms of binary matrices.
We first encode $b=c_\ell\otimes\cdots\otimes c_1\in \widehat{F}(s)$ by the $\infty\times \ell$ 
matrix $M$ defined by 
$$M_{i,j}=\left\{ 
\begin{array}{ll}
1 & \text{ if } i\in c_j \\
0 & \text{ otherwise } \\ 
\end{array}
\right.
$$
Now, for all $k\in\Z$, consider the submatrices $M^{(k)}$
of size $n\times\ell$ of $M$ defined by, for all $1\leq i\leq n$,
$$M^{(k)}_{i,j}=M_{(k-1)n+i,j}.$$
Set $N^{(k)}={M^{(k)}}^\trans$. Then the $\infty\times n$ matrix $N$ 
defined by, for all $j\in\Z$, say $(k-1)\ell+1\leq j\leq k\ell$,
$$N_{j,i}=N^{(k)}_{j-(k-1)\ell,i}$$
decodes to an element $b^\star$ of $\dot{\widehat{F}}(s)$.

\begin{exa}
Take $\ell=2, n=3$, and 
$b=\scriptsize
\gyoung(<\overline{2}>,0,1,3)
\otimes
\gyoung(<\overline{2}>,<\overline{1}>,2)
$.
Then we have
$$
M=
\begin{array}{ccl}
\vdots&\vdots&
\\
\begin{array}{c} 
\text{\scriptsize $-5$}\\ \text{\scriptsize $-4$}\\ \text{\scriptsize $-3$}
\end{array}
&
\begin{bmatrix}
1&1\\
1&1\\
1&1
\end{bmatrix}
&
M^{(-1)}
\\
\begin{array}{c}
\text{\scriptsize $-2$}\\ \text{\scriptsize $-1$}\\ \text{\scriptsize $0$}
\end{array}
&
\begin{bmatrix}
1&1\\
1&0\\
0&1
\end{bmatrix}
&
M^{(0)}
\\
\begin{array}{c}
\text{\scriptsize $1$}\\ \text{\scriptsize $2$}\\ \text{\scriptsize $3$}
\end{array}
&
\begin{bmatrix}
0&1\\
1&0\\
0&1
\end{bmatrix}
&
M^{(1)}
\\
\begin{array}{c}
\text{\scriptsize $4$}\\ \text{\scriptsize $5$}\\ \text{\scriptsize $6$}
\end{array}
&
\begin{bmatrix}
0&0\\
0&0\\
0&0
\end{bmatrix}
&
M^{(2)}
\\
\vdots&\vdots&
\\
&
\begin{array}{cc}
\text{\scriptsize $1$}&\text{\scriptsize $2$}
\end{array}
&
\end{array}
\text{\quad and we get \quad }
N=
\begin{array}{ccl}
\vdots&\vdots&
\\
\begin{array}{c}
\text{\scriptsize $-3$}\\ \text{\scriptsize $-2$}
\end{array}
&
\begin{bmatrix}
1&1&1\\
1&1&1
\end{bmatrix}
&
{M^{(-1)}}^\trans=N^{(-1)}
\\
\begin{array}{c}
\text{\scriptsize $-1$}\\ \text{\scriptsize $0$}
\end{array}
&
\begin{bmatrix}
1&1&0\\
1&0&1
\end{bmatrix}
&
{M^{(0)}}^\trans =N^{(0)}
\\
\begin{array}{c}
\text{\scriptsize $1$}\\ \text{\scriptsize $2$}
\end{array}
&
\begin{bmatrix}
0&1&0\\
1&0&1
\end{bmatrix}
&
{M^{(1)}}^\trans=N^{(1)}
\\
\begin{array}{c}
\text{\scriptsize $3$}\\ \text{\scriptsize $4$}
\end{array}
&
\begin{bmatrix}
0&0&0\\
0&0&0
\end{bmatrix}
&
{M^{(2)}}^\trans=N^{(2)}
\\
\vdots&\vdots&
\\
&
\begin{array}{ccc}
\text{\scriptsize $1$}&\text{\scriptsize $2$}&\text{\scriptsize $3$}
\end{array}
&
\end{array},
$$
where we indicated the rows and column indices. This gives
$b^\star=\scriptsize
\gyoung(<\overline{2}>,<\overline{1}>,0,2)
\otimes
\gyoung(<\overline{2}>,<\overline{1}>,1)
\otimes
\gyoung(<\overline{2}>,0,2)$.
\end{exa}

\begin{rem}\label{analogy_duality}
Assume that each column $c=\{x_1,\ldots,x_k\}$ of $b$
written in standard form verifies $0\leq x_1<\cdots< x_k\leq n$.
\begin{enumerate}
\item Then $b$ is determined by $M^{(1)}$,
and the duality $\star$ is just transposing this single matrix.
We recover precisely the duality $*$, see \Cref{binary_mat}.
For instance, for 
$\ell=4$, $n=3$ and
$b=\scriptsize\gyoung(0,2,3)\otimes\gyoung(0,3)\otimes\gyoung(0,1)\otimes\gyoung(0,1,3).$
The corresponding matrix is
$$
\begin{bmatrix}
1&1&0&0 \\
0&0&0&1 \\
1&0&1&1
\end{bmatrix},
\text{\quad whose transpose is \quad}
\begin{bmatrix}
1&0&1 \\
1&0&0 \\
0&0&1 \\
0&1&1
\end{bmatrix},
$$
so we get
$b^\star=
\scriptsize
\gyoung(0,1,2)\otimes\gyoung(0,4)\otimes\gyoung(0,1,3,4).$
\item Moreover, the orders $\prec$ and $\precdot$ used to compute the crystals coincide, 
in analogy with the finite case again.
\end{enumerate}
\end{rem}

The first result concerning the duality $\star$ is an analogue of 
\Cref{duality_intertwines_Yam_tab}.
The following result can be checked purely combinatorially. We omit its proof because it is essentially equivalent to \cite[Theorem 3.3]{Gerber2018} in the convention of the present paper.

\begin{prop}\label{totper_anticylindric}
Let $b\in \widehat{F}(\bs)$. Then $b$ is totally periodic if and only if $b^\star$ is anticylindric.
\end{prop}

\begin{rem} Similarly, there is a notion of antiperiods that yield a characterisation of 
highest weight vertices in $\dot{\widehat{F}}(\dbs)$.
We get that $b$ is cylindric if and only if $b^\star$ is totally antiperiodic, completing the previous proposition
\end{rem}

Now, we would like to establish an analogue of \Cref{kas-jdt_weyl-Rmat}.
Let $\widehat{S_n}$ be the affine symmetric group on $n$ elements, which is the Weyl group
of type $A^{(1)}_{n-1}$. We denote by $\si_i$, $i=0,\ldots, n-1$ the involutions generating $\widehat{S_n}$,
subject to the usual braid relations modulo $n$.
The group $\widehat{S_n}$ acts on the crystal $\widehat{F}(s)$ as in finite type.
Similarly, we denote $\dot{\si_j}$, $j=0,\ldots, \ell-1$ the generators of $\widehat{S_\ell}$.
First of all, we focus on the maps $\de_j$ and $\dot{\si_j}$ but only in the case $j=1,\ldots, \ell-1$.
The case $j=0$ will be treated in upcoming \Cref{affine_cyclage_promotion} using the promotion operator.
As in \Cref{duality_def}, for $j=1,\ldots, \ell-1$, we denote 
by $J_j$ the elementary Jeu de Taquin map on $\widehat{F}(s)$ between columns $j$ and $j+1$, and
by $R_{j}$ the combinatorial $R$-matrix on $\widehat{F}(s)$ realising the isomorphism
of crystals $$F((s_1,\ldots,s_j,s_{j+1},\ldots,s_\ell)) \overset{\sim}{\lra} F((s_1,\ldots,s_{j+1},s_j,\ldots,s_\ell)).$$

\begin{thm}\label{affine_kas-jdt_weyl-Rmat}For all $j=1,\ldots, \ell-1$, we have
$$\text{(1) }\quad \de^\star_j  =J_{j} \quad \mand \quad  \text{(2)} \quad  \dot{\si_j}^\star=R_{j}.$$
\end{thm}

\begin{proof}
The proof is analogous to that ot \Cref{kas-jdt_weyl-Rmat}.
In particular, (1) is immediate.
For (2), an explicit formula for $R_j$ was given in \cite[Proposition 5.2.2]{JaconLecouvey2010}.
This formula coincides with the formula for the $R$-matrix in finite type 
which can be found in \cite[Example 4.10]{Shimozono2005}.
Therefore, we conclude using \Cref{kas-jdt_weyl-Rmat}.
\end{proof}

\subsection{Affine cyclage and promotion}\label{affine_cyclage_promotion}

We now show that the results of \Cref{cyclage_promotion} generalise to the affine case.
Consider the \textit{affine cyclage} map 
$$
\begin{array}{cccc}
\widehat{\xi} : & \widehat{F}(s) & \lra & F(s+n)
\\
& c_\ell\otimes\cdots\otimes c_1 & \longmapsto & c_{\ell-1}\otimes \cdots \otimes c_1 \otimes c_\ell^+
\end{array}
$$
where $c_\ell^+$ at the beginning of \ref{affine_setting}. By \cite[Proposition 5.2.1]{JaconLecouvey2010}, $\widehat{\xi}$ 
is an isomorphism of $A^{(1)}_{n-1}$-crystals.
Now, consider the promotion map
$$
\begin{array}{cccc}
\pr : & \dot{\widehat{F}}(s) & \lra & \dot{\widehat{F}}(s+n)
\\
& d_1\otimes\cdots\otimes d_n & \longmapsto & \pr(d_1)\otimes \cdots \otimes  \pr(d_n)
\end{array}
$$
where $\pr(d_i)=\left\{ k+1 \,;\, k\in d_i\right\}$ for all $i=1,\ldots,n$.
Exactly as in \Cref{duality_intertwines_cyc_prom}, we have the following result,
which is immediate by definition of $\star$.
\begin{prop}\label{affine_duality_intertwines_cyc_prom}
For all $b\in \widehat{F}(s)$, $$\widehat{\xi}(b)^\star = \pr(b^\star).$$
\end{prop}

As a direct consequence, we get a description of the action of the maps $\de_0$
and $\dot{\si_0}$ directly on $\widehat{F}(s)$. This completes the statement of \Cref{affine_kas-jdt_weyl-Rmat}.

\begin{cor}\label{kas_weyl_0}
$$\text{(1) }\quad \de_0^\star=\widehat{\xi}^{-1}\circ J_1 \circ \widehat{\xi} 
\quad \mand \quad 
\text{(2)} \quad \dot{\si_0}^\star=\widehat{\xi}^{-1}\circ R_1\circ \widehat{\xi}.$$
\end{cor}

By analogy with \Cref{affine_kas-jdt_weyl-Rmat}, we denote $R_0=\dot{\si}_0^\star$.
If $c_j$ is a column, denote by $c_j^-$ the column such that $(c_j^-)^+=c_j$.
Since $\widehat{\xi}^{-1}(c_\ell\otimes\cdots\otimes c_1)=c_1^{-1}\otimes c_\ell\otimes \cdots\otimes c_2$,
we have
$$R_0(c_\ell\otimes \cdots\otimes c_1)=c_1^-\otimes c_{\ell-1}\otimes\cdots\otimes c_2\otimes c_\ell^+.$$

\subsection{Triple crystal structure}\label{tricrystal}

The affine duality $\star$ is more complex than its finite type counterpart $*$, and this is related to
the existence of a third crystal structure on $\widehat{F}(s)$ \cite{Gerber2016}. 
This arises from the action of a Heisenberg algebra $\cH$,
and the resulting crystal is of type $A_\infty$.
More precisely, each $b\in\widehat{F}(s)$ writes uniquely on the form  $b=a_{\ka}(\overline{b})$ for some partition $\ka$ 
and some highest weight vertex $\overline{b}$ for the $\cH$-structure. Here  $a_{\ka}$ is the so-called \textit{Heisenberg crystal operator} corresponding to the partition $\ka$.
The explicit action of $a_\ka$ is described in \cite{Gerber2017} in terms of translating (generalised) $n$-periods.
By writing $\ka(b)=\ka$, the map $b \mapsto \ka(b)$
yields a bijection between the connected component containing $b$ and the Young lattice $\cY$. 
Now, the Young lattice carries a crystal structure corresponding to the the basic representation of type $A_\infty$. 
More precisely, there is an arrow $\la\overset{k}{\lra}\mu$ in $\cY$ with $k\in \Z$
if and only if $\mu / \la$ has only one box of content $k$. 
This endows $\widehat{F}(s)$ with the structure of an $A_\infty$-crystal.
The following result is \cite[Theorem 6.17]{Gerber2016}.

\begin{thm}\label{tricrystal_commute} 
The three crystals commute\footnote{
Note that in \cite{Gerber2016}, the commutation was proved 
for a twisted version of this duality. 
This is accounted for here by enumerating columns in the reverse order
in $\dot{\widehat{F}}(s)$, and using the different ordering. 
}, i.e. for all $i=1,\ldots,n-1$, for all $j=1,\ldots,\ell-1$ and for all partition $\ka$, we have
$$ f_i \df_j^\star =\df_j^\star f_i, \quad a_\ka f_i = f_i a_\ka, \quad \text{ and } \quad \df_j^\star a_\ka = a_\ka \df_j^\star.$$
\end{thm}

Recall that we have introduced the sets $\sD(s), \dot{\sD}(s)$ and 
the notation $b_\bs, \dot{b}_{\dot{\bs}}$ for $\ba\in\sD(s),\dbs\in\dot{\sD}(s)$ in \Cref{affine_crystal}.
Write respectively $\dot{\om}_{0},\ldots,\dot{\om}_{\ell-1}$ and $\dot{\delta}$ the fundamental weights and the null root for the root system of type $A_{\ell-1}^{(1)}$. The definition of $\star$ implies the following important property.

\begin{lem}\label{duality_empty} Let $b\in \widehat{F}(s)$. We have
$b=b_\bs$ for some $\bs\in\sD(s)$ if and only if $b^\star=\dot{b}_{\bs^\star}$ for some $\bs^\star\in\dot{\sD}(s)$. 
Moreover, the $A_{\ell-1}^{(1)}$-dominant weight corresponding to $\dot{b}_{\bs^\star}$ has the form
\[
(n-s_{1}+s_{\ell})\dot{\omega}_{0}+\sum_{j=1}^{\ell-1}
(s_{j}-s_{j+1})\dot{\omega}_{j}+k\dot{\delta}
\]
where $k$ is an integer.
\end{lem}

\Yvcentermath0

\begin{exa}
Take $\ell=4$, $n=3$ and $\bs=(-2,-1,-1,1)\in\sD(-3)$.
Then $$b_\bs={\scriptsize\young(<\overline{2}>)\otimes\young(<\overline{1}>)\otimes\young(<\overline{2}>)\otimes\young(<1>)}
\quad , \quad \text{i.e.} \quad 
b_\bs=
\scriptsize
\gyoung(<\overline{2}>)
\otimes
\gyoung(<\overline{2}>,<\overline{1}>)
\otimes
\gyoung(<\overline{2}>,<\overline{1}>)
\otimes
\gyoung(<\overline{2}>,<\overline{1}>,<0>,<1>).
$$
The corresponding matrix is
$$
M=
\begin{array}{cl}
\vdots&
\\
\begin{bmatrix}
1&1&1&1\\
1&1&1&0\\
1&0&0&0
\end{bmatrix}
&
M^{(0)}
\\
\begin{bmatrix}
1&0&0&0\\
0&0&0&0\\
0&0&0&0\\
\end{bmatrix}
&
M^{(1)}
\\
\vdots&
\end{array}
\text{\quad which gives \quad }
M^\star=
\begin{array}{cl}
\vdots&
\\
\begin{bmatrix}
1&1&1\\
1&1&0\\
1&1&0\\
1&0&0
\end{bmatrix}
&
{M^{(0)}}^\trans
\\
\begin{bmatrix}
1&0&0\\
0&0&0\\
0&0&0\\
0&0&0
\end{bmatrix}
&
{M^{(1)}}^\trans
\\
\vdots&
\end{array},
$$
therefore
$$b_\bs^\star=
{\scriptsize
\gyoung(<\overline{3}>,<\overline{2}>,<\overline{1}>,<0>,<1>)
\otimes
\gyoung(<\overline{3}>,<\overline{2}>,<\overline{1}>)
\otimes
\gyoung(<\overline{3}>)
}
=\dot{b}_{\bs^\star}
\quad  \text{ with } \bs^\star=(1,-1,-3)\in\dot{\sD}(-3).$$
\end{exa}

\begin{rem}
We see that the map $\bs\mapsto \bs^\star$ is an analogue of the transposition of partitions used in \Cref{duality}.
\end{rem}

The commutation of the three crystals in \Cref{tricrystal_commute} induces
an $(A^{(1)}_{n-1}\times A_\infty\times A^{(1)}_{\ell-1})$-crystal structure on $\widehat{F}(s)$,
and as in the finite case (see \Cref{sources_finite}), 
each connected component of the tricrystal of $\widehat{F}(s)$ has a unique source vertex.
The following corollary is the translation of \cite[Theorem 6.19]{Gerber2016} in our terminology.

\begin{cor}\label{sources_affine} Let $b\in \widehat{F}(s)$.
Then $b$ is a highest weight vertex in the $(A^{(1)}_{n-1}\times A_\infty\times A^{(1)}_{\ell-1})$-crystal
if and only if $b=b_\bs$ for some $\bs\in\sD(s)$.
\end{cor}

To be complete, we give a characterisation of the highest weight vertices in
the different bicrystals.
In order to do this, we consider the vertices of $\hB(\bs)$, the connected component of the $A^{(1)}_{n-1}$-crystal containing $b_\bs$, where $\bs\in\sD(s)$.
These are called \textit{$n$-FLOTW} elements, and the have an explicit combinatorial description, see \cite[Definition 5.7.8]{GeckJacon2011}\footnote{ 
Note that FLOTW elements are originally defined for a more constrained condition on $\bs$, but it is easy to see that our condition induces the same
combinatorial characterisation
}.
In particular, FLOTW elements are cylindric, which is immediate from \Cref{thm_cylindric_totperiodic} because $b_\bs$ is.

\begin{thm}\label{hwv_affine_bicrystals}\
\begin{enumerate}
\item $b$ is a highest weight vertex in the $(A^{(1)}_{n-1}\times A^{(1)}_{\ell-1})$-crystal if and only if $b$ is cylindric and totally periodic.
\item $b$ is a highest weight vertex in the $(A^{(1)}_{n-1}\times A_\infty)$-crystal if and only if $b^\star$ is $\ell$-FLOTW.
\end{enumerate}
\end{thm}

\begin{proof}
Assertion (1) follows from the characterisation of the highest weight vertices in the 
$A^{(1)}_{n-1}$-crystal and the $A^{(1)}_{\ell-1}$-crystal given in \Cref{thm_cylindric_totperiodic}(2) and \Cref{totper_anticylindric}.
Assertion (2) is proved in \cite{Gerber2016}.
\end{proof}

\begin{rem}
Note that it is more challenging to give a simple description of the $(A^{(1)}_{n-1}\times A_\infty)$-highest weight vertices. Nevertheless, a  characterisation of the $ A_\infty$-highest weight vertices has been given in \cite[Theorem 5.1]{GerberNorton2018}, 
but this does not yield to an analogue of the previous theorem.
\end{rem}

By \Cref{tricrystal_commute} and \Cref{sources_affine}, 
for each $b\in \widehat{F}(s)$, there is a unique $\bs\in\sD(s)$
such that
$b=\df^\star_{\underline{j}} a_\ka f_{\underline{i}} b_\bs$
for some $\underline{i}$, $\underline{j}$ and some partition $\ka$.
Set 
$$\cP(b)=f_{\underline{i}} b_\bs, \quad \ka(b)=\ka, \mand \cQ(b)=\df_{\underline{j}}b_\bs^\star.$$
Note that the elements $\cP(b)$ and $\cQ(b)$ are FLOTW by definition.
Therefore, we get an analogue of \Cref{rsk}, yielding an affine crystal version of the RSK correspondence.

\begin{thm}
The assignment $$\widehat{\Phi}: b \longmapsto (\cP(b),\ka(b), \cQ(b)) $$
yields a bijection $\widehat{F}(s) \to \widehat{\Phi}(\widehat{F}(s))$.
In particular, we have the decomposition 
$$\widehat{F}(s)\simeq \bigoplus_{\bs\in\sD(s)} \hB(\bs)\otimes 
\cY
\otimes \dhB(\bs^\star).$$
\end{thm}

\begin{rem}\label{rem_affine_rsk}
Another way to express the affine crystal RSK correspondence is to consider the bijection
$$b\longleftrightarrow (\sP(b), \cQ(b))$$
where  $\cQ(b)$ is defined as before, and $\sP(b)=a_\ka f_{\underline{i}}b_\bs$ (in particular $\sP(b)$ is cylindric).
In a symmetric fashion, one can establish a bijection 
$$b\longleftrightarrow (\cP(b), \sQ(b))$$ where $\sQ(b)$ is anticylindric.
It would be interesting to compare this with the results of \cite{Langer2012}.
\end{rem}

\subsection{Bicrystal structure on self-dual elements}

As in \Cref{selfdual_finite}, let us consider the set of self-dual elements $\widehat{F}(s)^\star$, that is,
the set of all $b=c_1\otimes \cdots\otimes c_\ell \in F(s)$ such that 
$b^\star=c_\ell \otimes \cdots\otimes c_1$. Again, self-dual elements exist only if $n=\ell$.
Moreover, by \cite[Proposition 5.2]{Gerber2016}, for all $b\in\widehat{F}(s)$, we have
$\ka(b^\star)=\ka(b)^\trans$, therefore if $b\in \widehat{F}(s)^\star$,
we have
$$\ka(b)=\ka(b)^\trans.$$
Now, by \cite{Kashiwara1996}, self-transpose partitions realise the crystal graph $\widetilde{\cY}$
of the basic representation of type $B_\infty$,
by setting 
$$\ka\overset{k}{\lra}\ka' \text{ in } \widetilde{\cY}
\text{ \quad  if and only if \quad } 
\ka\overset{k}{\lra}\ka''\overset{-k}{\lra}\ka'\text{ in } \cY.$$
If one prefers, one can also consider the vertices of $\widetilde{\cY}$ as shifted Young diagrams since self-conjugate partitions are in bijection with strict partitions. Further, for all $i=1,\ldots, n-1$, set again
$$f_i^\star = \df_i^\star f_i.$$
Therefore, the operators $f_i^\star$ for $i=1,\ldots, n-1$ and $a_\ka$ for $\ka=\ka^\trans$ induce an 
$(A_{n-1}^{(1)}\times B_\infty)$-crystal structure on $\widehat{F}(s)^\star$. 
More precisely, we get the decomposition
$$\widehat{F}(s)^\star \simeq \bigoplus_{\substack{\bs\in\sD(s)\\ \bs^\star=\bs}}\widehat{B}(\bs)\otimes 
\widetilde{\cY}.$$
Other multicrystal structures on fixed points sets will be studied in \Cref{fixed_points_mullineux}.

\subsection{Affine keys and bikeys}\label{bikeys_aff}

We now construct affine analogues of keys and bikeys as defined in \Cref{bikeys}.

\begin{defi}\label{affine_key}
An element $b=c_\ell\otimes\cdots\otimes c_1\in \widehat{F}(s)$ is called an \textit{affine key}
if $b$ is cylindric and $c_\ell\subseteq \cdots\subseteq c_1\subseteq c_\ell^+$.
\end{defi}

\begin{rem}
	Note that \cite{JaconLecouvey2019} uses the terminology \textit{$(n,\bs)$-cores} instead of affine bikeys.
	Indeed, it is observed that if $b=c_\ell\otimes\cdots\otimes c_1$ is an affine key, then
	each $c_j$ is in particular the beta-set of a partition which is an $n$-core.
\end{rem}

In particular, if $b\in \widehat{F}(s)$ is an affine key, $b$ is cylindric and 
therefore $s(b)\in\sD(s)$.
Now, for an element $b\in \widehat{F}(s)$, denote $\cO_{\widehat{S_n}}(b)$ the orbit of $b$ under the action of $\widehat{S_n}$.
The following result is an analogue of \Cref{key_orbit} and is a reformulation of \cite[Proposition 5.14]{JaconLecouvey2019}.

\begin{prop}\label{affine_keys_orbit} Let $\bs\in\sD(s)$.
The set of all affine keys in $\widehat{F}(\bs)$ is equal to $\cO_{\widehat{S_n}}(b_\bs)$.
\end{prop}

\begin{exa}\label{exa_aff_key} Take $\ell=n=2$ and $\bs=(r_1,r_2)=(1,0)$, so that 
$b_\bs=\scriptsize
\gyoung(0)
\otimes
\gyoung(0,1)
$.
Let us compute the beginning of the connected component of $\widehat{F}(\bs)$ containing $b_\bs$.
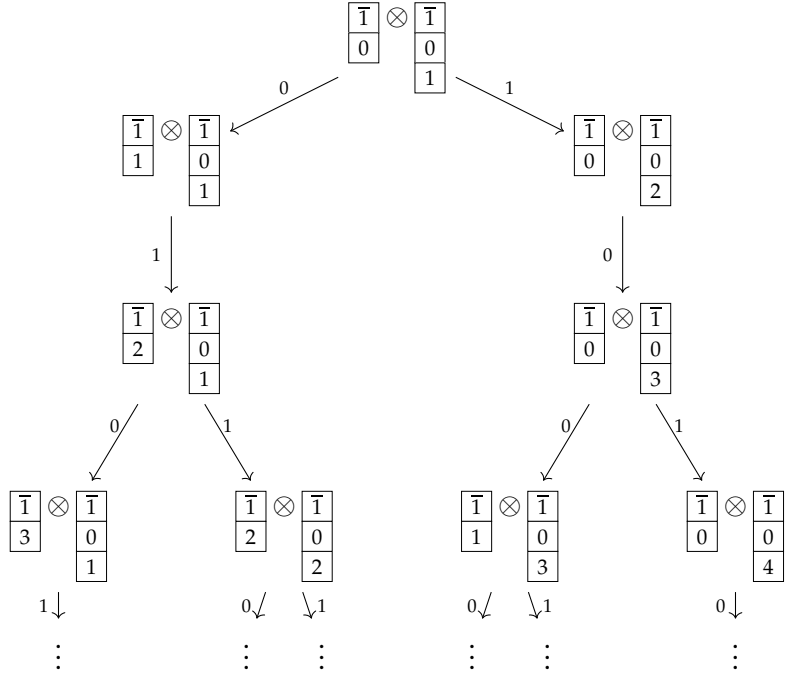
\begin{center}
\begin{tikzpicture}
\node (a) at (0,0) {$\scriptsize\young(<\overline{1}>,0)\otimes\young(<\overline{1}>,0,1)$};
\node (b1) at (-3,-1.5) {$\scriptsize\young(<\overline{1}>,1)\otimes\young(<\overline{1}>,0,1)$};
\node (b2) at (3,-1.5) {$\scriptsize\young(<\overline{1}>,0)\otimes\young(<\overline{1}>,0,2)$};
\node (c1) at (-3,-4) {$\scriptsize\young(<\overline{1}>,2)\otimes\young(<\overline{1}>,0,1)$};
\node (c2) at (3,-4) {$\scriptsize\young(<\overline{1}>,0)\otimes\young(<\overline{1}>,0,3)$};
\node (d1) at (-4.5,-6.5) {$\scriptsize\young(<\overline{1}>,3)\otimes\young(<\overline{1}>,0,1)$};
\node (d2) at (-1.5,-6.5) {$\scriptsize\young(<\overline{1}>,2)\otimes\young(<\overline{1}>,0,2)$};
\node (d3) at (1.5,-6.5) {$\scriptsize\young(<\overline{1}>,1)\otimes\young(<\overline{1}>,0,3)$};
\node (d4) at (4.5,-6.5) {$\scriptsize\young(<\overline{1}>,0)\otimes\young(<\overline{1}>,0,4)$};
\node (e1) at (-4.5,-8) {$\vspace{3mm}\vdots$};
\node (e2) at (-2,-8) {$\vspace{3mm}\vdots$};
\node (e3) at (-1,-8) {$\vspace{3mm}\vdots$};
\node (e4) at (1,-8) {$\vspace{3mm}\vdots$};
\node (e5) at (2,-8) {$\vspace{3mm}\vdots$};
\node (e6) at (4.5,-8) {$\vspace{3mm}\vdots$};

\draw[->] (a) --  node[pos=0.5,above]{\tiny 0} (b1);
\draw[->] (a) --  node[pos=0.5,above]{\tiny 1} (b2);
\draw[->] (b1) --  node[pos=0.5,left]{\tiny 1} (c1);
\draw[->] (b2) --  node[pos=0.5,left]{\tiny 0} (c2);
\draw[->] (c1) --  node[pos=0.5,above]{\tiny 0} (d1);
\draw[->] (c1) --  node[pos=0.5,above]{\tiny 1} (d2);
\draw[->] (c2) --  node[pos=0.5,above]{\tiny 0} (d3);
\draw[->] (c2) --  node[pos=0.5,above]{\tiny 1} (d4);
\draw[->] (d1) --  node[pos=0.5,left]{\tiny 1} (e1);
\draw[->] (d2) --  node[pos=0.5,left]{\tiny 0} (e2);
\draw[->] (d2) --  node[pos=0.5,right]{\tiny 1} (e3);
\draw[->] (d3) --  node[pos=0.5,left]{\tiny 0} (e4);
\draw[->] (d3) --  node[pos=0.5,right]{\tiny 1} (e5);
\draw[->] (d4) --  node[pos=0.5,left]{\tiny 0} (e6);
\end{tikzpicture}
\end{center}

Note that $b_\bs$ is cylindric, and therefore all vertices appearing are cylindric.
In the picture, only
$${\scriptsize b_\bs=\young(<\overline{1}>,0)\otimes\young(<\overline{1}>,0,1)}
\quad , \quad 
{\scriptsize\young(<\overline{1}>,1)\otimes\young(<\overline{1}>,0,1)}
\quad \mand\quad
{\scriptsize\young(<\overline{1}>,0)\otimes\young(<\overline{1}>,0,2)}$$
are in $\cO_{\widehat{S_n}}(b_\bs)$. 
One checks that these are affine keys in the sense of \Cref{affine_key}, and that the others are not,
which illustrates \Cref{affine_keys_orbit}.
\end{exa}

We can now give a characterisation of the dual of an affine key, in the spirit of \Cref{key_dual}.
It will be convenient to consider the action of $\widehat{S_n}$ on $\Z^n$ determined by the formulas
$$w(z_1,\ldots,z_n)=(z_{w(1)}, \ldots, z_{w(n)}) \text{ for } w\in S_n \mand
\si_0(z_1,\ldots,z_n)=(z_1-\ell,z_2,\ldots,z_{n-1},z_n+\ell).$$
Observe that the set $\dot{\sD(s)}$ is a fundamental domain for this action.

\begin{prop}
Let $b\in \widehat{F}(s)$. Then $b$ is an affine key if and only if 
$b^\star=\dot{b}_{\dot{\bs}}
$ 
for some $\dot{\bs}\in\Z^n(s)$.
\end{prop}

\begin{proof}
It suffices to observe that $b_\bs^\star = \dot{b}_{\dot{\bs}}$ for some $\dot{\bs}\in\sD(s)$.
The rest of the proof is analogous to that of \Cref{key_dual}. The maps $R_i, i=1,\ldots, n-1$
simply permute the columns of $\dot{b}_{\dot{\bs}}$, and by \Cref{kas_weyl_0}
$
{}^\star\si_0({\scriptsize \young(<\dr_1>)\otimes\cdots\otimes\young(<\dr_n>))=
\Yboxdimx{25pt}\young(<\dr_1-\ell>)\otimes\young(<\dr_2>)\otimes\cdots\otimes\young(<\dr_{n-1}>)\otimes\young(<\dr_n+\ell>)},
$
i.e. ${}^\star\si_0(\dot{b}_{\dot{\bs}}) = \dot{b}_{\si_0(\dot{\bs})}$.
Thus, for all $w\in\widehat{S_n}$, we have ${}^\star w( \dot{b}_{\dot{\bs}} ) = \dot{b}_{w(\dot{\bs})}.$
Now, we have
\begin{align*}
b \text{ is an affine key} 
& \eq  b\in\cO_{\widehat{S_n}}(b_\bs) \text{ for some } \bs\in\sD(s)
\\
& \eq b^\star \in\cO_{\widehat{S_n}}(\dot{b}_{\dot{\bs}}) \text{ for some } \dot{\bs}\in\dot{\sD}(s)
\\
& \eq \text{ there exists } w\in\widehat{S_n} \text{ such that } b^\star={}^\star w (\dot{b}_{\dot{\bs}})
\\
& \eq \text{ there exists } w\in\widehat{S_n} \text{ such that } b^\star=\dot{b}_{w(\dot{\bs})}
\end{align*}
\end{proof}

Let $\bs\in\sD(s)$.
By analogy with \Cref{bikeys_biorbit}, we say that $b\in \widehat{F}(s)$ is an \textit{affine bikey}
of shape $\bs$ if  $b$ in is the orbit of $b_\bs$ under the action of $\widehat{S_n}\times \widehat{S_\ell}$.
We denote $\widehat{\cK}(\bs)$ the set of bikeys of shape $\bs$.
Similarly, given $\dot{\bs}\in\dot{\sD(s)}$,
let $\dot{\widehat{\cK}}(\dot{\bs})\subseteq \dot{\widehat{F}}(\dot{\bs})$ be the set of bikeys of shape $\dot{\bs}$.
By \Cref{duality_empty}, for all $\bs\in\sD(s)$, there exists $\dot{\bs}\in\dot{\sD(s)}$ such that $b_\bs^\star=\dot{b}_{\dot{\bs}}$.
Therefore, $b\in\widehat{\cK}(\bs)$ if and only if  $b^\star\in \widehat{\cK}(\dot{\bs})$,
which gives an analogue to \Cref{bikeys_biorbit_2}.

\medskip

Finally, we can describe affine bikeys directly, in a slightly less explicit way than in finite type.
First, as in \Cref{bikeys} the maps $R_j$ for $j=1,\ldots, \ell$ act on affine keys by permuting their columns.
Therefore, the elements of $\widehat{\cK}(\bs)$ are obtained from affine keys by combining permutations of their columns
and applications of $R_0$.

\begin{exa} Take $\ell=n=2$ and $\bs=(r_1,r_2)=(1,0)$ as in \Cref{exa_aff_key}.
The orbit of $b_\bs$ under the action of $\widehat{S_\ell}$ consists of the following elements
$$b_\bs=\scriptsize
\gyoung(0)
\otimes
\gyoung(0,1)
\quad,\quad
\gyoung(0,1)
\otimes
\gyoung(0)
\quad,\quad
\gyoung(<\overline{1}>)
\otimes
\gyoung(<\overline{1}>,0,1,2)
\quad,\quad
\gyoung(<\overline{1}>,0,1,2)
\otimes
\gyoung(<\overline{1}>)
\quad,\quad
\gyoung(<\overline{2}>)
\otimes
\gyoung(<\overline{2}>,<\overline{1}>,0,1,2,3)
\quad,\quad
\gyoung(<\overline{2}>,<\overline{1}>,0,1,2,3)
\otimes
\gyoung(<\overline{2}>)
\quad,\quad
\cdots
$$
The same procedure applied to each affine key in $\widehat{F}(\bs)$ yields $\widehat{\cK}(\bs)$.
\end{exa}

\section{Bicrystal structures involving infinite rank classical root systems}

In this section, we denote by $X_{\infty}$ any infinite Dynkin diagram of type
$A_{+\infty},B_{\infty},C_{\infty}$ or $D_{\infty}$. We refer to
\cite{Lec2009} and the references therein for the notations and definition used
in this section. In particular, the nodes of the Dynkin diagrams of types
$B_{\infty},C_{\infty}$ or $D_{\infty}$ are parametrised by the integers of
$\mathbb{Z}_{\geq0}$ so that one gets a Dynkin diagram of type $A_{+\infty}$
by removing the $0$-node. Our aim in this section is to explain how the
$A_{n-1}\times A_{\ell-1}$ bicrystal structure described in \Cref{bicrystal} can be generalised to obtain $X_{\infty}\times A_{\ell-1}$
bicrystal structures.

\subsection{Recollection on extremal crystals of type $X_{\infty}$}

\label{Subsec_TabXinfinity}By \cite{Lec2009}, one can associate to each
partition $\lambda$ an extremal weight crystal\footnote{The crystal $B_{\infty
	}(\lambda)$ is not a highest weight crystal in types $C_{\infty},B_{\infty
	},D_{\infty}.$} $B_{\infty}(\lambda)$ of type $X_{\infty}$ which can
be regarded as the direct limit of the finite type $X_{n}$-crystal
$B_{n}(\lambda)$ associated to the dominant weight $\lambda$ when $n$ tends to
infinity.\ To be more precise let introduce the infinite alphabets%
\[
\mathcal{A}_{X_{\infty}}=\left\{
\begin{array}
[c]{l}%
\{1<\cdots<n<\cdots\}\text{ for }X=A,\\
\{\cdots<\overline{n}<\cdots<\overline{1}<0<1<\cdots<n<\cdots\}\text{ for
}X=B,\\
\{\cdots<\overline{n}<\cdots<\overline{1}<1<\cdots<n<\cdots\}\text{ for
}X=C,\\
\{\cdots<\overline{n}<\cdots<\overline{2}<%
\overline{1},1
<2<\cdots<n<\cdots\}\text{ for }X=D.
\end{array}
\right.
\]
The crystal $B_{\infty}(\lambda)$ admits a convenient realisation in terms of
Kashiwara-Nakashima tableaux of type $X_{\infty}$ defined exactly as their
finite rank counterpart by relaxing the admissibility condition of the
columns. The general picture is identical to the finite type $A$. First, the
crystals $B_{\infty}(1)$ for each type $A_{+\infty},B_{\infty},C_{\infty}$ and
$D_{\infty}$ are respectively
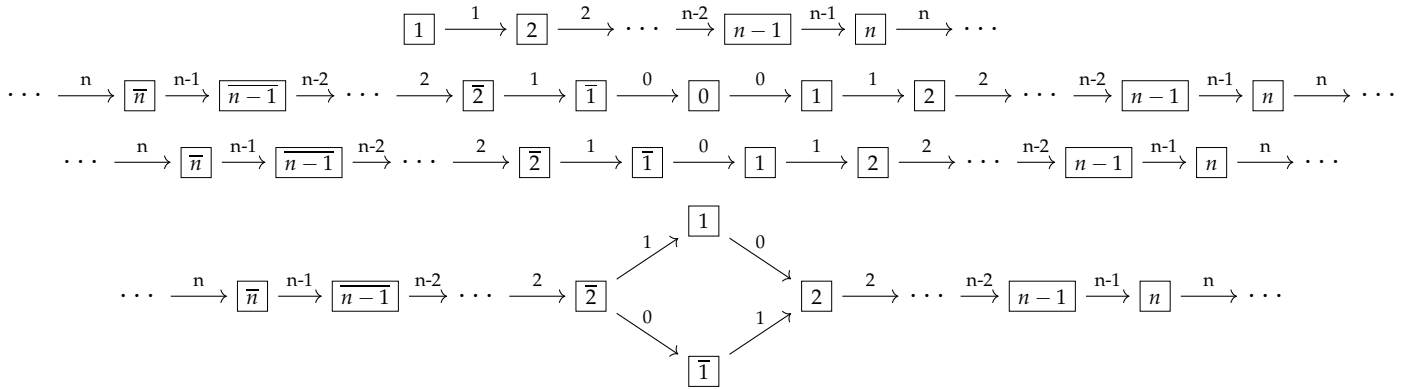
\begin{center}
\begin{tikzpicture}
\node (a) at (0,0) {$\scriptsize\young(1)$};
\node (b) at (1.5,0) {$\scriptsize\young(2)$};
\node (c) at (3,0) {$\cdots$};
\node (d) at (4.5,0) {$\scriptsize\Yboxdimx{25pt}\young(<n-1>)$};
\node (e) at (6,0) {$\scriptsize\young(n)$};
\node (f) at (7.5,0) {$\cdots$};

\draw[->] (a) --  node[pos=0.5,above]{\tiny 1} (b);
\draw[->] (b) --  node[pos=0.5,above]{\tiny 2} (c);
\draw[->] (c) --  node[pos=0.5,above]{\tiny n-2} (d);
\draw[->] (d) --  node[pos=0.5,above]{\tiny n-1} (e);
\draw[->] (e) --  node[pos=0.5,above]{\tiny n} (f);
\end{tikzpicture}
\end{center}
\begin{center}
\begin{tikzpicture}
\node (a) at (0,0) {$\cdots$};
\node (b) at (1.5,0) {$\scriptsize\young(<\overline{n}>)$};
\node (c) at (3,0) {$\Yboxdimx{25pt}\scriptsize\young(<\overline{n-1}>)$};
\node (d) at (4.5,0) {$\cdots$};
\node (e) at (6,0) {$\scriptsize\young(<\overline{2}>)$};
\node (f) at (7.5,0) {$\scriptsize\young(<\overline{1}>)$};
\node (g) at (9,0) {$\scriptsize\young(0)$};
\node (h) at (10.5,0) {$\scriptsize\young(1)$};
\node (i) at (12,0) {$\scriptsize\young(2)$};
\node (j) at (13.5,0) {$\cdots$};
\node (k) at (15,0) {$\Yboxdimx{25pt}\scriptsize\young(<n-1>)$};
\node (l) at (16.5,0) {$\scriptsize\young(n)$};
\node (m) at (18,0) {$\cdots$};

\draw[->] (a) --  node[pos=0.5,above]{\tiny n} (b);
\draw[->] (b) --  node[pos=0.5,above]{\tiny n-1} (c);
\draw[->] (c) --  node[pos=0.5,above]{\tiny n-2} (d);
\draw[->] (d) --  node[pos=0.5,above]{\tiny 2} (e);
\draw[->] (e) --  node[pos=0.5,above]{\tiny 1} (f);
\draw[->] (f) --  node[pos=0.5,above]{\tiny 0} (g);
\draw[->] (g) --  node[pos=0.5,above]{\tiny 0} (h);
\draw[->] (h) --  node[pos=0.5,above]{\tiny 1} (i);
\draw[->] (i) --  node[pos=0.5,above]{\tiny 2} (j);
\draw[->] (j) --  node[pos=0.5,above]{\tiny n-2} (k);
\draw[->] (k) --  node[pos=0.5,above]{\tiny n-1} (l);
\draw[->] (l) --  node[pos=0.5,above]{\tiny n} (m);
\end{tikzpicture}
\end{center}
\begin{center}
\begin{tikzpicture}
\node (a) at (0,0) {$\cdots$};
\node (b) at (1.5,0) {$\scriptsize\young(<\overline{n}>)$};
\node (c) at (3,0) {$\Yboxdimx{25pt}\scriptsize\young(<\overline{n-1}>)$};
\node (d) at (4.5,0) {$\cdots$};
\node (e) at (6,0) {$\scriptsize\young(<\overline{2}>)$};
\node (f) at (7.5,0) {$\scriptsize\young(<\overline{1}>)$};
\node (h) at (9,0) {$\scriptsize\young(<1>)$};
\node (i) at (10.5,0) {$\scriptsize\young(<2>)$};
\node (j) at (12,0) {$\cdots$};
\node (k) at (13.5,0) {$\Yboxdimx{25pt}\scriptsize\young(< n-1>)$};
\node (l) at (15,0) {$\scriptsize\young(<n>)$};
\node (m) at (16.5,0) {$\cdots$};

\draw[->] (a) --  node[pos=0.5,above]{\tiny n} (b);
\draw[->] (b) --  node[pos=0.5,above]{\tiny n-1} (c);
\draw[->] (c) --  node[pos=0.5,above]{\tiny n-2} (d);
\draw[->] (d) --  node[pos=0.5,above]{\tiny 2} (e);
\draw[->] (e) --  node[pos=0.5,above]{\tiny 1} (f);
\draw[->] (f) --  node[pos=0.5,above]{\tiny 0} (h);
\draw[->] (h) --  node[pos=0.5,above]{\tiny 1} (i);
\draw[->] (i) --  node[pos=0.5,above]{\tiny 2} (j);
\draw[->] (j) --  node[pos=0.5,above]{\tiny n-2} (k);
\draw[->] (k) --  node[pos=0.5,above]{\tiny n-1} (l);
\draw[->] (l) --  node[pos=0.5,above]{\tiny n} (m);
\end{tikzpicture}
\end{center}
\begin{center}
\begin{tikzpicture}
\node (a) at (0,0) {$\cdots$};
\node (b) at (1.5,0) {$\scriptsize\young(<\overline{n}>)$};
\node (c) at (3,0) {$\Yboxdimx{25pt}\scriptsize\young(<\overline{n-1}>)$};
\node (d) at (4.5,0) {$\cdots$};
\node (e) at (6,0) {$\scriptsize\young(<\overline{2}>)$};
\node (f) at (7.5,-1) {$\scriptsize\young(<\overline{1}>)$};
\node (h) at (7.5,1) {$\scriptsize\young(<1>)$};
\node (i) at (9,0) {$\scriptsize\young(<2>)$};
\node (j) at (10.5,0) {$\cdots$};
\node (k) at (12,0) {$\Yboxdimx{25pt}\scriptsize\young(<n-1>)$};
\node (l) at (13.5,0) {$\scriptsize\young(<n>)$};
\node (m) at (15,0) {$\cdots$};

\draw[->] (a) --  node[pos=0.5,above]{\tiny n} (b);
\draw[->] (b) --  node[pos=0.5,above]{\tiny n-1} (c);
\draw[->] (c) --  node[pos=0.5,above]{\tiny n-2} (d);
\draw[->] (d) --  node[pos=0.5,above]{\tiny 2} (e);
\draw[->] (e) --  node[pos=0.5,above]{\tiny 0} (f);
\draw[->] (e) --  node[pos=0.5,above]{\tiny 1} (h);
\draw[->] (f) --  node[pos=0.5,above]{\tiny 1} (i);
\draw[->] (h) --  node[pos=0.5,above]{\tiny 0} (i);
\draw[->] (i) --  node[pos=0.5,above]{\tiny 2} (j);
\draw[->] (j) --  node[pos=0.5,above]{\tiny n-2} (k);
\draw[->] (k) --  node[pos=0.5,above]{\tiny n-1} (l);
\draw[->] (l) --  node[pos=0.5,above]{\tiny n} (m);
\end{tikzpicture}
\end{center}

Next, one introduces for any $k\geq1$, column tableaux $c$ of height $k$
as the vertices of the crystal $B_{\infty}(1)^{\otimes k}$ with extremal vertex the
column $c(k)=1\otimes\cdots\otimes k$. Finally, for any partition $\lambda$
with $\ell$-columns of heights $k_{1} \geq \cdots \geq k_{\ell}$ the tableaux of shape
$\lambda$ are identified with the vertices $b=c_{\ell}\otimes\cdots\otimes
c_{1}$ in the connected component of the crystal $B_{\infty}(1^{k_{\ell}}%
)\otimes\cdots\otimes B_{\infty}(1^{k_{1}})$ with extremal vertex $c(k_{\ell}%
)\otimes\cdots\otimes c(k_{1})$.

In type $C_{\infty}$, one can show that a tableau $T=c_{1}\cdots c_{l}$ of shape $\lambda$ is a filling of the Young diagram $\lambda$ by letters of
$\mathcal{A}_{C_{\infty}}$ so that the filling of each column $c_{i}$ of
$\lambda$ is strictly increasing from top to bottom and the split form of $T$
is semistandard. The split form of $T$ is obtained by splitting each column
$c_{i}$ according to the following procedure. Given a column $c$, its split
form is the pair $(\mathrm{l}c,\mathrm{r}c)$ of columns containing no pair of
letters $(z,\overline{z})$ with $z$ unbarred defined as follows. Let $I=\{z_{1}<\cdot\cdot\cdot<z_{r}\}$ the set of unbarred letters
$z$ such that the pair $(z,\overline{z})$ occurs in $c$. Define the set
$J=\{t_{1}<\cdot\cdot\cdot<t_{r}\}$ of unbarred letters such that:

\begin{itemize}
	\item $t_{1}$ is the lowest unbarred letter satisfying: $t_{1}>z_{1}%
	,t_{1}\notin c$ and $\overline{t_{1}}\notin c,$
	
	\item for $i=2,...,r$, $t_{i}$ is the lowest unbarred letter satisfying:
	$t_{i}>\min(t_{i-1,}z_{i}),$ $t_{i}\notin c$ and $\overline{t_{i}}\notin c.$
\end{itemize}

\noindent Then write:

\noindent$\mathrm{r}c$ for the column obtained by changing in $c,$
$\overline{z}_{i}$ into $\overline{t}_{i}$ for each letter $z_{i}\in I$ and by
reordering if necessary,

\noindent$\mathrm{l}c$ for the column obtained by changing in $c,$ $z_{i}$
into $t_{i}$ for each letter $z_{i}\in I$ and by reordering if necessary.

\noindent Now $T$ is a tableau of type $C_{\infty}$ when its split form
$\mathrm{spl}(T)=\mathrm{l}c_{1}\mathrm{r}c_{1}\cdots\mathrm{l}c_{l}%
\mathrm{r}c_{l}$ is semistandard. Observe this implies that $T$ itself is
semistandard but a semistandard tableau on $\mathcal{A}_{C_{\infty}}$ is not
always of type $C_{\infty}$.\ Also, the tableaux of type $A_{+\infty}$ which
are the semistandard tableaux on the alphabet $\mathcal{A}_{A_{+\infty}}$
coincide with the tableaux of type $C_{\infty}$ with no barred letter.

\Yvcentermath1

\begin{exa}
For 
$$
T={\scriptsize
\young(<\overline{2}><\overline{1}>2,12,23)}
\text{
\quad one gets \quad
}
\mathrm{spl}(T)=
{\scriptsize
\young(<\overline{3}><\overline{2}><\overline{1}><\overline{1}>22,1122,2333)}
$$
therefore $T$ is a tableau of type $C_{\infty}$, but
$$
T'={\scriptsize
\young(<\overline{2}><\overline{2}>2,12,23)}
\text{
\quad with \quad
}
\mathrm{spl}(T')=
{\scriptsize
\young(<\overline{3}><\overline{2}><\overline{4}><\overline{2}>22,1122,2334)}
$$
is not, although it is semistandard on $\mathcal{A}_{C_{\infty}}$.
\end{exa}

The tableaux of types $B_{\infty}$ and $D_{\infty}$ can be described in a
similar way, through a splitting operation on their columns. Nevertheless, in
type $B_{\infty}$ (respectively $D_{\infty}$), blocks of the form $\scriptsize\young(0)$ 
(respectively $\scriptsize\young(<\overline{1}>,1)$) can appear in the same column. 
This slightly modifies the procedure for computing the splitting form of the tableaux. A juxtaposition $T=c_{1}\cdots
c_{l}$ of columns with decreasing heights will be a tableau of type
$B_{\infty}$ (respectively $D_{\infty}$) if its split form $\mathrm{spl}%
(T)=\mathrm{l}c_{1}\mathrm{r}c_{1}\cdots\mathrm{l}c_{l}\mathrm{r}c_{l}$ is
semistandard (respectively is semistandard and each two columns tableau $\mathrm{r}c_{i}\mathrm{l}c_{i+1}$ avoids a particular pattern $\pi_{D} $). We refer the reader to
\cite{Lec2007} for a complete description.

There exists a convenient notion of weight on the crystals
$B_{\infty}(1)^{\otimes m},m\geq0$ (and thus also on any tensor product
$B_{\infty}(1^{k_{\ell}})\otimes\cdots\otimes B_{\infty}(1^{k_{1}})$). For any $b=x_{1}%
\otimes\cdots\otimes x_{m}$ in $B(1)^{\otimes m}$, the weight $\mathrm{wt}(b)$ is
the sequence $(\gamma_{i})_{0\leq}$ where for any $0\leq i$ the
integer $\gamma_{i}$ is equal to the number of letters $\overline{i}$ in $b$
minus its number of letters $i$.

\medskip

Given $\lambda$ and $\mu$ two partitions, it was proved in \cite{Lec2009} that
the crystal $B_{\infty}(\lambda)\otimes B_{\infty}(\mu)$ decomposes as the
disjoint union of extremal weight crystals and this decomposition is
independent of the type considered.

\begin{thm}
	\label{Th_Lec_crystal}For the type $X_{\infty}$ extremal crystals $B_{\infty
	}(\lambda)$ and $B_{\infty}(\mu)$ we have
	\[
	B_{\infty}(\lambda)\otimes B_{\infty}(\mu)\simeq
	{\textstyle\bigoplus\limits_{\nu}}
	B_{\infty}(\nu)^{c_{\lambda,\mu}^{\nu}}
	\]
	where $c_{\lambda,\mu}^{\nu}$ is the usual Littlewood-Richardson coefficients
	associated to $\lambda,\mu$ and $\nu$.
\end{thm}

\begin{rem}
	\ \label{RemarkExt}
	
	\begin{enumerate}
		\item The theorem extends to the tensor products $B_{\infty}(\lambda
		^{1})\otimes B_{\infty}(\lambda^{2})\otimes\cdots\otimes B_{\infty}%
		(\lambda^{\ell})$ associated to any sequence $(\lambda^{1},\ldots,\lambda
		^{\ell})$ of partitions.\ This tensor product contains a unique component
		isomorphic to $B_{\infty}(\lambda^{1}+\cdots+\lambda^{\ell})$ that we will refer
		as the principal component (here $\lambda^{1}+\cdots+\lambda^{\ell}$ is the
		partition whose $k$-th part is the sum of the $k$-th part of each $\lambda
		^{a},a=1,\ldots,\ell$).
		
		\item For any $m\geq0$, any connected component $B$ of $B(1)^{\otimes m}$
		contains a unique \emph{distinguished extremal vertex}, namely a vertex $b_{\mathrm{ext}}%
		=x_{1}\otimes\cdots\otimes x_{m}$ such that $x_{1},\ldots,x_{m}$ are unbarred
		letters and $x_{1}\cdots x_{m}$ is a Yamanouchi word. When $b_{\mathrm{ext}}=b_\la$
		is the tableau of shape $\lambda$ whose $i$-th row contains only letters $i$,
		$B(b_\la)$ coincides with the set of tableaux of shape $\lambda$. In the
		general case, we get that $\mathrm{wt}(b)=\lambda$ is a partition and there is
		a unique crystal isomorphism from $B(b_{\mathrm{ext}})=B$ to $B(b_\la)$ sending
		$b_{\mathrm{ext}}$ on $b_\la$.
	\end{enumerate}
\end{rem}

\subsection{Jeu de Taquin on two columns and $A_{1}$-crystal structure}

For any integer $u\geq0$, write $B_{\infty}(\omega_{u})$ for the type
$X_{\infty}$-crystal extremal associated to the partition $1^{u}$. It contains
exactly the column tableaux of type $X_{\infty}$ and shape $1^{u}$.
More generally, we
identify the partition $\lambda$ with $a_{u}$ columns of height $u$ with the
formal weight $\lambda=\sum_{u}a_{u}\omega_{u}$.\ By \Cref{Th_Lec_crystal}, we get for any integers $u,v\geq1$
\begin{multline*}
B_{\infty}(\omega_{u})\otimes B_{\infty}(\omega_{v})\simeq%
{\textstyle\bigoplus\limits_{t=0}^{\min(u,v)}}
B_{\infty}(\omega_{t}+\omega_{u+v-t})=
B_{\infty}(\omega_{u+v})\oplus B_{\infty}(\omega_{1}+\omega_{u+v-1})\oplus\cdots\oplus
B_{\infty}(\omega_{\min(u,v)}+\omega_{\max(u,v)}).
\end{multline*}
When $u\leq v$, the principal component $B_{\infty}(\omega_{u}+\omega_{v})$ of
$B_{\infty}(\omega_{u})\otimes B_{\infty}(\omega_{v})$ contains exactly the
tableaux of type $X_{\infty}$ and shape $\omega_{u}+\omega_{v}$. When $u>v$,
the vertices of the principal component $B_{\infty}(\omega_{v}+\omega_{u})$ of
$B_{\infty}(\omega_{u})\otimes B_{\infty}(\omega_{v})$ are called antitableaux
of type $X_{\infty}.\ $As in \Cref{bicrystal}, they can be easily
described by using the splitting operation given in
\Cref{Subsec_TabXinfinity}.

Consider a vertex $c_{2}\otimes c_{1}$ in $B_{\infty}(\omega_{u})\otimes
B_{\infty}(\omega_{v})$ which is not a tableau. This means we have either
$u>v$, or $u\leq v$ and the connected component $B_{\infty}(c_{2}\otimes
c_{1})$ containing $c_{2}\otimes c_{1}$ is isomorphic to a crystal
$B_{\infty}(\omega_{t}+\omega_{u+v-t})$ with $t\leq u-1$.\ In both cases the
above crystal decomposition implies that the crystal $B_{\infty}(\omega
_{u-1})\otimes B_{\infty}(\omega_{v+1})$ contains a unique component
isomorphic to $B_{\infty}(c_{2}\otimes c_{1})$. Write $\dot{e}(c_{2}\otimes
c_{1})$ for the vertex in $B_{\infty}(\omega_{u-1})\otimes B_{\infty}%
(\omega_{v+1})$ matched with $c_{2}\otimes c_{1}$ by this isomorphism. When
$c_{2}\otimes c_{1}$ is a tableau, we set $\dot{e}(c_{2}\otimes c_{1})=0$.

\noindent Similarly, for any vertex $c_{2}\otimes c_{1}$ in $B_{\infty}%
(\omega_{u})\otimes B_{\infty}(\omega_{v})$ which is not an antitableau, there
exists a unique vertex $\dot{f}(c_{2}\otimes c_{1})$ in $B_{\infty}%
(\omega_{u+1})\otimes B_{\infty}(\omega_{v-1})$ such that the components
$B_{\infty}(c_{2}\otimes c_{1})$ and $B_{\infty}(\dot{f}(c_{2}\otimes c_{1}))$
are isomorphic and $\dot{f}(c_{2}\otimes c_{1})$ is matched with $c_{2}\otimes
c_{1}$ by this isomorphism. When $c_{2}\otimes c_{1}$ is an antitableau, we
set $\dot{f}(c_{2}\otimes c_{1})=0$.

For any $c_{2}\otimes c_{1}$ in $B_{\infty}(\omega_{u})\otimes B_{\infty
}(\omega_{v})$, write $\dot{B}_{\infty}(c_{2}\otimes c_{1})$ for the set obtained
by applying operators $\dot{e}$ and $\dot{f}$ to $c_{2}\otimes c_{1} $.\ We
get immediately the following proposition.

\begin{prop}
	The set $\dot{B}_{\infty}(c_{2}\otimes c_{1})$ has the structure of a
	highest weight $A_{1}$-crystal .
\end{prop}

The previous proposition can be regarded as an analogue of the Jeu de Taquin
procedure on skew tableaux of type $A$ with two columns. In fact, such a
notion of skew tableaux also exists in type $B_{\infty},C_{\infty}$ and
$D_{\infty}$. A skew tableau of type $X_{\infty}$ is defined as the filling of a skew Young diagram by letters
of $\mathcal{A}_{X_{\infty}}$ whose duplicated form is semistandard (and avoid
the pattern $\pi_{D}$ in type $D_{\infty}$).\ To any skew tableau
$c_{1}\cdots c_{\ell}$, one associates the tensor product of columns $c_{\ell
}\otimes\cdots\otimes c_{1}$.\ 
Conversely, any tensor product $c_{2}\otimes c_{1}$ of two columns can be
regarded as a minimal skew tableau as in \Cref{exa_jdt}.
The operators $\dot{e}$ and $\dot{f}$ can
then be interpreted as horizontal sliding operations on this skew tableau.
In type $A_{+\infty},$ one so recovers the usual Jeu de Taquin procedure, see \Cref{exa_jdt},
in types $B_{\infty}$ and $C_{\infty}$ this corresponds to the sliding
operations described in \cite{Lec2007}.

\Yvcentermath1

\begin{exa}
In type $C_{\infty}$, for
$$
c_1\otimes c_2=
{\scriptsize
\young(<\overline{2}>,<\overline{1}>,2,5)\otimes 
\young(<\overline{4}>,<\overline{2}>,3,5)
}
,
\text{\quad the corresponding minimal skew tableau is \quad }
{
\scriptsize\gyoung(:~<\overline{2}>,<\overline{4}><\overline{1}>,<\overline{2}>2,35,5)
}.
$$
By using the sliding procedure in type $C_{\infty}$ or the previous crystal
isomorphisms, one gets
$$
\de\left(
{
\scriptsize\gyoung(:~<\overline{2}>,<\overline{4}><\overline{1}>,<\overline{2}>2,35,5)
}
\right)
=
{
\scriptsize\gyoung(<\overline{4}><\overline{2}>,<\overline{1}><\overline{1}>,15,3,5)
}
\mand
\df\left(
{
\scriptsize\gyoung(:~<\overline{2}>,<\overline{4}><\overline{1}>,<\overline{2}>2,35,5)
}
\right)
=
{
\scriptsize\gyoung(:~<\overline{2}>,:~<\overline{1}>,<\overline{4}>2,<\overline{2}>3,55)
}.
$$
\end{exa}

\subsection{Bicrystal structure on product of columns of type $X_{\infty}$}

Consider $n,\ell\in\mathbb{Z}_{\geq2},\boldsymbol{s}=(s_{1},\ldots,s_{\ell
})\in\mathbb{Z}_{\geq0}^{\ell}$ and $s\in\mathbb{Z}_{\geq0}$. With the
notation of the previous paragraph, write $B_{\infty}(\boldsymbol{s}%
)=B_{\infty}(\omega_{s_{1}})\otimes\cdots\otimes B_{\infty}(\omega_{s_{l}})$
and set $B_{\infty}(s)=\oplus_{\left\vert \boldsymbol{s}\right\vert
	=s}B_{\infty}(\boldsymbol{s})$. By definition $B_{\infty}(s)$ is a type
$X_{\infty}$-crystal. We are going to show that it admits in fact the structure of
a type $X_{\infty}\times A_{\ell-1}$-crystal. First, for any $j=1,\ldots
,\ell-1$ define the operators $\dot{e}_{j}$ and $\dot{f}_{j} $ on each
$B_{\infty}(\boldsymbol{s})$ as the action of $\dot{e}$ and $\dot{f}$ on its
$j$ and $j+1$-th components, that is for any $c_{1}\otimes\cdots\otimes
c_{\ell}\in B_{\infty}(\boldsymbol{s})$%
\[
\dot{e}_{j}=c_{1}\otimes\cdots\otimes\dot{e}(c_{j}\otimes c_{j+1}%
)\otimes\cdots c_{\ell}\text{ and }\dot{f}_{j}=c_{1}\otimes\cdots\otimes
\dot{f}(c_{j}\otimes c_{j+1})\otimes\cdots c_{\ell}.
\]
By definition, the operators $\dot{e}$ and $\dot{f}$ commute with the action
of the $X_{\infty}$-crystal operators, therefore $\dot{e}_{j}$ and $\dot
{f}_{j}$ also commute with the operators $ e_{i}$ and $ f_{i}$
with $i\in X_{\infty}$.

\begin{prop}
	For any $b=c_{1}\otimes\cdots\otimes c_{\ell}\in B_{\infty}(\boldsymbol{s})$, the colored and oriented graph $\dot{B}(b)$ defined by applying the operators
	$\dot{e}_{j}$ and $\dot{f}_{j}$ to $b$ has the structure of a type $A_{\ell
		-1}$ crystal.
\end{prop}

\begin{proof}
	Assume first that $b$ only contains unbarred letters. Then the result follows
	from \Cref{crystals_commute}.\ In the general case, there exists a path
	in $B_{\infty}(\boldsymbol{s})$ from $b$ to the distinguished extremal
	vertex $b_{\mathrm{ext}}$ defined in \Cref{RemarkExt}. This path corresponds to a
	composition of operators $ e_{i}$ and $ f_{i},i\in X_{\infty}$
	which commute with the operators $\dot{e}_{j}$ and $\dot{f}_{j}%
	,j\in\{1,\ldots,\ell\}$.\ Therefore the graph $\dot{B}(b)$ and $\dot{B}%
	(b_{\mathrm{ext}})$ are isomorphic and we are done because $b_{\mathrm{ext}}$ only contains unbarred letters.
\end{proof}

\medskip

Denote by $\dot{f}_{\underline{j}}$ any finite composition of operators
$\dot{f}_{j}$ with $j\in\{1,\ldots,\ell\}$ and by $ k_{\underline{i}}$ any finite composition of operators $ e_{i}, f_{i}$ with
$i\geq0$.

\begin{thm}
	\ \label{Th_BiAXinfinity}
	
	\begin{enumerate}
		\item The crystal $B_{\infty}(s)$ has the structure of an $(X_{\infty}\times
		A_{\ell-1})$-crystal. This bicrystal is extremal for the $X_{\infty}
		$-structure and of highest weight for the $A_{\ell-1}$-structure.
		
		\item The tableaux $b_\la$ associated to the partitions of rank $s$ are
		the unique vertices in $B_{\infty}(s)$ both distinguished extremal for
		$X_{\infty}$ and of highest weight for $A_{\ell-1}$.
		
		\item We have the decomposition 
		$$B_{\infty}(s)=
		{\textstyle\bigoplus\limits_{\left\vert \lambda\right\vert =s}}
			\dot{f}_{\underline{j}} k_{\underline{j}}b_\la$$
		where the sum runs over all the possible $\underline{i}$ and $\underline{j}$. In particular
		\[
		B_{\infty}(s)\simeq{\textstyle\bigoplus\limits_{\left\vert \lambda\right\vert =s}}
		B_{\infty}(\lambda)\otimes B(\lambda^{\mathrm{tr}})
		\]
		
	\end{enumerate}
\end{thm}

\begin{proof}
	Assertion (1) is just a reformulation of the previous proposition.\ To prove
	Assertion (2), observe first that the vertices $b_\la$ are clearly both distinguished 	extremal for $X_{\infty}$ and of highest weight for $A_{\ell-1}$. Consider now
	$b=c_{\ell}\otimes\cdots\otimes c_{1}$ with this property.\ Since $b$ only
	contains unbarred letters, we can apply \Cref{sources_finite} and conclude that
	$b=b_\la$ for $\lambda\vdash s$. Assertion (3) easily follows from 1 et 2.
\end{proof}

\medskip

\begin{cor}
	The $A_{\ell-1}$-weight highest vertices in $B_{\infty}(s)$ are exactly the
	tableaux of type $X_{\infty}$.
\end{cor}

\begin{proof}
	Consider $b$ a $A_{\ell-1}$-weight vertices in $B_{\infty}(s)$. Then, by
	Assertion (3) of the previous theorem, there should exist a partition $\lambda$
	of rank $s$ and sequence $ k_{\underline{i}}$ such that $b= 
	{k}_{\underline{i}}b_\la$. This means that $b$ is a tableau of type
	$X_{\infty}.$ Conversely, any tableau $b$ can be written $b=
	{k}_{\underline{i}}b_\la$ and for $j=1,\ldots,\ell-1,$ we have $\dot
	{e}_{j}(b)=\dot{e}_{j}( k_{\underline{i}}b_\la)= 
	{k}_{\underline{i}}\dot{e}_{j}(b_\la)=0$ since both crystal structures commute.
\end{proof}

\medskip

We conclude this section with some important remarks.

\begin{rem}\

	\begin{enumerate}
	\item In another direction, Kwon established and studied in \cite{Kwon2009} a bicrystal structure
	arising from a duality between extremal $A_{+\infty}$-crytals and generalised Verma $A_\infty$-crystals.
		\item Contrary to \Cref{bicrystal}, we dit not introduce a duality
		$b\rightarrow b^{\ast}$ from which the $A_{\ell-1}$-crystal structure on
		$B_{\infty}(s)$ can easily be made explicit. Although the $A_{\ell-1}$-highest vertices of $B_{\infty}(s)$ correspond to Kashiwara-Nakashima tableaux of type $X_{\infty
		}$, we cannot use the same duality as in type $A_{\ell-1}$. Indeed, in types $B_{\infty}$ and $D_{\infty}$, these tableaux are not semistandard on
		$\mathcal{A}_{X_{\infty}}$ in general. Moreover, even in type $C_{\infty}$, the action of a crystal
		operator $\dot{e}_{j}$ and $\dot{f}_{j}$ does not coincide with a horizontal
		sliding of the ordinary Jeu de Taquin of type $A$. This problem, also related
		to the cyclage operation on tableaux defined in \cite{Lec2005a} and
		\cite{Lec2006b} will be considered elsewhere.
		
		\item The previous corollary can also be obtained without referring to 
		\Cref{bicrystal}. For $b=c_{1}\otimes\cdots\otimes c_{\ell}$ a $A_{\ell-1}%
		$-highest weight in $B_{\infty}(s)$, the equality $\dot{e}_{j}(b)=0$ for any
		$j=1,\ldots,\ell-1$ means that $c_{j}\otimes c_{j+1}$ is a tableau of two columns
		for any $j=1,\ldots,\ell-1$.\ This implies that $b$ is a tableau of type
		$X_{\infty}$ because the tableaux are characterised locally by conditions on
		their pairs of consecutive columns.
		
		\item The results of this section can also be regarded as bicrystal
		structures on particular matrix sets. In type $A_{+\infty}$ or $C_{\infty\text{,
		}}$each vertices $b=c_{1}\otimes\cdots\otimes c_{\ell}$ of $B_{\infty}(s)$ can
		indeed be encoded in an infinite matrix $M=(m_{i,j})$ with rows indexed by $\mathbb{Z}\setminus\{0\}$
		and columns indexed by $\{1,\ldots,\ell\}$,  where $m_{i,j}=1$ if $c_{j}$ contains the
		letter $i\in\mathbb{Z}\setminus\{0\}$ and $m_{i,j}=0$ otherwise. In type
		$B_{\infty}$ we proceed similarly, except that the integers in the rows $m_{0,j}$ may be greater than $1$. 
		Finally in type $D_{\infty}$, the coefficient $m_{0,j}$
		counts the number of blocks $\scriptsize\young(<\overline{1}>,1)$ in $c_{j}$.
	\end{enumerate}
\end{rem}

\section{Fixed points in bicrystals}\label{fixed_points_mullineux}

In this section, we show that bicrystals structures of classical types can be obtained by considering fixed points sets under involutions defined on our combinatorial Fock spaces. 
We shall detail mainly the classical case, the methods and arguments being similar in the affine situation.

\subsection{Dynkin diagram involutions in finite type A}\label{mullineux-type}

In type $A_{n-1}$, let $\theta$ be the Dynkin diagram automorphism mapping the vertex $i$ to $n-i$, for all $i=1,\ldots, n-1$.
Then $\theta$ induces a map on the set of dominant weights of type $A_{n-1}$ that we denote by the same symbol
$$\theta : P_+ \to P_+, \quad \sum_{i=1}^{n-1}a_i\om_i
\mapsto \sum_{i=1}^{n-1}a_i\om_{n-i}.$$
We can also define the map $\theta$ on the partitions of length at most $n$: $\theta$ replaces each column of height $0<i<n$ by a column of height $n-i$ and fixes the columns of height $0$ or $n$.

For any partition $\la$, recall that $B(\la)$ is the crystal of highest weight vertex $b_\la$. It can be realised by using semistandard tableaux of shape $\la$, and in this case 
$b_\la$ is the tableau of shape $\la$ with only $k$'s in row $k$ \cite[Section 7.3]{HongKang2002}.
The map $\theta$ induces a map on $B(\la)$ that we denote by the same symbol
$$ \theta : B(\la) \lra B(\theta(\la)),$$
which maps $b_\la$ to $b_{\theta(\la)}$ and 
which is a \textit{$\theta$-isomorphism of crystals}, i.e. for all $b\in B(\la)$, and for all $i=1,\ldots, n-1$,
$$\theta (f_i b)=f_{\theta(i)} (\theta(b))=f_{n-i} (\theta(b)). $$

\begin{exa}
Take $n=3$ and let $\la=\om_1=
\Yboxdim{7pt}\yng(1)\ $, 
so that $\theta(\la)=\om_2=
\Yboxdim{7pt}\yng(1,1)\ $.
We get the following crystals.
$$
\begin{array}{ccc}
B(\la) = \quad
{\scriptsize\young(1)}
\overset{1}{\lra}
{\scriptsize\young(2)}
\overset{2}{\lra}
{\scriptsize\young(3)}
&
\quad\quad\text{ and } \quad\quad
&
B(\theta(\la)) = \quad
{\scriptsize\young(1,2)}
\overset{2}{\lra}
{\scriptsize\young(1,3)}
\overset{1}{\lra}
{\scriptsize\young(2,3)}.
\end{array}
$$
\end{exa}

\subsection{Crystal structure on fixed points sets}\label{fixed_points}

Denote  
$$ P_+^\theta=\left\{ \la\in P_+ \mid \theta(\la)=\la \right\}$$
the set of dominant weights fixed under $\theta$, so that
$$
\begin{array}{rcl}
\la\in P_+^\theta 
&
\eq
&
\exists \, a_1,\ldots, a_{\lfloor n/2\rfloor} \in\Z \text{ s.t. }
\la=\left\{ 
\begin{array}{ll}\displaystyle
\sum_{i=1}^{n/2-1}  a_i(\om_i+\om_{n-i}) + a_{n/2}\om_{n/2} &\quad\text{ if $n-1$ is odd,}
\medskip
\\
\displaystyle
\sum_{i=1}^{(n-1)/2} a_i(\om_i+\om_{n-i}) &\quad\text{ if $n-1$ is even.}
\end{array}
\right.
\end{array}
$$

\begin{defi}
	For any weight $\lambda$ in $ P^\theta $, set $\la'=\sum_{i=1}^{\lfloor n/2\rfloor}  a_i \om_i^{\prime}$ where
	$\om_i'$ is the $i$-th fundamental weight of type $X_{\lfloor n/2\rfloor}$, with
	$$ X=B \text{ if } n-1 \text{ is odd}
	\quad \mand\quad
	X=C \text{ if } n-1 \text{ is even}
	$$	
\end{defi}

For any partition $\la$ with at most $n$ parts and such that $\theta
(\la)=\la$, there exists a unique crystal involution on $B(\lambda)$ such that%
\[
\theta(\tilde{f}_{i_{1}}\cdots\tilde{f}_{i_{r}}(b_{\lambda}))=\tilde{f}%
_{\theta(i_{1})}\cdots\tilde{f}_{\theta(i_{r})}(b_{\lambda})%
\]
for any sequence $i_{1},\ldots,i_{r}$ in $\{1,\ldots,n-1\}$ of arbitrary
length. This permits to consider the set%
\[
B(\la)^{\theta}=\left\{  b\in B(\la)\mid\theta(b)=b\right\}.
\]
In \cite{NaitoSagaki2001,NaitoSagaki2004} Naito and Sagaki established the following theorem. Consider a partition $\la$ with at most $n$ parts and such that $\theta(\la)=\la$.

\begin{thm}\label{NS} For $i=1,\ldots, \left\lfloor n/2 \right\rfloor$, define the modified crystal operators
\begin{enumerate}
\item\label{odd} if $n-1$ is odd, 
$$f_i^\theta=
\left\{
\begin{array}{ll}
f_i f_{n-i} & \quad \text{ if } i=1,\ldots, n/2-1
\medskip
\\
f_{i} & \quad \text{ if } i=n/2. 
\end{array}
\right.$$
\item\label{even} if $n-1$ is even,
$$f_i^\theta=
\left\{
\begin{array}{ll}
f_i f_{n-i} & \quad \text{ if } i=1,\ldots, (n-1)/2-1 
\medskip
\\
f_i f_{i+1}^2 f_i & \quad \text{ if } i=(n-1)/2. 
\end{array}
\right.$$
\end{enumerate}
Then every $b\in B(\la)^\theta$ is obtained by applying a sequence of modified crystal operators to 
the highest weight vertex $b_\la$ of $B(\la)$. This induces a crystal structure on $B(\la)^\theta$. Moreover, we have
$$B(\la)^\theta\simeq B^{(X_{\lfloor n/2\rfloor})}(\la'),$$
the crystal of classical type $X_{\lfloor n/2\rfloor}$ with highest weight $\la'$.
In each case, the action of the modified crystal operators is mapped to the action of the classical crystal operators.
\end{thm}

\begin{exa}\
\begin{enumerate}
\item Take $n=4$, and  $\la=\om_1+\om_3=\Yboxdim{7pt}\yng(2,1,1)$, so that $\la\in P_+^\theta$.
Then we have
$$B(\la)^\theta = 
\quad
{\scriptsize\young(11,2,3)}
\overset{1}{\longrightarrow}
{\scriptsize\young(12,2,4)}
\overset{2}{\lra}
{\scriptsize\young(12,3,4)}
\overset{2}{\lra}
{\scriptsize\young(13,3,4)}
\overset{1}{\lra}
{\scriptsize\young(24,3,4)}
\quad
\simeq B^{(B_2)}(\la')
$$
with $\la'=\om_1'$.
\item Take $n=5$, and  $\la=\om_1+\om_4=\Yboxdim{7pt}\yng(2,1,1,1)$, so that $\la\in P_+^\theta$.
Then we have
$$B(\la)^\theta = 
\quad
{\scriptsize\young(11,2,3,4)}
\overset{1}{\longrightarrow}
{\scriptsize\young(12,2,3,5)}
\overset{2}{\lra}
{\scriptsize\young(14,3,4,5)}
\overset{1}{\lra}
{\scriptsize\young(25,3,4,5)}
\quad
\simeq B^{(C_2)}(\la')
$$
with $\la'=\om_1'$.
\end{enumerate}
\end{exa}

Case (\ref{odd}) is a particular occurrence of a phenomenon called ``similarity of crystal bases'' by Kashiwara \cite[Theorem 5.1]{Kashiwara1996}.
In fact, in the case where $n-1$ is odd, we can also exhibit a type $C_{n/2}$ crystal structure on a subset of $B(\la)^\theta$ as follows.

Assume $n-1$ is odd and let $\la$ is a partition with at most $n$ parts and
with associated weight of the form $\la=\sum_{i=1}^{n/2}a_{i}(\om_{i}%
+\om_{n-i})$. Let $\eta:B(\la)\rightarrow B(\la)$ be the involution defined by
$\eta(b_{\la})=b_{\la}$ and
\[
\eta(\tilde{f}_{i_{1}}\cdots\tilde{f}_{i_{r}}(b_{\lambda}))=\tilde{f}%
_{n-i_{1}}^{a_{i_{1}}}\cdots\tilde{f}_{n-i_{r}}^{a_{i_{r}}}(b_{\lambda})\text{
	with }a_{i_{k}}=\left\{
\begin{array}
[c]{ll}%
1 & \text{if }i_{k}\neq n/2,\\
2 & \text{otherwise.}%
\end{array}
\right.
\]
for any sequence $i_{1},\ldots,i_{r}$ in $\{1,\ldots,n-1\}$ of arbitrary
length. Accordingly, we write 
$$P_+^\eta=\left\{ \la\in P_+ \bigg| \la=\sum_{i=1}^{n/2}a_i(\om_i+\om_{i+1})\right\}\subset P_+^\theta,$$
and for $\la\in P_+^\eta$, 
let $\la^\dagger=\sum_{i=1}^{n/2}  a_i \om_i^\dagger$
where $\om_i^\dagger$ denotes the $i$-th fundamental weight of type $C_{n/2}$.
Let $$B(\la)^\eta =\left\{ b\in B(\la) \mid \eta(b)=b  \right\}.$$

\begin{thm}\label{kashiwara} Assume $n-1$ is odd and let $\la $ be a partition with associated weight in $P_+^\eta$.
Consider the modified crystal operators
$$f_i^\eta=
f_i f_{n-i} \quad \text{ for all } i=1,\ldots, n/2.$$
Then every $b\in B(\la)^\eta$ is obtained by applying a sequence of modified crystal operators $f_i^\eta$ to 
the highest weight vertex $b_\la$ of $B(\la)$. This induces a crystal structure on $B(\la)^\eta$. Moreover, we have
$$B(\la)^\eta\simeq B^{(C_{n/2})}(\la^\dagger),$$
the crystal of type $C_{n/2}$ with highest weight $\la^\dagger$,
where the action of the modified crystal operators is mapped to the action of the classical crystal operators.
\end{thm}

\begin{exa}
Take $n=4$, and  $\la=\om_1+\om_3=\Yboxdim{7pt}\yng(2,1,1)$, so that $\la\in P_+^\eta$.
Then we have
$$B(\la)^\eta = 
\quad
{\scriptsize\young(11,2,3)}
\overset{1}{\longrightarrow}
{\scriptsize\young(12,2,4)}
\overset{2}{\lra}
{\scriptsize\young(13,3,4)}
\overset{1}{\lra}
{\scriptsize\young(24,3,4)}
\quad
\simeq B^{(C_{2})}(\la')
$$
with $\la'=\om_1'$.
\end{exa}

\begin{rem} For $\la\in P_+$, one might also consider the set
$B(\la)_\theta = \left\{ b\in B(\la) \mid \theta(\wt(b)) = \wt(b) \right\} \supset B(\la)^\theta$. 
Nevertheless, the modified crystal operators of \Cref{NS} and \Cref{kashiwara} 
do not endow $B(\la)_\theta$ with the structure of a crystal. 
For example in type $A_{3}$ for $\la=(4,3,3)$, $B(\la)_\theta$ would so have a source vertex of weight $(4,3,2,1)$ whose associated connected component is not a type $B_2$-crystal.
\end{rem}
\subsection{Bicrystal structure on fixed points sets: the finite case}\label{bicrystals_fixed_points}

In this section, we combine the results from \Cref{bicrystal} and \Cref{fixed_points}
to exhibit bicrystals structures of mixed types $A\times B$, $A\times C$ and $B\times C$.
Recall that we have fixed $s\in\Z_{\geq 0}$ and considered 
$$F(s)=\bigoplus_{\bs\in\Z^\ell(s)}F(\bs)= 
\bigoplus_{\bs\in\Z^\ell(s)} B(\om_{s_1})\otimes\cdots\otimes B(\om_{s_n}).$$
By \Cref{sources_finite} and \Cref{rsk}, we get
$$F(s)=\bigoplus_{\la\in\sS(s)}\df_{\underline{j}}^* f_{\underline{i}} b_\la,
\text{\quad which implies  \quad}
F(s) \simeq \bigoplus_{\la\in\sS(s)} B(\la)\otimes B(\la^\trans)$$ where $\sS(s)$ is the set of partitions of $s$ with Dynkin diagram contained in the rectangle $n\times\ell$.
Now, for all $b=\df_{\underline{j}}^* f_{\underline{i}} b_\la \in F(s)$ such that $\theta(\la)=\la$, we set
$$\theta(b)=\df_{\underline{j}}^* \theta(f_{\underline{i}} b_\la)$$
where $ \theta(f_{\underline{i}} b_\la) = f_{n-i_r}\cdots f_{n-i_1} b_\la$ if $\underline{i}=(i_1,\ldots, i_r)$,
and we consider
$$ F(s)^\theta = \left\{ b\in F(s)\mid b=\df_{\underline{j}}^* f_{\underline{i}} b_\la \text{ with } \theta(\la)=\la \text{ and } \theta(b)=b \right\}.$$
By Theorem \ref{NS}, we then get a bicrystal structure.

\begin{thm}
The set $F(s)^\theta$ has an $(X_{\lfloor n/2\rfloor} \times A_{\ell-1})$-crystal structure, where
$$ X= 
\left\{
\begin{array}{ll}
B & \quad \text{if $n-1$ is odd,}
\\
C & \quad \text{if $n-1$ is even.}
\end{array}
\right.
$$
This yields the decomposition
$$
F(s)^{\theta}\simeq%
\bigoplus_{\la\in\sS(s),\ \theta(\la)=\la}
B^{(X_{\left\lfloor n/2\right\rfloor})}(\lambda^{\prime})\otimes B(\lambda
^{\mathrm{tr}}).
$$
\end{thm}

In fact, we can consider fixed points on either side, and on both sides simultaneously.
To see this, we need to consider the automorphism $\dot{\theta}$ of the Dynkin diagram of type $A_{\ell-1}$.
Similarly to $\theta$, the map $\dot{\theta}$ induces 
an involution on the set of partitions with at most $\ell$ columns flipping rows of length $j$ and $\ell-j$ for $0<j<\ell$ and fixing rows of length $0$ or $\ell$. 
Then, for any partition $\la$ fixed by $\dot{\theta}$,  we can define $\dot{\theta}$ on $b=\df_{\underline{j}}^* f_{\underline{i}} b_\la \in F(s)$ by setting
$$\dot{\theta}(b)= f_{\underline{i}} (\dot{\theta}\df_{\underline{j}})^*b_\la$$
where $ \dot{\theta}(\df_{\underline{j}}^* b_\la) = 
\df_{\ell-j_r}^*\cdots \df_{\ell-j_1}^* b_\la$ if $\underline{j}=(j_1,\ldots, j_r)$,
and also consider
$$ F(s)^{\dot{\theta}} = \left\{ b\in F(s)\mid b=\df_{\underline{j}}^* f_{\underline{i}} b_\la \text{ with }\dot{\theta}(\la)=\la \text{ and } \dot{\theta}(b)=b \right\}.$$

\begin{thm}
The set $F(s)^{\dot{\theta}}$ has an $(A_{n-1} \times \dot{X}_{\lfloor \ell/2\rfloor})$-crystal structure, where
$$ \dot{X}= 
\left\{
\begin{array}{ll}
B & \quad \text{if $\ell-1$ is odd}
\\
C & \quad \text{if $\ell-1$ is even}
\end{array}
\right.
$$
This yields the decomposition
$$
F(s)^{\dot{\theta}}\simeq%
\bigoplus_{\la\in\sS(s),\ \dot{\theta}(\la)=\la}
B(\lambda)\otimes B^{(\dot{X}_{\left\lfloor \ell/2\right\rfloor })}((\lambda
^{\mathrm{tr}})^{\prime}).
$$
\end{thm}
Now set
$$
\sS(s)^{\theta,\dot{\theta}}=\{\la\in\sS(s)\mid
\theta(\lambda)=\lambda\text{ and }\dot{\theta}(\lambda)=\lambda\}.
$$
Then the set%
$$
F(s)^{\theta,\dot{\theta}}=F(s)^{\theta}\cap F(s)^{\dot{\theta}}.
$$
is the subset of $F(s)$ obtained from the double highest weight vertices
$b_{\lambda}$ with $\la\in\sS(s)^{\theta,\dot{\theta}}$ by applying
the modified crystal operators $f_{i}^{\theta}$ and $\dot{f}_{j}^{\dot{\theta
}}$.

\begin{thm}\label{Thbicrystalfinite}
	The set $F(s)^{\theta,\dot{\theta}}$ has the structure of an $(X_{\left\lfloor
		n/2\right\rfloor } \times \dot{X}_{\left\lfloor \ell/2\right\rfloor })$-crystal where
	\[
	X_{\left\lfloor n/2\right\rfloor }=\left\{
	\begin{array}
	[c]{l}%
	B\text{ if }n-1\text{ is odd,}\\
	C\text{ if }n-1\text{ is even,}%
	\end{array}
	\right.  \text{ and }\dot{X}_{\left\lfloor \ell/2\right\rfloor }=\left\{
	\begin{array}
	[c]{l}%
	B\text{ if }\ell-1\text{ is odd,}\\
	C\text{ if }\ell-1\text{ is even.}%
	\end{array}
	\right.
	\]
	Moreover, we have the decomposition
	\[
	F(s)^{\theta,\dot{\theta}}\simeq\bigoplus_{\la\in\sS(s)^{\theta,\dot{\theta}}}
	B^{(X_{\left\lfloor n/2\right\rfloor })}(\lambda^{\prime})\otimes B^{(\dot{X}_{\left\lfloor \ell/2\right\rfloor })}%
	((\lambda^{\mathrm{tr}})^{\prime}).
	\]
	
\end{thm}

\begin{rem}
	\ 
	
	\begin{enumerate}

		\item It would be interesting to have a combinatorial description of the set
		$F(s)^{\theta,\dot{\theta}}$ as tensor products of columns or binary matrices.
		
		\item By applying the relevant weight functions on the previous decompositions
		obtained for $F(s)^{\theta},F(s)^{\dot{\theta}}$ and $F(s)^{\theta,\dot{\theta
		}}$, one can get analogues of Cauchy identities in our bicrystal context . For
		example, for any vertex $b$ in $F(s)^{\theta,\dot{\theta}}$, $\mathrm{wt}%
		(b)^{\prime}$ is a weight of type $X_{\left\lfloor n/2\right\rfloor }$ whereas
		$\mathrm{wt}(b^{\ast})^{\prime}$ is a weight of type $\dot{X}_{\left\lfloor
			\ell/2\right\rfloor }$. One obtains
		\[
		\sum_{b\in F(s)^{\theta,\dot{\theta}}}x^{\mathrm{wt}(b)^{\prime}}y^{\mathrm{wt}%
			(b^{\ast})^{\prime}}=
		\sum_{\lambda\in\sS(s)^{\theta,\dot{\theta}}}s_{\lambda^{\prime}%
		}^{(X_{\left\lfloor n/2\right\rfloor })}(x)s_{(\lambda^{\mathrm{tr}})^{\prime}%
		}^{(\dot{X}_{\left\lfloor \ell/2\right\rfloor })}(y)
		\]
		where for any partition $\nu$ of length at most $m$ (respectively at most $p$), $s_{\nu}^{(X_{m})}$ (respectively $s_{\nu}^{(\dot{X}_{p})}$) stands
		for the Weyl character of type $X_{m}$ (respectively $\dot{X}_p$) associated to $\nu$.
		
		\item The results of this paragraph have analogues when the map $\theta$ is
		replaced by the map $\eta$ defined in \Cref{fixed_points} and $P^{\theta}$
		by $P^{\eta}$. 
		\item In \cite{King1975}, King established the interesting Cauchy type formula
		\[
		\prod_{i=1}^{n}\prod_{j=1}^{\ell}(x_{i}+x_{i}^{-1}+y_{j}+y_{j}^{-1}%
		)=\sum_{\lambda\subset n\times\ell}s_{\lambda}^{(C_{n})}(x)s_{[\lambda
			]}^{(C_{\ell})}(y)
		\]
		where $[\lambda]$ is the transposed of the rectangular complement of $\lambda$
		in $n\times\ell$. Although the right-hand side looks similar to the sum appearing in our results when $n-1$ and $\ell-1$ are even,
		it is not obvious to relate both.
		
	\end{enumerate}
\end{rem}

To conclude this paragraph, let us introduce a last natural involution
$\ddot{\theta}$ on the set $\sS(s)$. Here for any $\lambda\in\sS(s)$, the Young diagram of $\ddot{\theta}(\lambda)$ is obtained
from that of $\lambda$ by changing each column of height $0 \leq i\leq n$ in
a column of height $n-i$.
In other words,  $\ddot{\theta}(\la)$ is the $\pi$-rotation of the complement of $\la$ in the $n\times\ell$-rectangle as illustrated in \Cref{exa_inv_diag} below.
In particular $\theta$ and $\ddot{\theta}$ coincide
on the partitions with no column of height $n$ or $0$ but this is not true in
general. 

\begin{exa}\label{exa_inv_diag}
Take $\ell=5,n=3$ and $\Yboxdim{7pt}\la= \yng(3,2,1)$. Then $\la$ and $\ddot{\theta}(\la)$ fit in the $n\times\ell$-rectangle as follows
\begin{center}
$\vcenter{\hbox{
\begin{tikzpicture}
[scale=0.3, every node/.style={scale=1}]

\node (a) at (-4,-1) {$\la$};
\node (b) at (14,-1) {$\pi$-rotation of $\ddot{\theta}(\la)$};

\filldraw [fill=gray, fill opacity=0.5] (2,-3) rectangle (1,-2);
\filldraw [fill=gray, fill opacity=0.5] (3,-3) rectangle (2,-2);
\filldraw [fill=gray, fill opacity=0.5] (3,-2) rectangle (2,-1);
\filldraw [fill=gray, fill opacity=0.5] (4,-3) rectangle (3,-2);
\filldraw [fill=gray, fill opacity=0.5] (4,-2) rectangle (3,-1);
\filldraw [fill=gray, fill opacity=0.5] (4,-1) rectangle (3,-0);
\filldraw [fill=gray, fill opacity=0.5] (5,-3) rectangle (4,-2);
\filldraw [fill=gray, fill opacity=0.5] (5,-2) rectangle (4,-1);
\filldraw [fill=gray, fill opacity=0.5] (5,-1) rectangle (4,-0);

\draw (0,0) -- (0,-3);
\draw (1,0) -- (1,-3) ;
\draw (2,0) -- (2,-3) ;
\draw (3,0) -- (3,-3) ;
\draw (4,0) -- (4,-3) ;
\draw (5,0) -- (5,-3) ;
\draw (0,0) -- (5,0) ;
\draw (0,-1) -- (5,-1) ;
\draw (0,-2) -- (5,-2) ;
\draw (0,-3) -- (5,-3) ;

\draw[-] (a) -- (-0.5,-1.5);
\draw[-] (b) -- (5.5,-1.5);
\end{tikzpicture}
}}$
\quad so we have \quad 
$\ddot{\theta}(\la)=\Yboxdim{7pt}\yng(4,3,2)$.
\end{center}
\end{exa}

One may also observe that $\ddot{\theta}(\lambda)=\lambda$ if and
only if $\ddot{\theta}(\lambda)$ is obtained from 
$\lambda$ by
changing each row of length $0\leq j\leq \ell$ in a row of length $\ell-j$. In fact, we see that
$\ddot{\theta}(\lambda)=\lambda$ if and only if
each box of the Young diagram of $\lambda$ is paired with a box of
$n\times\ell$ outside $\lambda$.

Also, the weights of type $A_{n-1}$ and
$A_{\ell-1}$ associated to the partitions in the set%
\[
\sS^{\ddot{\theta}}(s)=\{\lambda\in\sS(s)\mid\ddot{\theta
}(\lambda)=\lambda\}
\]
are fixed simultaneously by $\theta$ and $\dot{\theta}$ (since columns of
height $n$ and rows of length $\ell$ do not contribute to $A_{n-1}$ and
$A_{\ell-1}$ weights, respectively). Thus
\[
F(s)^{\ddot{\theta}}=\{b\in F(s)^{\theta,\dot{\theta}}\mid b=f_{\underline{i}%
}\dot{f}_{\underline{j}}^{\ast}b_{\lambda}\text{ with }\ddot{\theta}%
(\lambda)=\lambda\}
\]
has a bicrystal structure exactly as in Theorem \ref{Thbicrystalfinite}. 
Also, note that when $\ell$ and $n$ are
both odd, the set $\sS^{\ddot{\theta}}(s)$ is empty. 

\begin{exa}
	Assume $n=4$ and $\ell=3$. Then the set $\sS^{\ddot{\theta}}(s)$
	contains exactly the partitions $(2,2,1,1),(3,3,0,0),(3,2,1,0)$. We have
	\[
	\sum_{b\in F(s)^{\ddot{\theta}}}x^{\mathrm{wt}(b)^{\prime}}y^{\mathrm{wt}%
		(b^{\ast})^{\prime}}=s_{(1,1)}^{(B_{2})}(x)s_{(2)}^{(\dot{C}_{1})}+s_{(3,3)}%
	^{(B_{2})}(x)s_{(0)}^{(\dot{C}_{1})}(y)+s_{(2,1)}^{(B_{2})}(x)s_{(1)}^{(\dot{C}_{1})%
	}(y).
	\]
	
\end{exa}

\subsection{Bicrystal struture on fixed points sets: the affine case}

Consider the type $A_{n-1}^{(1)}$ Dynkin diagram automorphism 
$\theta: i\mapsto -i \mod n$, for all $i=0,\ldots, n-1$. Its induces an involution on the cone of dominant weights $P_{+}$ sending each fundamental weight $\omega_{i},i=0,\ldots,n-1$ on $\omega_{-i \mod n}$. Let $P_{+}^{\theta}$
be the subset of $P_{+}$ of dominant weights fixed by $\theta$.\ When $n=2m$
is even (respectively $n=2m-1$ is odd), there is a bijection $\lambda\mapsto
\lambda^{\prime}$ between the sets $P_{+}^{\theta}$ and the set $P_{+}%
^{(D_{m+1}^{(2)})}$ (respectively $P_+^{(A_{2(m-1)}^{(2)})}$) of dominant weights for the root
system $D_{m+1}^{(2)}$ (respectively $A_{2(m-1)}^{(2)}$) (in Kac's classification of affine Dynkin diagrams \cite{Kac1990}).

For any $\lambda\in P_{+}$, similarly to the classical case, we have a
crystal anti-isomorphism also denoted by $\theta$ from $B(\lambda)$ to
$B(\theta(\lambda))$ which flips the labels $i$ and $-i \mod n$ of the
arrows.\ Assuming that $\theta(\lambda)=\lambda$, one gets an
involution on $B(\lambda)$ (also denoted by $\theta$) and it makes
sense to set $B^{\theta}(\lambda)=\{b\in B(\lambda)\mid\theta(b)=b\}$.
Let us define some modified crystal operators by%
\begin{align*}
{f}_{i}^{\theta}  & ={f}_{i}{f}_{2m-i},i=1,\ldots
,n-1,\quad{f}_{0}^{\theta}={f}_{0}\text{ and }\quad{f}%
_{m}^{\theta}={f}_{m}\text{ when }n=2m,\\
{f}_{i}^{\theta}  & ={f}_{i}{f}_{2m-1-i},i=1,\ldots
,n-1,\quad{f}_{0}^{\theta}={f}_{0}\text{ and }\quad{f}%
_{m}^{\theta}={f}_{m}{f}_{m-1}^{2}{f}_{m}\text{ when }n=2m-1.
\end{align*}
In \cite{NaitoSagaki2004}, it was proved that when $n=2m-1$ (respectively $n=2m$)
these operators stabilize $B(\lambda)^{\theta}$ and the crystal structure
obtained in this way is isomorphic to $B^{(D_{m+1}^{(2)})}(\lambda^{\prime})$ \ (respectively to
$B^{(A_{2m}^{(2)})}(\lambda^{\prime})$). Write for short $B^{(X_{})}%
(\lambda^{\prime})$ the crystal obtained in both cases.

By Lemma \ref{duality_empty}, for any $\bs\in\sD(s)$, the image of the vertex $b=b_\bs$ by the duality $\star$ is $\dot{b}_{\bs^\star}$ with $\bs^\star\in\dot{\sD}(s)$. Recall that we had denoted $\om_\bs=\om_{s_\ell}+\cdots+\om_{s_1}$ for $\bs\in\sD(s)$
and $\dot{\om}_{\dot{\bs}}=\dot{\om}_{\dot{s}_n}+\cdots+\dot{\om}_{\dot{s}_1}$ for $\dot{\bs}\in\dot{\sD(s)}$.
By mimicking the construction in \Cref{bicrystals_fixed_points}, one can define, for any combinatorial
Fock space $\widehat{F}(s)$, the crystal $\widehat{F}(s)^{\theta}$ as the crystal generated from
the triple highest weight vertices $b_{\bs}$ with $\bs\in\sD(s)^\theta$
where
$$\sD(s)^\theta=\left\{ \bs\in\sD(s) \mid \theta(\om_\bs)=\om_\bs \right\}.$$ 
Then, $\widehat{F}(s)^{\theta}$ has the structure
of an $(X_{}\times A_\infty \times A_{\ell-1}^{(1)})$-crystal.
Similarly, 
by considering $\dot{\theta}$ the Dynkin diagram automorphism of type
$A_{\ell-1}^{(1)}$, one can define the crystal $\widehat{F}(s)^{\dot{\theta}}$,
which has an $(A_{n-1}^{(1)}\times A_\infty \times\dot{X}_{})$-crystal structure. 
Finally, write%
\[
\sD(s)^{\theta,\dot{\theta}}=\{\bs\in\sD(s)\mid\theta(\omega_{\bs})=\omega_{\bs}\text{ and }%
\dot{\theta}(\dot{\omega}_{\bs^\star})=\dot{\omega}_{\bs^\star}\}
\]

\begin{exa}
	Assume $\boldsymbol{s}=(s_{1},\ldots,s_{\ell})$ belong to $\sD(s)$ and
	is such that $0\leq s_{1}\leq s_{1}\leq\cdots\leq s_{\ell}<n$ with
	$s_{j}=s_{\ell-j+1}$ for any $j=1,\ldots,\ell$.\ Then we clearly have
	$\theta(\boldsymbol{s})=\boldsymbol{s}$.\ But by Lemma \ref{duality_empty} we also get
	$\dot{\theta}(\boldsymbol{s})=\boldsymbol{s}$ since for any $j=1,\ldots
	,\ell-1$ we have $s_{j+1}-s_{j}=(n-s_{j})-(n-s_{j+1})=s_{\ell-j+1}-s_{\ell-j}%
	$.\ In particular $\boldsymbol{s}\in\sD(s)^{\theta,\dot{\theta}}$.
\end{exa}

Now define $\widehat{F}(s)^{\theta,\dot{\theta}}$ as the crystal generated by the
operators ${f}_{i}^{\theta}$ and $(\dot{f}_{i}^{\dot{\theta}})^{\ast}$
applied on the triple highest weight vertex $b_{\bs}$ with
$\bs\in\sD(s)^{\theta,\dot{\theta}}$. 
Then $\widehat{F}(s)^{\theta,\dot{\theta}}$ has the structure of an $(X_{}\times A_\infty \times
{\dot{X}}_{})$-crystal. The crystals $\widehat{F}(s)^{\theta},\widehat{F}(s)^{\dot{\theta}}$ and
$\widehat{F}(s)^{\theta,\dot{\theta}}$ can then be regarded as combinatorial Fock spaces
carrying a crystal structure other than type $A$.

\medskip

When $n=2m$ is even, one can also define the subset $P_{+}^{\eta}$ (respectively
$P_{+}^{\zeta}$) of $P_{+}^{\theta}$ of weights with an even $\omega_{0}%
$-coordinate (respectively with even $\omega_{0}$ and $\omega_{m}$%
-coordinates).\ Then, there is a bijective map $\om_\bs\mapsto\om_\bs^\dagger$
between $P_{+}^{\eta}$ and $P_{+}^{(\tilde{A}_{2m}^{(2)})}$ the set of dominant
weights for the affine root system $\tilde{A}_{2m}^{(2)}$.\ We also have a
bijective map $\om_\bs\mapsto\om_\bs^{\ddagger}$ between $P_{+}^{\zeta}$ and
$P_{+}^{(C_{m}^{(1)})}$ the set of dominant weights for the affine root system
$C_{m}^{(1)}$.\ Let us define some modified crystal operators by%
\begin{align*}
{f}_{i}^{\eta}  & ={f}_{i}{f}_{2m-i},i=1,\ldots
,n-1,\quad{f}_{0}^{\eta}={f}_{0}^{2}\text{ and }\quad{f}%
_{m}^{\eta}={f}_{m}\text{,}\\
{f}_{i}^{\zeta}  & ={f}_{i}{f}_{2m-i},i=1,\ldots
,n-1,\quad{f}_{0}^{\zeta}={f}_{0}^{2}\text{ and }\quad
{f}_{m}^{\zeta}={f}_{m}^{2}.
\end{align*}
Write $\widehat{F}(s)^\eta$ (respectively $\widehat{F}(s)^\zeta$) for the subcrystal of
$\widehat{F}(s)$ obtained by applying the operators ${f}_{i}^{\eta}$ (respectively
the operators ${f}_{i}^{\zeta}$) to vertices of the form $b_\bs$ with $\bs$ fixed by $\eta$
(respectively $\zeta$). By the results of
\cite{Kashiwara1996}, one gets that $\widehat{F}(s)^\eta$ and $\widehat{F}(s)^{\zeta
}$ are isomorphic to the crystals $F^{(\tilde{A}_{2m}^{(2)})}%
(s^\dagger)$ and $F^{(C_{m}^{(1)})}(s^{\ddagger})$ of type $\tilde{A}_{2m}^{(2)}$
and $C_{m}^{(1)}$, respectively. From any combinatorial Fock space $\widehat{F}(s)$, one
can then define the combinatorial Fock spaces $\widehat{F}(s)^\eta,F^{\dot{\eta}%
}(s),\widehat{F}(s)^\zeta$ and $\widehat{F}(s)^{\dot{\zeta}}$ as previously and also get
triple structures of crystal with one structure of affine type other than $A$.

\noindent In general, one can define sets%
\[
\sD(s)^{\sharp,\dot{\flat}}=\{\bs\in\sD(s)\mid\omega_{\bs}\in P_{+}^{\sharp}\text{ and }%
\dot{\omega}_{\dbs}\in\dot{P}_{+}^{\flat}\}
\]
where the symbols $\sharp$ and $\flat$ belong to the set
$\{\theta,\eta,\zeta\}$. 
We can then define $\widehat{F}(s)^{\sharp,\dot{\flat}}$ as the crystal generated by the operators ${f}_{i}^{\sharp}$
and $(\dot{f}_{i}^{\dot{\flat}})^{\ast}$ applied on the triple highest
weight crystal $b_{\bs}$ with $\bs\in\sD(s)^{\sharp,\dot{\flat}}$. It admits the structure of an 
$(X\times A_\infty \times {\dot{X}}_{})$-crystal described
by the table below.
\[%
\begin{tabular}
[c]{c|c|c|c|c}
& $\theta,n=2m-1$ & $\theta,n=2m$ & $\eta,n=2m$ & $\zeta,n=2m$\\\hline
$\dot{\theta},\ell=2p-1$ & $A_{2(m-1)}^{(2)}\times A_{\infty}\times\dot
{A}_{2(p-1)}^{(2)}$ & $D_{m+1}^{(2)}\times A_{\infty}\times\dot{A}%
_{2(p-1)}^{(2)}$ & $\tilde{A}_{2m}^{(2)}\times A_{\infty}\times\dot
{A}_{2(p-1)}^{(2)}$ & $C_{m}^{(1)}\times A_{\infty}\times\dot{A}%
_{2(p-1)}^{(2)}$\\\hline
$\dot{\theta},\ell=2p$ & $A_{2(m-1)}^{(2)}\times A_{\infty}\times{\dot{D}%
}_{p+1}^{(2)}$ & $D_{m+1}^{(2)}\times A_{\infty}\times{\dot{D}}_{p+1}^{(2)}$ &
$\tilde{A}_{2m}^{(2)}\times A_{\infty}\times{\dot{D}}_{p+1}^{(2)}$ &
$C_{m}^{(1)}\times A_{\infty}\times{\dot{D}}_{p+1}^{(2)}$\\\hline
$\dot{\eta},\ell=2p$ & $A_{2(m-1)}^{(2)}\times A_{\infty}\times\dot
{\tilde{A}}_{2p}^{(2)}$ & $D_{m+1}^{(2)}\times A_{\infty}\times\dot{\tilde{A}%
}_{2p}^{(2)}$ & $\tilde{A}_{2m}^{(2)}\times A_{\infty}\times\dot{\tilde{A}%
}_{2p}^{(2)}$ & $C_{m}^{(1)}\times A_{\infty}\times\dot{\tilde{A}}_{2p}^{(2)}%
$\\\hline
$\dot{\zeta},\ell=2p$ & $A_{2(m-1)}^{(2)}\times A_{\infty}\times{\dot{C}%
}_{p}^{(1)}$ & $D_{m+1}^{(2)}\times A_{\infty}\times{\dot{C}}_{p}^{(1)}$ &
$\tilde{A}_{2m}^{(2)}\times A_{\infty}\times{\dot{C}}_{p}^{(1)}$ &
$C_{m}^{(1)}\times A_{\infty}\times{\dot{C}}_{p}^{(1)}$\\
&  &  &  &
\end{tabular}
\]
In each case, $\widehat{F}(s)^{\sharp,\dot{\flat}}$ can be regarded as a
combinatorial Fock space carrying a triple crystal structure, two of them
being of affine type other than $A$. 
Observe that this gives rises to 
all possible classical affine crystal 
structures except those corresponding to Dynkin diagrams containing a sub-Dynkin diagram 
of classical type $D$.

\section{Promotion operator and a generalisation of Pitman's 2M-X transform}

In this section, we first relate the Pitman transform 2M-X to the affine
crystals $A_{1}^{(1)}$ and show how the energy can be used to prove that the
successive iterations of this transform on trajectories on $\mathbb{Z}$ with
steps $\pm1$ tend to the trivial trajectory, all of whose steps are equal to $1$.
We next define a transformation analogue in higher dimension and establish
that it also yields a natural convergence of trajectories.

\subsection{Affine type $A_{1}^{(1)}$-crystal and Pitman's 2M-X transform}

\label{SubSec_Pit1}In this paragraph, we shall consider tensor products
$B^{\otimes n}$ of the affine Kirillov-Reshetikhin crystal of type $A_{1}^{(1)}$ where
\[
B:1\overset{0}{\underset{1}{\leftrightarrows}}2.
\]
There is a straightforward bijection between the vertices $b^{\ast
}=\varepsilon_{1}\otimes\cdots\otimes\varepsilon_{n}\in B^{\otimes n}$ and the
trajectories $\pi$ of length $n$ on the set $\mathbb{Z}$ of integers starting
at $0$ with steps $+1$ or $-1$ defined by $\pi(k)=\pi_{+}(k)-\pi_{-}(k)$ for
any $k=0,\ldots,n$ where $\pi_{+}(k)$ (respectively $\pi_{-}(k)$) is the numbers of
letters $1$ (respectively of letters $2$) in $\varepsilon_{1}\otimes\cdots
\otimes\varepsilon_{k}$.\ In the sequel, we shall abuse the notation and
identify the vertices $b^{\ast}$ with their associated path $\pi$. This
corresponds to the Littelmann path model for $A_{1}$ and it is easy to
check that $b^{\ast}$ is a Yamanouchi word if and only if $\pi(k)\geq0$ for
any $k=0,\ldots,n$.

We now define the two Pitman transforms $\mathcal{P}_{\min}$ and
$\mathcal{P}_{\max}$ on the trajectories $\pi\in B^{\otimes n}$ by%
\[
\mathcal{P}_{\min}(\pi)(k)=\pi(k)-2\min_{0\leq a\leq k}\pi(a)\text{ and
}\mathcal{P}_{\max}(\pi)(k)=2\max_{0\leq a\leq k}\pi(a)-\pi(k)
\]
for any $0\leq k\leq n$. The following properties are easy to check.

\begin{enumerate}
	\item The image by $\mathcal{P}_{\min}$ or $\mathcal{P}_{\max}^{\ast}$ of any
	trajectory $\pi\in B^{\otimes n}$ is a trajectory which always remains nonnegative.
	
	\item The trajectory $\mathcal{P}_{\min}(\pi)$ corresponds to the highest
	weight vertex associated to $\pi$ in $B^{\otimes n}$ for the $A_{1}$-structure
	obtained by deleting the $0$-arrows. 
	
	\item The nonnegative trajectories are fixed by the transformation
	$\mathcal{P}_{\min}$ (since we then have $\inf_{0\leq a\leq k}\pi(k)=0$ for
	any $0\leq a\leq n$). This is not true for the transformation $\mathcal{P}%
	_{\max}$.
	
	\item We have $\mathcal{P}_{\min}(\pi)=\pi$ if and only if $\min_{0\leq a\leq
		k}\pi(a)=0$ for any $0\leq k\leq n$, that is $\pi$ remains nonnegative.
	
	\item We have $\mathcal{P}_{\max}(\pi)=\pi$ if and only if $\max_{0\leq a\leq
		k}\pi(a)=\pi(k)$ for any $0\leq k\leq n$. This means that $\pi(k)=k$ for any
	$0\leq k\leq n$, i.e. $\pi$ is the trivial trajectory $\pi_{0}$ whose all
	steps are equal to $1$.
\end{enumerate}

In the particular case $A_{1}^{(1)}$, the promotion operator $\mathrm{pr}$
acts on each trajectory $\pi$ just by flipping the steps $+1$ and $-1$. Thus,
the path $\mathrm{pr}(\pi)$ is obtain by reflecting the path $\pi$ i.e. we
have $\mathrm{pr}(\pi)(k)=-\pi(k)$ for any $0\leq k\leq n$. This implies that
\[
\mathcal{P}_{\max}=\mathcal{P}_{\min}\circ\mathrm{pr}\text{.}%
\]
By using the results of \Cref{cyclage_promotion}, we get the following proposition

\begin{prop}
	\label{Prop_Pmax1} \ 
	
	\begin{itemize}
		\item For any $\pi$ in $B^{\otimes n}$, we have $D(\mathcal{P}_{\min}%
		(\pi))=D(\pi)$: the Pitman transform $\mathcal{P}_{\min}$ preserves the energy
		$D$.
		
		\item For any nonnegative trajectory $\pi$ in $B^{\otimes n}$, we have
		$D(\mathrm{pr}(\pi))=D(\pi)-\pi_{-}$where $\pi_{-}$ is equal to the number of
		steps $-1$ in $\pi$: the promotion operator makes decrease the energy of a
		trajectory as the number of its negative steps.
		
		\item For any trajectory $\pi$ in $B^{\otimes n}$, there exists an integer
		$m_{0}$ such that for any $m\geq m_{0}$, we have $\mathcal{P}_{\max}^{(m)}%
		(\pi)=\pi_{0}$.
	\end{itemize}
\end{prop}

\begin{proof}
	The first claim follows from the fact that $\mathcal{P}_{\min}(\pi)$ is the
	highest weight vertex of the $A_{1}$-connected component containing $\pi$ and
	$D$ is constant on classical components. The second one is a consequence of
	\Cref{lem_energy_2}. Finally, the sequence of integers $D(\pi^{(m)})$ is nonnegative,
	strictly decreasing while $\pi^{(m)}$ contains at least a step $-1$. It will
	eventually becomes equal to zero for $m$ sufficiently large. Then $\pi
	^{(m)}=\pi_{0}$ because $\pi_{0}$ is the unique nonnegative trajectory such
	that $D(\pi_{0})=0$ (or the unique fixed point by the transform $\mathcal{P}%
	_{\max}$.
\end{proof}

\begin{exa}
	Starting with $\pi=112121$, we get%
	\begin{gather*}
	\mathrm{pr}(\pi)=221212,\quad\mathcal{P}_{\max}=\mathcal{P}_{\min}%
	\circ\mathrm{pr}(\pi)=111212,\quad\mathrm{pr}\mathcal{P}_{\max}(\pi)=222121,\\
	\mathcal{P}_{\max}^{2}(\pi)=111121,\quad\mathrm{pr\circ}\mathcal{P}_{\max}%
	^{2}(\pi)=222212,\quad\mathcal{P}_{\max}^{3}(\pi)=111112,\\
	\mathrm{pr}\circ\mathcal{P}_{\max}^{3}(\pi)=222221,\quad\mathcal{P}_{\max}%
	^{4}(\pi)=111111.
	\end{gather*}
	
\end{exa}

\begin{rem}
	In his seminal article \cite{Pit78}, Pitman proves that the image by the
	transforms $\mathcal{P}_{\max}$ and $\mathcal{P}_{\min}$ of a one-dimensional
	Brownian motion is a 3-dimensional Bessel process (i.e. the norm of a
	3-dimensional Brownian motion). One can replace this Brownian motion by a
	random walk with transitions $+1$ and $-1$ and related probabilities
	$p_{1},p_{-1}$ where $p_{1}+p_{-1}=1$.\ Its image by $\mathcal{P}_{\min}$ or
	$\mathcal{P}_{\max}$ yields a Markov chain on $\mathbb{Z}_{\geq0}$. Later, it
	was observed by Biane, Bougerol and O'Connell \cite{BBO} that $\mathcal{P}%
	_{\min}$ can be interpreted in Littelmann's path theory as the transform
	associating to each path its corresponding highest weight path. The previous
	one-dimensional results then admit higher dimensional generalisations (see for
	example \cite{BBO} and \cite{LLP2012}). As far as we are aware, our
	interpretation of $\mathcal{P}_{\max}$ in terms of affine crystals is new.
\end{rem}

\subsection{Generalisation to higher dimension}

The generalised Pitman transform $\mathcal{P}_{\min}$ introduced in 
\cite{BBO}
is defined for any finite root system $R$ with Dynkin diagram $I$. It
associates to each Littelmann path, its highest weight path. Here, one can
first define a Pitman transform $\mathcal{P}_{i,\min}$ for any node $i\in I$
as in \Cref{SubSec_Pit1}: it just computes the highest weight path 
for the
$A_{1}$-crystal corresponding to the node $i$.\ Then one has $\mathcal{P}%
_{\min}=\mathcal{P}_{\min,i_{1}}\cdots\mathcal{P}_{\min,i_{r}}$ where
$s_{i_{1}}\cdots s_{i_{r}}$ is a decomposition of the longest element 
$w_{0}$
of the Weyl group of $R$ as a product of elementary reflections $s_{i},i\in
I$. In particular $\mathcal{P}_{\min}$ does not depend on the reduced
decomposition considered.

\Cref{Prop_Pmax1} suggests to look for a natural higher 
dimensional
generalisation of the results of \Cref{bicrystal}. Let $B$ be the
$A_{\ell-1}^{(1)}$-crystal with set of vertices $\{1,2,\ldots,\ell\}$ such
that for any $j=0,\ldots,\ell-1$ we have
\[
\dot{f}_{j}(k)=\left\{
\begin{array}
[c]{l}%
k+1\operatorname{mod}\ell\text{ if }k=j\operatorname{mod}\ell,\\
0\text{ otherwise.}%
\end{array}
\right.
\]
This is the simplest nontrivial Kirillov-Reshetikhin crystal of type $A_{\ell-1}^{(1)}$
associated to the rectangle $1\times1$.\ For any integer $n$, the vertices of
$\dot{B}_{n,\ell}=B^{\otimes n}$ coincide with the tensor product $b^{\ast
}=d_{1}\otimes\cdots\otimes d_{n}$ of $n$ columns with height $1$ on the
alphabet $\{1,\ldots,\ell\}$. With the notation of \Cref{bicrystal},
the corresponding vectors $b\in F(n)$ are the tensor products $b=c_{\ell}%
\otimes\cdots\otimes c_{1}$ of $\ell$-columns in which each letter of
$\{1,\ldots,n\}$ appears exactly once. Set 
$b_{\mathrm{fin}}^{\ast}=1^{\otimes
     n}$. Then $b_{\mathrm{fin}}=\emptyset\otimes\emptyset\otimes\cdots
\otimes c_{1,\ldots,n}$ where $c_{1,\ldots,n}$ is the column containing 
exactly the
letters $\{1,\ldots,n\}$.

One might first define a transform $\mathcal{P}_{\max}^{\prime}$ as in
\cite{BBO} by setting $\mathcal{P}_{\max}^{\prime}=\mathcal{P}_{\max,i_{1}%
}\cdots\mathcal{P}_{\max,i_{r}}$ where $w_{0}=s_{i_{1}}\cdots s_{i_{r}}$ 
is a
reduced decomposion of $w_{0}$.\ But then, as illustrated by the following
example, $\mathcal{P}_{\max}^{\prime}$ would depend on the chosen reduced
decomposition.

\begin{exa}
     Assume $\ell=3$ and $n=5$. To apply $\mathcal{P}_{\max,1}$ (respectively
     $\mathcal{P}_{\max,2}$) to a vertex $b^{\ast}$ in $\dot{B}_{5,3}$, 
we have
     first to flip the letters $1$ and $2$ (respectively $2$ and $3$) and next 
compute the
     source vertex of the $1$-chain (respectively the $2$-chain) corresponding 
to the
     vertex so obtained. For the vertex $b^{\ast}=2\otimes1\otimes2\otimes
     3\otimes2=21232$ (we omit the symbol $\otimes$ for short), we get%
     \begin{align*}
     & 21232\overset{\mathcal{P}_{\max,1}}{\rightarrow}12131\overset
{\mathcal{P}_{\max,2}}{\rightarrow}12121\overset{\mathcal{P}_{\max,1}%
     }{\rightarrow}11212,\\
     & 21232\overset{\mathcal{P}_{\max,2}}{\rightarrow}21223\overset
{\mathcal{P}_{\max,1}}{\rightarrow}12113\overset{\mathcal{P}_{\max,2}%
     }{\rightarrow}12112.
     \end{align*}
     Thus $\mathcal{P}_{\max,1}\mathcal{P}_{\max,2}\mathcal{P}_{\max,1}%
\neq\mathcal{P}_{\max,2}\mathcal{P}_{\max,1}\mathcal{P}_{\max,2}$.
\end{exa}

In order to generalize \Cref{Prop_Pmax1}, we rather set%
\[
\mathcal{P}_{\max}:\left\{
\begin{array}
[c]{l}%
\dot{B}_{n,\ell}\rightarrow\dot{B}_{n,\ell}\\
b^{\ast}\rightarrow\mathrm{pr}\circ\mathcal{P}_{\min}(b^{\ast})
\end{array}
\right.
\]
where $\mathcal{P}_{\min}$ is the generalised Pitman transform of 
\cite{BBO},
that is $\mathcal{P}_{\min}(b^{\ast})$ is the highest weight vertex of the
$A_{\ell-1}$-connected component containing $b^{\ast}$.

\begin{thm}
     \

     \begin{enumerate}
         \item For any vertex $b^{\ast}$ in $\dot{B}_{n,\ell}$, we have 
$D(\mathcal{P}%
         _{\min}(b^{\ast}))=D(b^{\ast})$.

         \item For any vertex $b^{\ast}$ in $\dot{B}_{n,\ell}$, there 
exists an integer
         $m_{0}$ such that for any $m\geq m_{0}$, we have 
$\mathcal{P}_{\max}%
         ^{(m)}(b^{\ast})=b_{\mathrm{fin}}^{\ast}.$
     \end{enumerate}
\end{thm}

\begin{proof}
     Assertion (1) follows from the fact that the energy $D$ is constant over
     classical components and $\mathcal{P}_{\min}(b^{\ast})$ is the 
highest weight
     vertex associated to $b^{\ast}$. For Assertion (2), set $b_{m}^{\ast
     }=\mathrm{pr}^{-1}\circ\mathcal{P}_{\max}^{(m)}(b^{\ast})$ for any 
$m\geq
     0$.\ Then $b_{1}^{\ast}=\mathcal{P}_{\min}(b^{\ast})$ and thus 
$b_{1}$ is a
     tableau. More generally, by using the results of \Cref{cyclage_promotion}, we get that the
     sequence $b_{m},m\geq1$ coincides with the sequence of tableaux $\xi
     ^{m-1}(b_{1})$, that is with the cyclage sequence defined from the 
standard
     tableau $b_{1}$.\ Since all these tableaux are standard, this 
sequence is
     indeed well-defined and eventually ends on the column 
$c_{1,\ldots,n}$. This
     means that there exists an integer $m_{0}$ such that $\mathcal{P}_{\max
     }^{(m_{0})}(b^{\ast})=b_{\mathrm{fin}}^{\ast}$. Since 
$\mathcal{P}_{\max
     }^{(a)}(b_{\mathrm{fin}}^{\ast})=b_{\mathrm{fin}}^{\ast}$ for any 
$a\geq0$, we
     get $\mathcal{P}_{\max}^{(m)}(b^{\ast})=b_{\mathrm{fin}}^{\ast}$ 
for any
     $m\geq m_{0}$.
\end{proof}

\begin{rem}
     \

     \begin{enumerate}
         \item By sligthly generalizing the notion of an authorised 
cyclage operation,
         it is possible to define an analogue of $\mathcal{P}_{\max}$ on 
any tensor
         product of columns (not only for columns of height $1$) which 
yields a similar
         convergence property.

         \item In \cite{LLP2012}, random walks are defined from tensor 
products of
         crystals. In particular, one can endow the crystal 
$\dot{B}_{n,\ell}$ with a
         probability distribution compatible with the 
$A_{\ell-1}$-weight graduation
         (i.e.; two vertices with the same $A_{\ell-1}$-weight have the same
         probability). This permits to define a random walk on the 
weight lattice of
         type $A_{\ell-1}$ whose image by $\mathcal{P}_{\min}$ is a 
Markov chain in the
         Weyl chamber. We though that our transform $\mathcal{P}_{\max}$ 
also admits
         interesting probabilistic properties that we aim to study later.

         \item For the other classical affine root systems, it is also 
possible to
         define an analogue of the transformation $\mathcal{P}_{\max}$ 
by using the
         results of \cite{LOS2} which essentially reduces their study 
(and notably the
         computation of the energy) to affine type $A$ crystals by using 
relevant
         Dynkin diagram automorphisms.
     \end{enumerate}
\end{rem}

\pagebreak

\bibliographystyle{plain}

\end{document}